\documentclass[12pt,reqno,oneside]{amsart}
\usepackage{comment}
\usepackage[latin1]{inputenc}
\usepackage[T1]{fontenc}
\usepackage{amsmath,amsfonts,amscd,amssymb,amsthm}
\usepackage{mathtools}
\usepackage{dsfont}
\usepackage{color}
\usepackage[mathscr]{euscript}
\usepackage{caption}
\usepackage{longtable}

\makeatletter
\long\def\@makecaption#1#2{
	\vskip\abovecaptionskip
	\small 
	\textbf{#1.} #2\par
	\vskip\belowcaptionskip
}
\makeatother

\usepackage{graphicx}

\usepackage{placeins}
\usepackage{extsizes}
\usepackage{float}
\usepackage{multirow}
\usepackage{pdflscape}
\usepackage{stackengine}
\usepackage{caption}
\usepackage{pgfplots}
\usepackage{subcaption}
\usepackage{mathabx}
\usepackage{tikz}
\usepackage{hyperref}
\usepackage[dvipsnames]{xcolor}
\usetikzlibrary{decorations.pathreplacing}
\usetikzlibrary{backgrounds} 
\usetikzlibrary{calc}
\pgfdeclarelayer{bg}    
\pgfsetlayers{bg,main}
\usepackage{rotate}
\usepackage[noend]{algorithmic}
\usepackage{algorithm}
\usepackage{setspace}
\graphicspath{{fig/}}
\definecolor{color1bg}{HTML}{f73d28}
\definecolor{color2bg}{HTML}{FA8072}
\definecolor{bblue}{HTML}{00BFFF}
\definecolor{bblue2}{HTML}{00ffff}

\usetikzlibrary{arrows,calc}
\tikzset{
	>=stealth',
	help lines/.style={dashed, thick},
	axis/.style={<->},
	important line/.style={thick},
	connection/.style={thick, dotted},
}
\usetikzlibrary{
	arrows.meta,
	positioning,
	shapes.geometric,
	calc,
	backgrounds,
	fit
}

\usetikzlibrary{shadows}
\usetikzlibrary{backgrounds}
\tikzset{
	diagonal fill/.style 2 args={fill=#2, path picture={
			\fill[#1, sharp corners] (path picture bounding box.south west) -|
			(path picture bounding box.north east) -- cycle;}},
	reversed diagonal fill/.style 2 args={fill=#2, path picture={
			\fill[#1, sharp corners] (path picture bounding box.north west) |- 
			(path picture bounding box.south east) -- cycle;}}
}
\usetikzlibrary{arrows.meta}

\makeatletter
\@addtoreset{equation}{section}
\makeatother
\newcounter{as}[section]

\newcommand{\eps}{\epsilon}

\title[Stochastic variational PINNs]{Random test functions, $H^{-1}$ norm equivalence, and stochastic variational physics-informed neural networks}

\author{Diego Marcondes \\ $\;$ \\\today}
\address{Mathematical Sciences Institute and France-Australia Mathematical Sciences and Interactions ANU-CNRS International Research Lab, The Australian National University}
\email{\href{diego.marcondes@anu.edu.au}{diego.marcondes@anu.edu.au} }
\allowdisplaybreaks
\newtheorem{theorem}{Theorem}[section]
\newtheorem{remark}[theorem]{Remark}

\newtheorem{definition}[theorem]{Definition}

\newtheorem{corollary}[theorem]{Corollary}
\newtheorem{lemma}[theorem]{Lemma}
\newtheorem{proposition}[theorem]{Proposition}

\newcommand{\mc}[1]{{\mathcal #1}}

\newcommand{\bs}[1]{{\boldsymbol #1}}

\newcommand{\mfB}{\mathfrak{B}}

\newcommand{\mfF}{\mathfrak{F}}

\newcommand{\mcL}{\mathcal{L}}

\newcommand{\mcW}{\mathcal{W}}

\newcommand{\mbR}{\mathbb{R}}
\newcommand{\mbE}{\mathbb{E}}

\newcommand{\mbN}{\mathbb{N}}

\newcommand{\lb}{L-BFGS}
\newcommand{\lbs}{L-BFGS }

\newcommand{\bli}{\begin{linenomath}}
\newcommand{\eli}{\end{linenomath}}

\usepackage{lineno}

\definecolor{bblue}{rgb}{.2,0.2,.8}

\begin{document}
	\maketitle
	
	\begin{abstract}
		The dual norm characterisation of weak solutions of second-order linear elliptic partial differential equations is mathematically natural but computationally intractable: evaluating the $H^{-1}$ norm of the residual requires a supremum over an infinite-dimensional test space. We prove that the $H^{-1}$ norm of any functional is equivalent to its expected squared evaluation against a random test function whose probability distribution depends only on the domain. Crucially, realisations of this random test function have negative Sobolev regularity for $d \geq 2$, yet this roughness is not an obstacle: averaging over the distribution exactly recovers the correct weak topology, independently of the differential operator, and no supremum evaluation is necessary. This equivalence introduces the notion of stochastically weak solutions, which coincide with classical weak solutions, and motivates stochastic variational physics-informed neural networks (SV-PINNs): neural networks trained by minimising an empirical approximation of the stochastic norm of the PDE residual. Although instantiated here with neural networks, the underlying principle is independent of the trial space and suggests a broader paradigm for numerical methods based on stochastic rather than deterministic test spaces. The framework extends naturally to higher-order elliptic, parabolic and hyperbolic equations and to abstract operator equations on Hilbert spaces. As a proof of concept, we present numerical experiments on eight challenging second-order linear elliptic problems spanning high-frequency and multi-scale solutions, indefinite operators, variable coefficients, and non-standard domains, in which SV-PINNs consistently and significantly outperform standard PINNs, recovering solutions to within one percent relative error in hundreds of L-BFGS steps.
		
	\end{abstract}
	
	\tableofcontents
	
	\section{Introduction}
	
	\subsection{Background}
	
	Physics-informed neural networks (PINNs), proposed by \cite{raissi2019physics} to solve forward and inverse problems associated with partial differential equations (PDEs), have been widely adopted in scientific and engineering applications \cite{cuomo2022scientific,karniadakis2021physics,luo2025physics,toscano2025pinns}. However, despite their success, training stability and accuracy remain challenging in several regimes \cite{wang2021understanding,wang2021eigenvector,wang2022and}. In response to these limitations, a growing body of work has explored weak formulations of PDEs in the context of PINNs, motivated by their superior stability properties. These approaches, including variational PINNs \cite{kharazmi2019variational}, replace strong residual enforcement with weak residual evaluation against suitable test functions.
	
	As a concrete example of a weak solution, let $\Omega \subset \mbR^d$ be a bounded convex domain and consider the Poisson equation
	\begin{equation}
		\label{BVP1}
		\begin{aligned}
			-\Delta u = f & \text{ in } \Omega\\
			u = 0 & \text{ on } \partial\Omega
		\end{aligned}
	\end{equation}
	for $f \in L^{2}(\Omega)$. The weak solution of \eqref{BVP1} is the function $u$ in the Sobolev space $H^{1}_{0}(\Omega)$, which in this case also belongs to $H^{2}(\Omega)$ since $\Omega$ is convex, that satisfies
	\begin{align*}
		R_{u}(\varphi) \coloneqq \int_{\Omega} (-\Delta u(x) - f(x)) \, \varphi(x) \, dx = 0 \, \,  \text{ for all } \varphi \in H_{0}^{1}(\Omega)
	\end{align*} 
	or, equivalently,
	\begin{align*}
		\lVert R_{u} \rVert_{-1} \coloneqq \sup_{\substack{\varphi \in H_{0}^{1}(\Omega)\\ \varphi \neq 0}} \frac{|R_{u}(\varphi)|}{\lVert \varphi \rVert_{H_{0}^{1}}} = 0.
	\end{align*}
	We refer to Section \ref{SecDef} for the definition of the Sobolev spaces above and their norms.
	
	Based on the concept of weak solutions, the variational PINNs with fixed finite-dimensional test space
	\begin{equation*}
		V_{n} \coloneqq \text{span}\{v_{1},\dots,v_{n}\} \subset H_{0}^{1}(\Omega), \, \, \, \, \,  n < \infty
	\end{equation*}
	are given by minimising, over the parameters $\theta$ of a neural network $u_{\theta}: \Omega \to \mbR$, the sum of the squared weak residual evaluated on a basis of $V_{n}$ plus a soft penalty to enforce the boundary condition of \eqref{BVP1} when it is not imposed strongly by the architecture:
	\begin{align}
		\label{loss_VPINNs}
		\mcL_{V_{n}}(\theta) \coloneqq \sum_{j = 1}^{n} |R_{u_{\theta}}(v_{j})|^{2} + \lambda \, \int_{\partial\Omega} |u_{\theta}(x)|^2 \, dx
	\end{align}
	for $\lambda > 0$.
	
	In order to train the neural network, the integrals in \eqref{loss_VPINNs} are rewritten by integration by parts as
	\begin{align*}
		R_{u_{\theta}}(v_{j}) = \int_{\Omega} \nabla u_{\theta}(x) \cdot \nabla v_{j}(x) - f(x) \, v_{j}(x) \, dx
	\end{align*}
	and then approximated by a quadrature rule or Monte Carlo integration. Typical choices of $V_{n}$ include Legendre or Chebyshev polynomials, hat functions, trigonometric functions or splines. This approach can be applied to other PDEs beyond the Poisson equation, and we refer to \cite{kharazmi2019variational} for more details.
	
	In weak formulations, the choice of test functions determines how the residual probes the differential operator and, consequently, what information is used to guide the approximation. In classical numerical methods, this choice is often adapted to the operator, so that the resulting weak residual provides reliable control of the solution error. In variational PINNs, however, the test functions are typically selected independently of the operator, for example by fixing a generic polynomial or spline basis. While this leads to a simple and flexible formulation, it also means that the quantity being minimised depends on the particular choice of test functions. In general, it does not coincide with the natural weak residual norm associated with the PDE. Consequently, two different choices of test spaces may lead to qualitatively different training objectives for the same equation. Methods that attempt to reintroduce operator dependence by adapting or optimising test functions introduce additional complexity. 
	
	In order to address this issue, we consider an alternative way of evaluating weak residuals, based on a stochastic characterisation of weak solutions, which remains operator independent by considering random test functions, yet recovers the correct weak topology induced by the $\lVert \cdot \rVert_{-1}$ norm, apart from sampling approximations.
	
	\subsection{Stochastically weak solutions and SV-PINNs}
	
	In this paper, we introduce the concept of \textit{stochastically weak solutions}, which fundamentally alters how weak formulations are operationalised and leads to the stochastic variational PINNs. In contrast to weak solutions that must satisfy $R_{u}(\varphi) = 0$ for all test functions $\varphi$ in $H_{0}^{1}(\Omega)$, a stochastically weak solution should satisfy $|R_{u}(\varphi)|^{2} = 0$ in expectation over a random test function $\Phi$, that is, 
	\begin{align*}
		\lVert R_{u} \rVert_{\Phi}^{2} \coloneqq \mbE\left[|R_{u}(\Phi)|^2\right] = 0
	\end{align*}
	in which the expectation is over $\Phi$. Taking $\Phi: \Xi \times \Omega \to \mbR$ as a random measurable map defined in a probability space $(\Xi,\mfF,\mu)$, the expectation is given by
	\begin{align*}
		\mbE\left[|R_{u}(\Phi)|^2\right] \coloneqq \int_{\Xi} |R_{u}(\Phi(\xi,\cdot))|^{2} \, d\mu(\xi).
	\end{align*}
	We call $\lVert R_{u} \rVert_{\Phi}$ the $\Phi$-stochastically weak norm of $R_{u}$.
	
	This definition naturally leads to a PINN-based approach to approximate stochastically weak solutions of \eqref{BVP1} by solving the optimisation problem
	\begin{align}
		\label{p1}
		\min_{\theta} \, \lVert R_{u_{\theta}} \, \rVert_{\Phi}^{2} + \lambda \, \int_{\partial\Omega} |u_{\theta}(x)|^2 \, dx
	\end{align}
	for $\lambda > 0$. In order to obtain a computable loss function, one can consider an empirical version of this loss
	\begin{align*}
		\mcL_{\Phi}(\theta) = \frac{1}{N} \sum_{j = 1}^{N} |R_{u_{\theta}}(\varphi_{j})|^{2} + \lambda \, \int_{\partial\Omega} |u_{\theta}(x)|^2 \, dx
	\end{align*}
	for random functions $\varphi_{1},\dots,\varphi_{N}$ sampled independently from the same distribution as $\Phi$, which can be minimised in practice using quadrature or Monte Carlo integration to approximate the integrals appearing in the loss. We call this the \textit{stochastic variational physics-informed neural networks} (SV-PINNs) framework.
	
	The well-posedness and viability of this method are not guaranteed a priori, unless two key mathematical questions are answered. The first one is the existence and uniqueness of stochastically weak solutions, ensuring that \eqref{p1} is well-posed when the neural network has sufficient approximation capacity \cite{devore2021neural}. The second question is how the stochastically weak solution is related to the weak solution, which is ultimately the function one seeks to approximate.
	
	We answer these questions by proving that, for a carefully constructed random test function $\Phi$, the $\Phi$-stochastically weak norm of the residual is \textit{equivalent} to the dual $H^{-1}(\Omega)$ norm. Formally, we prove (cf. Proposition \ref{prop_equiv}) that
	\begin{align}
		\label{equiv0}
		c \, \lVert R_{u} \rVert_{-1} \leq \lVert R_{u} \rVert_{\Phi} \leq C \, \lVert R_{u} \rVert_{-1}
	\end{align}
	for all $u \in H_{0}^{1}(\Omega)$ and constants $c,C > 0$ that depend on the domain $\Omega$ and the scaling of $\Phi$. In particular, neither $\Phi$ nor the constants depend on the PDE operator being considered in the weak residual, but only on the domain $\Omega$. The equivalence above holds for any functional $R \in H^{-1}(\Omega)$, and not only for the residual functional, so the norms are in fact equivalent.
	
	It is immediate from \eqref{equiv0} that the stochastically weak solution is unique and coincides with the weak solution. As a consequence, the infeasible minimax optimisation problem defined by the dual norm
	\begin{align*}
		\min_{\theta} \, \lVert R_{u_{\theta}} \, \rVert_{-1}^{2} + \lambda \, \int_{\partial\Omega} |u_{\theta}(x)|^2 \, dx,
	\end{align*}
	can be replaced by the stochastic optimisation problem \eqref{p1} that admits a practical empirical approximation, converting an abstract dual variational characterisation into a feasible stochastic optimisation problem.
	
	While in this paper we focus on second-order linear elliptic operators for the sake of simplicity and proof-of-concept, the SV-PINNs provide a mathematically grounded framework for solving a broad class of PDEs. In this framework, a stochastically weak norm of the residuals, equivalent to the respective weak norm, is minimised over the parameters of a neural network. Our experiments in challenging problems, involving high-frequency solutions, indefinite operators and different domains, illustrate the capabilities of this method, which consistently outperforms standard PINNs in challenging regimes, with SV-PINNs recovering solutions to within one percent relative error in hundreds of \lbs steps.
	
	\subsection{Related work}
	
	PINNs approaches that rely on the strong form of PDEs by minimising the $L^2(\Omega)$ norm of the residuals have known failure modes \cite{wang2021understanding,wang2021eigenvector,wang2022and} related to, for example, spectral bias and gradient pathologies. These have been addressed in the strong formulation setting in extensions of the basic algorithm proposed by \cite{raissi2019physics} and we refer to the reviews \cite{cuomo2022scientific,karniadakis2021physics,luo2025physics,toscano2025pinns} for more details. 
	
	Another line of research to mitigate PINNs failures considers neural-network-based methods for solving PDEs relying on weak formulations. These in fact predate PINNs \cite{raissi2019physics},  a prominent example being the Deep Ritz Method (DRM) \cite{yu2018deep}. This framework proposes training a neural network by minimising the energy functional associated with the variational formulation of elliptic PDEs plus a soft penalty to enforce boundary conditions. The DRM usually requires computing fewer derivatives than PINNs based on the strong form of PDEs, so it can be more efficient in certain cases. We refer to \cite{hu2025iterative,yu2025natural} for recent DRM-based frameworks.
	
	The closest PINN-based method to that proposed in this paper is the variational PINNs (V-PINNs) proposed by \cite{kharazmi2019variational} and refined by \cite{kharazmi2021hp} with domain decomposition. The V-PINNs consist of minimising a loss like \eqref{loss_VPINNs} for a fixed finite-dimensional subspace of the true test space, which is $H_{0}^1(\Omega)$ in the example above. Many choices of test spaces $V_{n}$ have been studied in the literature, for instance an adversarial approach was proposed by \cite{zang2020weak}, in which the test functions are trainable neural networks. The error of V-PINNs was studied in \cite{berrone2022variational}. Unlike V-PINNs, which approximate the test space by a deterministic finite-dimensional subspace, or a trained adversary, SV-PINNs use random test functions drawn from a fixed distribution that induces a norm equivalent to the $H^{-1}(\Omega)$ dual norm.
	
	We note that V-PINNs are a least-squares Petrov-Galerkin method \cite[Section~II.4]{braess2001finite} where the trial space is generated by the neural network architecture and the goal is to find parameters that satisfy the boundary condition and
	\begin{align*}
		R_{u_{\theta}}(v) = 0 \text{ for all } v \in V_{n}.
	\end{align*}
	We refer to \cite{braess2001finite} for more details and to \cite{ainsworth2021galerkin,gao2022physics,xu2026weak} for other hybrid neural-network-based (Petrov-)Galerkin methods that also consider finite-dimensional test spaces in some way. These methods, and Petrov-Galerkin methods in general, are conceptually different from the proposed SV-PINNs which do not consider a test space, but only a trial space, given by a neural network, and the optimisation problem is given by minimising the $\Phi$-stochastically weak norm over it.
	
	Nevertheless, in practice, the SV-PINNs can be interpreted as a Petrov-Galerkin method in which the test space is \textit{random} and given by
	\begin{align*}
		\hat{V}_{N} = \text{span}\{\varphi_{1},\dots,\varphi_{N}\}
	\end{align*}
	in which the test functions $\varphi_{j}$ are sampled independently according to the distribution of $\Phi$. Therefore, pragmatically, the SV-PINNs are a Petrov-Galerkin method based on least squares with a neural network trial space and $\hat{V}_{N}$ as test space. 
	
	However, there are at least two important caveats. First, the test functions are not tailored to the specific partial differential operator, but depend only on the domain $\Omega$, as $\Phi$ does not depend on the operator. Therefore, in principle, the method does not change when the operator changes as long as the weak formulation is based on the $H^{-1}(\Omega)$ norm. This differs from, for example, Discontinuous Petrov-Galerkin \cite{demkowicz2017discontinuous}, where optimal test functions for a given operator are chosen. Second, instead of depending directly on the \textit{size} of the test set, the performance of SV-PINNs relies on how well the $\Phi$-stochastically weak norm can be approximated by the empirical version uniformly in $\theta$. This implies that the error analysis of SV-PINNs is a statistical problem that depends on concentration inequalities, rather than a classical numerical analysis problem. As will be discussed in Section \ref{Sec_error}, these concentration inequalities depend on the \textit{geometry} of the trial space.
	
	We note that there are many numerical methods for solving PDEs that involve \textit{randomness}, in particular probabilistic numerical methods \cite{cockayne2019bayesian} and physics-informed Gaussian processes \cite{cross2024spectrum}, and numerical methods to approximate \textit{random solutions} of stochastic PDEs \cite{lord2014introduction}. There are also the so-called stochastic Galerkin methods \cite{babuvska2007stochastic,jakeman2008stochastic}, in which the randomness is in the problem variables, such as inputs, boundary conditions and equation parameters. The randomness of these kinds of methods differs from that of this paper, since the optimisation problem is not based on statistical models, such as a Bayesian model, and both the problem and its solution are \textit{deterministic}, but the solution is rather determined via a stochastically weak norm. 
	
	Randomness also appears in numerical linear algebra approaches to solving discretised PDEs, where random sketching matrices or randomised Kaczmarz iterations are used to accelerate the solution of large linear systems \cite{mahoney2011randomized, strohmer2009randomized, woodruff2014sketching}. In these methods, randomness is a computational tool applied after discretisation and is not designed to recover any particular norm of the continuous residual. Moreover, random feature methods \cite{chen2021solving,rahimi2007random} use randomly sampled basis functions to span the trial space for PDE approximation, with the distribution chosen to approximate a kernel rather than to induce a dual norm equivalence. In both cases, the role of randomness differs from that in SV-PINNs, where the distribution of the random test functions is chosen specifically so that the resulting expected squared residual is equivalent to the dual norm of the residual.
	
	Norm equivalences like \eqref{equiv0} have been studied in the context of Malliavin calculus \cite{nualart2006malliavin}, which is particularly concerned with computing derivatives of functionals applied to stochastic processes. In that theory, the norm of a Hilbert space is showed to be equivalent to the norm induced by an associated \textit{isonormal Gaussian} processes on the space. Nevertheless, the main interest of that theory is on spaces that are subsets of $L^2(\Omega)$, while our interest is on a stochastic norm equivalent to that of $H^{-1}(\Omega)$. Furthermore, the process $\Phi$ we consider is not a Gaussian process unless $d = 1$.
	
	\subsection{Main contributions}
	
	The main contributions of this paper are:
	\begin{itemize}
		\item We characterise the $\Phi$-stochastically weak solutions of second-order linear elliptic PDEs and prove that they coincide with the weak solutions when $\Phi$ is defined as
		\begin{align}
			\label{Phi0}
			\Phi(\xi,\cdot) = \tau \sum_{k = 1}^{\infty} (1 + \lambda_{k})^{-1/2} \, w_{k}(\xi) \, \phi_{k}(\cdot), \, \, \tau > 0
		\end{align}
		in which $w_{k}$ are independent random variables with mean zero and variance one (e.g., standard Gaussian), and $\lambda_{k}$ and $\phi_{k}$ are the eigenvalues and eigenfunctions of the Dirichlet Laplacian on $\Omega$. In particular, we prove the norm equivalence \eqref{equiv0}, which generalises, to non-Gaussian weights and to the setting of dual Sobolev spaces of PDE residuals, the isometry property of isonormal Gaussian processes studied in Malliavin calculus \cite{nualart2006malliavin}. 
		\item We propose the SV-PINNs as neural networks that are trained by minimising the empirical $\Phi$-stochastically weak norm of the residuals with a soft penalty for boundary conditions when they are not obtained by hard constraints in the architecture. This formulation seeks to mitigate failure modes of PINNs and eliminate the issue of selecting a suitable test space for each PDE operator.
		\item As a proof-of-concept, we perform experiments with challenging problems on hypercube, circular, and L-shaped domains, comparing SV-PINNs with PINNs to illustrate the capabilities of the proposed method. 
		\item Although we focus on second-order linear elliptic PDEs for simplicity, we argue that the method is actually more general and discuss how it can be applied to higher-order elliptic PDEs, and parabolic and hyperbolic PDEs. 
	\end{itemize}
	
	While weak formulations and Galerkin methods are classical, they rely on deterministic test spaces and supremum-based dual norms. To the best of our knowledge, the idea of inducing an equivalent weak norm through expectations over random test functions has neither been made explicit in the numerical analysis literature nor exploited algorithmically.
	
	\subsection{Comments}
	
	A central and perhaps counterintuitive aspect of the proposed framework is the nature of the random test functions that define the $\Phi$-stochastically weak norm. In classical weak formulations and Petrov-Galerkin methods, test functions are typically chosen to be smooth, structured, and closely related to the differential operator or the geometry of the domain. In contrast, the random test functions induced by $\Phi$ in \eqref{Phi0}, although they depend strongly on the geometry of the domain, have inherently low regularity: for spatial dimensions $d \geq 2$, realisations of $\Phi$ are not elements of $L^{2}(\Omega)$, but rather distributions belonging almost surely to negative-order Sobolev spaces (cf. Lemma \ref{lemma_traj}). From the perspective of numerical analysis, such test functions would normally be regarded as unsuitable for defining stable weak formulations.
	
	The key insight of this work is that, despite their low regularity, these random test functions induce through averaging a norm that is equivalent to the dual $H^{-1}(\Omega)$ norm of the residual functional. The equivalence \eqref{equiv0} shows that no smoothness, adaptivity, or operator-specific construction of test functions is required: a fixed probability distribution over rough, operator-independent random test functions that depend only on the domain geometry suffices to recover the correct weak topology. This stands in contrast to standard deterministic approaches, where stability and accuracy are tightly linked to the careful design of the test space.
	
	From this perspective, SV-PINNs should not be viewed merely as another weak-form PINN variant, but rather as a conceptually different way of operationalising weak formulations. The stochastically weak norm replaces a supremum over an infinite-dimensional test space by an expectation over a fixed random process, thereby transforming an intractable minimax dual-norm minimisation problem into a stochastic one. While in this paper this idea is instantiated using neural networks as trial spaces, the underlying principle is independent of the approximation architecture and suggests a broader paradigm for constructing numerical methods based on stochastic rather than deterministic test spaces.	
	
	On the one hand, a possible limitation is that sampling from \eqref{Phi0} usually requires eigenvalues and eigenfunctions of the Laplace operator, which do not always admit closed-form expressions. On the other hand, this spectral decomposition needs to be computed only once for a given domain. Moreover, we can leverage the same eigenfunctions to impose hard boundary constraints through domain-aware Fourier features \cite{calero2026enhancing}. In this way, SV-PINNs yield a method for solving second-order linear elliptic PDEs that relies solely on the spectral decomposition of the Laplacian and remains independent of the underlying PDE operator, as will be illustrated in the experiments.
	
	\subsection{Paper structure}
	
	In Section \ref{SecDef}, we define the Sobolev spaces and abstract framework considered in this paper, and in Section \ref{SecWeak}, we define the $\Phi$-stochastically weak solutions and prove the equivalence of norms in \eqref{equiv0} for the particular $\Phi$ defined in \eqref{Phi0}. In Section \ref{SecSV}, we define the SV-PINNs and propose algorithms to sample random test functions. In Section \ref{Sec_error}, we bound the approximation of the $\Phi$-stochastically weak norm by its empirical version and discuss what is necessary to control the error of SV-PINNs. In Section \ref{Sec_exp}, we present six challenging experiments in hypercube domains, while in Sections \ref{Sec_exp_circle} and \ref{Sec_exp_L} we present experiments in circular and L-shaped domains, respectively. Finally, in Section \ref{SecGen} we discuss generalisations and extensions of the SV-PINNs, and in Section \ref{SecDis} we present a discussion. More details about the implementation of the experiments and auxiliary results can be found in the appendix.
	
	\section{Fractional Sobolev spaces in domains}
	\label{SecDef}
	
	Let $\Omega \subset \mbR^{d}, d \geq 1,$ be a bounded, open and connected subset with Lipschitz boundary. We denote the inner product in $L^{2}(\Omega)$ by $\langle \cdot,\cdot \rangle$ and the norm by $\lVert \cdot \rVert_{L^2}$. The norm in $L^{2}(\partial\Omega)$ is denoted by $\lVert \cdot \rVert_{L^2(\partial\Omega)}$. Throughout this paper, $c,C > 0$ denote generic constants whose value may change from line to line, or even within the same line, and their dependence on fixed parameters may be stated explicitly when needed.
	
	Consider the spectral decomposition of the Dirichlet Laplacian
	\begin{equation}
		\label{DLoper}
		\begin{aligned}
			-\Delta \phi_{k} &= \lambda_{k} \phi_{k} & & \text{ in } \Omega \\
			\phi_{k} &= 0 & & \text{ on } \partial \Omega
		\end{aligned}
	\end{equation}
	in which $\{\phi_{k}\}$ can be chosen to form an orthonormal basis of $L^{2}(\Omega)$ and $\{\lambda_{k}\}$ is a non-decreasing sequence satisfying Weyl asymptotic bounds \cite{ivrii2016100}
	\begin{equation}
		\label{wlaw}
		c \ \! k^{2/d} \leq \lambda_{k} \leq C \ \! k^{2/d}
	\end{equation}
	for constants $c,C > 0$ that depend on the domain $\Omega$.
	
	For each $s \in \mbR_{+}$ we define the space
	\begin{equation}
		\label{Sob_space}
		\dot{H}^{s}(\Omega) = \left\{u \in L^{2}(\Omega): \sum_{k = 1}^{\infty} \lambda_{k}^{s} \ |\langle u,\phi_{k} \rangle|^{2} < \infty\right\}
	\end{equation}
	and its dual
	\begin{equation*}
		\dot{H}^{-s}(\Omega) = \left\{R: \dot{H}^{s}(\Omega) \to \mbR \text{ bounded linear}: \sum_{k = 1}^{\infty} \lambda_{k}^{-s} |R(\phi_{k})|^{2} < \infty\right\}
	\end{equation*}
	equipped with the norms
	\begin{align}
		\label{norms}
		\lVert u \rVert_{\dot{H}^{s}}^{2} \coloneqq \sum_{k = 1}^{\infty} \lambda_{k}^{s} \ |\langle u,\phi_{k} \rangle|^{2} & & \text{ and } & & \lVert R \rVert_{-s}^{2} \coloneqq \sum_{k = 1}^{\infty} \lambda_{k}^{-s} |R(\phi_{k})|^{2},
	\end{align}
	respectively. We note that $\lVert R \rVert_{-s}$ can be equivalently defined as
	\begin{align*}
		\lVert R \rVert_{-s} = \sup_{\substack{\varphi \in \dot{H}^{s}(\Omega) \\ \varphi \neq 0}} \frac{|R(\varphi)|}{\lVert \varphi \rVert_{\dot{H}^{s}}}.
	\end{align*}
	For $0 \leq s \leq 1$, the space $\dot{H}^{s}(\Omega)$ is equivalent to the Sobolev space $H_{0}^{s}(\Omega)$ defined as the closure of $C_{c}^{\infty}(\Omega)$, the space of infinitely differentiable functions $\phi: \Omega \to \mbR$ with compact support, with respect to equivalent norms on fractional Sobolev spaces (e.g., based on the Gagliardo semi-norm (cf. \eqref{Gnorm}) or the Bessel potentials; see \cite{di2012hitchhikers} for more details). We refer to \cite{kim2020fractional}, \cite[Section~2.2]{nochetto2015pde}, \cite[Chapter~3]{thomee2007galerkin} and the references therein for more details about this equivalent definition of fractional Sobolev spaces\footnote{We note that the space $H_{0}^{s}(\Omega) = \dot{H}^{s}(\Omega)$ defined in \eqref{Sob_space} for $0 \leq s \leq 1$ is equivalent to $H^{s}(\Omega)$  for $0 < s < 1/2$, equivalent to the Lions-Magenes space $H_{00}^{1/2}$ for $s = 1/2$ and equivalent to $L^{2}(\Omega)$ for $s = 0$, although we always use the same notation $H_{0}^{s}(\Omega)$ for simplification. We refer to \cite{nochetto2015pde} for more details.} $H_{0}^{s}(\Omega)$. From now on we denote $H_{0}^{s}(\Omega)$ instead of $\dot{H}^s(\Omega)$, and its dual $H^{-s}(\Omega)$ instead of $\dot{H}^{-s}(\Omega)$, when $s \leq 1$.
	
	A Lebesgue measurable function $u: \Omega \to \mbR$ is said to have a weak derivative $D^{\alpha}u$ with multi-index $\alpha = \alpha_{1}\dots\alpha_{d}, \alpha_{j} \in \mbN,$ if there exists a measurable function $v$ such that
	\begin{equation*}
		\int_{\Omega} u \ \! \partial^{\alpha} \phi \ dx = (-1)^{|\alpha|} \int_{\Omega} v \ \! \phi \ dx
	\end{equation*}
	for all $\phi \in C_{c}^{\infty}(\Omega)$ in which $|\alpha| = \sum_{i} \alpha_{i}$. We call $D^{\alpha}u \coloneqq v$ the weak derivative of $u$ corresponding to $\alpha$. For $s \in \mbR_{+}$, we define
	the Sobolev space
	\begin{align*}
		H^{s}(\Omega) = \left\{u \in L^{2}(\Omega): \lVert u \rVert_{H^{s}} < \infty\right\}
	\end{align*}
	with the norm\footnote{With the convention that the second term is omitted when $s$ is an integer.}
	\begin{align}
		\label{Gnorm}
		\|u\|_{H^{s}}^{2} =	\sum_{|\alpha| \leq \lfloor s \rfloor} \|D^{\alpha} u\|_{L^{2}(\Omega)}^{2} + \sum_{|\alpha| = \lfloor s \rfloor} \int_{\Omega} \! \int_{\Omega} \frac{|D^{\alpha}u(x) - D^{\alpha}u(y)|^{2}}{|x-y|^{\,d + 2(s - \lfloor s \rfloor)}} \, dx \, dy
	\end{align}
	based on the Gagliardo semi-norm \cite{di2012hitchhikers}. For $s^{\prime} > 0$, there exists a continuous embedding $H^{s + s^{\prime}}(\Omega) \hookrightarrow H^{s}(\Omega)$. Since $\Omega$ is bounded and has a Lipschitz boundary, the space $H_{0}^{1}(\Omega)$ can be defined equivalently as the elements of $H^1(\Omega)$ which are equal to zero on $\partial\Omega$ in the trace sense:
	\begin{align*}
		H_{0}^{1}(\Omega) = \{u \in H^1(\Omega): u|_{\partial\Omega} = 0\}.
	\end{align*}
		
	\section{Stochastically weak solutions of elliptic equations}
	\label{SecWeak}
	
	Let $L$ be a second-order linear elliptic operator and consider the homogeneous Dirichlet problem
	\begin{equation}
		\label{basic_problem}
		\begin{aligned}
			Lu &= f \hspace{1cm} \text{ in } \Omega \\
			u &= 0 \hspace{1cm} \text{ on } \partial \Omega
		\end{aligned}
	\end{equation}
	with $f \in H^{-1}(\Omega)$. For $u \in H_{0}^{1}(\Omega)$, assuming $Lu \in H^{-1}(\Omega)$, we can define
	\begin{equation*}
		R_{u}(\varphi) \coloneqq (Lu - f)(\varphi)
	\end{equation*}
	as the weak residual of the problem \eqref{basic_problem} for $\varphi \in H_{0}^{1}(\Omega)$. The weak solutions $u \in H_{0}^{1}(\Omega)$ of \eqref{basic_problem} are such that
	\begin{align*}
		R_{u}(\varphi) = 0, \ \forall \varphi \in H_{0}^{1}(\Omega) \text{ or, equivalently, } \lVert R_{u} \rVert_{-1} = 0.
	\end{align*}
	We assume that the weak solution of \eqref{basic_problem} is unique.
	
	If $f \in L^{2}(\Omega)$ and $\Omega$ and/or its boundary is more \textit{regular} (e.g. $\Omega$ convex), then the weak solution coincides with the strong solution $u \in H^{2}(\Omega) \cap H_{0}^{1}(\Omega)$ that satisfies $\lVert Lu - f \rVert_{L^2} = 0$ . We refer to \cite[Chapter~7]{hackbusch2017elliptic} for sufficient conditions for the weak solution to be unique and to coincide with the strong solution.

	\subsection{Stochastically weak solution}
	\label{SecSW}
	
	An apparently \textit{weaker} notion of solution would be to, instead of requiring $R_{u}(\varphi) = 0$ for all test functions $\varphi \in H_{0}^{1}(\Omega)$, requiring the expectation of  $|R_{u}(\varphi)|^{2}$ over \textit{random test functions} to be equal to zero. By properly choosing the probability distribution of the test functions, this notion of solution is actually equivalent to that of weak solution.
	
	Formally, let $(\Xi,\mfF,\mu)$ be a probability space and consider the spatial measurable space $(\Omega,\mfB(\Omega))$, in which $\mfB(\Omega)$ is the Borel $\sigma$-algebra of $\Omega$. Let $\{w_{k}\}$ be mean-zero independent random variables with unit variance defined on $(\Xi,\mfF,\mu)$, fix $\tau  > 0$, and define the measurable map $\Phi: \Xi \times \Omega \to \mbR$ as
	\begin{align}
		\label{random_map}
		\Phi(\xi,\cdot) \coloneqq \tau \sum_{k = 1}^{\infty} (1 + \lambda_{k})^{-1/2} \ w_{k}(\xi) \ \phi_{k}(\cdot)
	\end{align}
	recalling that $\lambda_{k}$ and $\phi_{k}$ are the eigenvalues and eigenfunctions of the Dirichlet Laplacian (cf. \eqref{DLoper}). We omit the scale parameter $\tau$ from $\Phi$ to ease notation.
	
	The random test functions we consider will be realisations of $\Phi$ which we show are in $\dot{H}^{1 - d/2 - \eps}(\Omega)$ for any $\eps > 0$. In particular, realisations of $\Phi$ are in $L^{2}(\Omega)$ only if $d = 1$, otherwise they are distributions in $\dot{H}^{-(d/2 + \eps - 1)}(\Omega)$. For $n \geq 1$, let
	\begin{align*}
		\Phi^{(n)}(\xi,\cdot) = \tau \sum_{k = 1}^{n} (1 + \lambda_{k})^{-1/2}  \ w_{k}(\xi) \ \phi_{k}(\cdot)
	\end{align*}
	be the partial sums of $\Phi$ which satisfy $\Phi^{(n)}(\xi,\cdot) \in H_{0}^{1}(\Omega)$ almost surely. We denote by $\mu[\cdot]$ expectation under $\mu$.
	
	\begin{lemma}
		\label{lemma_traj}
		The random map $\Phi$ is almost surely in $\dot{H}^{1 - d/2 - \eps}(\Omega)$ for any $\eps > 0$.
	\end{lemma}
	\begin{proof}
		We proceed as in \cite[Proposition~3.1]{korte2025smoothness}. Denote $t = 1 - d/2 - \eps$ and $A \coloneqq A(\xi) =  \sup_{n} \lVert \Phi^{(n)}(\xi,\cdot) \rVert_{\dot{H}^{t}}^{2}$, which exists since the sequence of norms\footnote{With the abuse of notation $\lVert \cdot \rVert_{-|t|} = \lVert \cdot \rVert_{\dot{H}^{t}}$ for $t < 0$.} is increasing for all $\xi \in \Xi$. By the monotone convergence theorem, the definition of the norm in $\dot{H}^{t}(\Omega)$ (cf. \eqref{norms}) and Weyl's law (cf. \eqref{wlaw}),
		\begin{align*}
			\mu\left[A\right] &= \mu\left[\sup_{n} \lVert \Phi^{(n)}(\xi,\cdot) \rVert_{\dot{H}^{t}}^{2}\right] = \lim\limits_{n \to \infty} \mu\left[\sum_{k = 1}^{n} \lambda_{k}^{t} \ |\langle \Phi^{(n)}(\xi,\cdot),\phi_{k} \rangle|^{2}\right] \\
			&= \lim\limits_{n \to \infty} \tau^{2} \sum_{k = 1}^{n} \lambda_{k}^{t} \ (1 + \lambda_{k})^{-1} \ \mu\left[(w_{k}(\xi))^{2}\right]\\
			&\leq \lim\limits_{n \to \infty} C \sum_{k = 1}^{n} (k^{2/d})^{t - 1} = \lim\limits_{n \to \infty} C \sum_{k = 1}^{n} k^{-1 - 2\eps/d} < \infty
		\end{align*}
		in which the third equality follows since $\{\phi_{k}\}$ is an orthonormal basis. This implies that $A$ is almost surely finite so $A - \lVert \Phi^{(n)}(\xi,\cdot) \rVert_{\dot{H}^{t}}^{2} \to 0$ as $n \to \infty$ almost surely. Now, for any $n_{0} \leq n \leq m$,
		\begin{align*}
			\lVert \Phi^{(n)} - \Phi^{(m)} \rVert_{\dot{H}^{t}}^{2} &= \tau^{2} \sum_{k = n + 1}^{m} \lambda_{k}^{t} \ (1 + \lambda_{k})^{-1} \ (w_{k}(\xi))^{2} \\
			&\leq \tau^{2} \sum_{k = n_{0} + 1}^{\infty} \lambda_{k}^{t} \ (1 + \lambda_{k})^{-1} \ (w_{k}(\xi))^{2} \\
			&= A - \lVert \Phi^{(n_{0})}(\xi,\cdot) \rVert_{\dot{H}^{t}}^{2} \to 0 \text{ as } n_{0} \to \infty
		\end{align*}
		so $\Phi^{(n)}$ is almost surely a Cauchy sequence in $\dot{H}^{t}(\Omega)$, hence has a limit $\Phi(\xi,\cdot) \in \dot{H}^{t}(\Omega)$ almost surely by completeness.
	\end{proof}
	
	For $u \in H_{0}^{1}(\Omega)$, define
	\begin{align}
		\label{Phi_norm}
		\lVert R_{u} \rVert_{\Phi}^{2} \coloneqq \lim\limits_{n \to \infty} \int_{\Xi} |R_{u}(\Phi^{(n)}(\xi,\cdot))|^{2} \ d\mu(\xi) = \lim\limits_{n \to \infty} \mu[|R_{u}(\Phi^{(n)})|^2].
	\end{align}
	We may also denote $\lVert R_{u} \rVert_{\Phi}^{2}$ as $\lVert Lu - f \rVert_{\Phi}^{2}$. Although $\lVert R_{u} \rVert_{\Phi}$ depends on both $\Phi$ and $\mu$, we omit the latter to ease notation. Observe that, since $\{w_{k}\}$ are independent, with mean zero and unit variance,
	\begin{align}
		\label{multi_tau} \nonumber
		\lVert R_{u} \rVert_{\Phi}^{2} &= \lim\limits_{n \to \infty} \mu\left[|R_{u}(\Phi^{(n)})|^{2}\right] \\ \nonumber
		&= \lim\limits_{n \to \infty} \tau^{2} \sum_{k,k^{\prime} = 1}^{n} (1 + \lambda_{k})^{-1/2}(1 + \lambda_{k^{\prime}})^{-1/2} R_{u}(\phi_{k}) R_{u}(\phi_{k^{\prime}}) \mu\left[w_{k} w_{k^{\prime}}\right]\\ \nonumber
		&= \lim\limits_{n \to \infty} \tau^{2} \sum_{k = 1}^{n} (1 + \lambda_{k})^{-1} \ |R_{u}(\phi_{k})|^{2} = \tau^{2} \sum_{k = 1}^{\infty} (1 + \lambda_{k})^{-1} \ |R_{u}(\phi_{k})|^{2}\\
		&\leq C  \sum_{k = 1}^{\infty} \lambda_{k}^{-1} \ |R_{u}(\phi_{k})|^{2} = C \lVert R_{u} \rVert_{-1}^{2} < \infty
	\end{align}
	so $\lVert R_{u} \rVert_{\Phi}$ is well-defined for all $u \in H_{0}^{1}(\Omega)$.
	
	We define the $\Phi$-stochastically weak solutions of \eqref{basic_problem} based on $\lVert R_{u} \rVert_{\Phi}$.
	
	\begin{definition}
		\label{def_mu_weak}
		A $\Phi$-stochastically weak solution of the homogeneous Dirichlet problem \eqref{basic_problem} is any $u \in H_{0}^{1}(\Omega)$ satisfying $\lVert R_{u} \rVert_{\Phi} = 0$.
	\end{definition}
	
	Define
	\begin{align*}
		R_{u}(\Phi(\xi,\cdot)) \coloneqq \lim\limits_{n \to \infty} R_{u}(\Phi^{(n)}(\xi,\cdot))
	\end{align*}
	for $\xi \in \Xi$. We show $|R_{u}(\Phi(\xi,\cdot))| < \infty$ almost surely and that $\lVert R_{u} \rVert_{-1}$ is equivalent to $\lVert R_{u} \rVert_{\Phi}$, which will imply that the weak and $\Phi$-stochastically weak solutions coincide.
	
	\begin{proposition}
		\label{prop_equiv}
		For all $\tau > 0$, $|R_{u}(\Phi(\xi,\cdot))| < \infty$ almost surely and there exists $c, C \coloneqq c(\Omega,\tau),C(\Omega,\tau) > 0$ such that
		\begin{align}
			\label{equiv}
			c \ \lVert R_{u} \rVert_{\Phi}^{2} \leq	\lVert R_{u} \rVert_{-1}^{2} \leq C \ \lVert R_{u} \rVert_{\Phi}^{2}
		\end{align}
		for all $u \in H_{0}^{1}(\Omega)$.
	\end{proposition}
	\begin{proof}
		Fix $u \in H_{0}^{1}(\Omega)$ and let $ a_{k} = R_{u}(\phi_{k})$. By the linearity of $R_{u}$,
		\begin{align*}
			R_{u}(\Phi^{(n)}(\xi,\cdot)) = \tau \sum_{k = 1}^{n} (1 + \lambda_{k})^{-1/2} \, w_{k}(\xi) \, a_{k},
		\end{align*}
		and it follows from Kolmogorov's two-series Theorem \cite[Theorem~2.5.6]{durrett2019probability} that
		\begin{align*}
			R_{u}(\Phi(\xi,\cdot)) = \lim\limits_{n \to \infty} \tau \sum_{k = 1}^{n} (1 + \lambda_{k})^{-1/2} \, w_{k}(\xi) \, a_{k} < \infty
		\end{align*}
		almost surely since
		\begin{equation*}
			\tau^{2} \sum_{k = 1}^{\infty} \mu\left[(1 + \lambda_{k})^{-1} (w_{k}(\xi))^{2} |a_{k}|^{2}\right] = \tau^{2} \sum_{k = 1}^{\infty} (1 + \lambda_{k})^{-1} |a_{k}|^{2} < \infty
		\end{equation*}
		by \eqref{multi_tau}. 
		
		The inequalities \eqref{equiv} are direct by recalling that
		\begin{align*}
			\lVert R_{u} \rVert_{\Phi}^{2} = \tau^{2} \sum_{k = 1}^{\infty} (1 + \lambda_{k})^{-1} \ |R_{u}(\phi_{k})|^{2} & & 
			\lVert R_{u} \rVert_{-1}^{2} = \sum_{k = 1}^{\infty} \lambda_{k}^{-1} \ |R_{u}(\phi_{k})|^{2}
		\end{align*}
		and taking
		\begin{align*}
			c = \inf_{k} \frac{\lambda_{k}^{-1} }{\tau^{2} (1 + \lambda_{k})^{-1}} & & C = \sup_{k} \frac{\lambda_{k}^{-1} }{\tau^{2} (1 + \lambda_{k})^{-1}}.
		\end{align*}
		Since $\lambda_k \geq \lambda_1 > 0$ for all $k$ and $\lambda_k \to \infty$ as $k \to \infty$, we have $c > 0$ and $C < \infty$.
	\end{proof}
	
	\begin{remark}
		We note that one can define $\lVert R \rVert_{\Phi}$ analogously to \eqref{Phi_norm} for all $R \in H^{-1}(\Omega)$ and the proof of Proposition \ref{prop_equiv} applied to $R$ actually implies the equivalence between the $H^{-1}(\Omega)$ and $\Phi$-stochastically weak norms. We also note that the unit variance assumption can be eased to, for instance, $\sup_{k} \text{var}(w_{k}) < \infty$ and the results above remain true.
	\end{remark}
	
	We have shown that the weak solutions and the $\Phi$-stochastically weak solutions coincide, as stated in the next corollary. In particular, the latter is unique if, and only if, the weak solution is unique.
	
	\begin{corollary}
		A $u \in H_{0}^{1}(\Omega)$ is a weak solution of \eqref{basic_problem} if, and only if, it is a $\Phi$-stochastically weak solution for any sequence $\{w_{k}\}$ of independent random variables with mean zero and unit variance.
	\end{corollary}
	
	\subsection{Inhomogeneous boundary conditions}
	
	An analogous equivalence between weak and $\Phi$-stochastically weak solutions holds for inhomogeneous Dirichlet problems
	\begin{equation}
		\label{inh_problem}
		\begin{aligned}
			Lu = f & & \text{ in } \Omega \\
			u = g & & \text{ on } \partial \Omega
		\end{aligned}
	\end{equation}
	for $f \in H^{-1}(\Omega)$ and $g \in H^{1/2}(\partial\Omega)$. A weak solution is a $u \in H^{1}(\Omega)$ satisfying $u|_{\partial\Omega} = g$ and $\lVert R_{u} \rVert_{-1} = 0$. We assume that the weak solution is unique (e.g., $L$ is bounded and coercive). We note that it coincides with the strong solution when $f \in L^{2}(\Omega)$ and $\Omega$ and/or its boundary is more regular, and we refer to \cite[Chapter~7]{hackbusch2017elliptic} for sufficient conditions for the solution to be unique.
	
	Since $Lu - f \in H^{-1}(\Omega)$ when $f \in H^{-1}(\Omega)$ and $u \in H^{1}(\Omega)$, the equivalence between weak and $\Phi$-stochastically weak solutions follows from a proof analogous to that of Proposition \ref{prop_equiv} which yields $c \ \lVert R_{u} \rVert_{\Phi}^{2} \leq	\lVert R_{u} \rVert_{-1}^{2} \leq C \ \lVert R_{u} \rVert_{\Phi}^{2}$ for all $u \in H^{1}(\Omega)$ with $u|_{\partial\Omega} = g$.
	
	\section{Stochastic variational physics-informed neural networks}
	\label{SecSV}
	
	In view of the equivalence between the $H^{-1}(\Omega)$ and $\Phi$ norms of the residual functional $R_{u}$, approximations to the weak solution of the boundary value problem \eqref{basic_problem} can be obtained by minimising $\lVert R_{u} \rVert_{\Phi}^{2}$ over $u$ in a rich class of functions contained in $H_{0}^{1}(\Omega)$. Likewise, approximations to the weak solution of the inhomogeneous boundary value problem \eqref{inh_problem} can be obtained by minimising $\lVert R_{u} \rVert_{\Phi}^{2}$ over $u$ in a rich class of functions contained in $H^{1}(\Omega)$ while enforcing $u|_{\partial\Omega} = g$.
	
	In this section, we propose to approximate the respective solution by minimising the \textit{empirical stochastically weak norm} of $\tilde{R}_{u_{\theta}}$, a numerical approximation of $R_{u_{\theta}}$, over the parameters $\theta$ of a neural network architecture that represents a function $u_{\theta}: \Omega \to \mbR$. Assuming that $f \in L^{2}(\Omega)$, the empirical norm squared is computed as the average of $(\tilde{R}_{u_{\theta}}(\varphi_{i}))^2$ over $\varphi_{1},\dots,\varphi_{N}$ sampled from an approximation of the truncated random map $\Phi^{(n)}$, where $\tilde{R}_{u_{\theta}}(\varphi)$ is given by numerically approximating the residual $R_{u_{\theta}}(\varphi)$, which in this case reduces to the integral
	\begin{equation}
		\label{res_L2}
		R_{u_{\theta}}(\varphi) = \int_{\Omega} (Lu_{\theta} - f)(x) \, \varphi(x) \ dx.
	\end{equation}
	This yields a general, stable and efficient algorithm, which we call \textit{stochastic variational physics-informed neural networks} (SV-PINNs), for approximating the solution of second-order linear elliptic PDEs.
	
	In Section \ref{Sec31}, we define the neural network architecture we will consider in this paper and in Section \ref{Sec32} we formally define the SV-PINNs for both homogeneous and inhomogeneous boundary conditions. In Section \ref{Sec34}, we propose a method to sample test functions in hypercube, circular and L-shaped domains and in Section \ref{Sec33} we propose two algorithms for training SV-PINNs, based on gradient descent and \lb.
	
	\subsection{Neural network ansatz}
	\label{Sec31}
	
	Empirical evidence suggests that neural networks with residual connections perform better for approximating the solution of PDEs; see for example \cite{wang2024piratenets,wang2021understanding,yu2025natural}. In this paper, we will consider the modified multilayer perceptron (MLP) architecture proposed by \cite{wang2021understanding}, which has multiplicative residual connections and is as follows\footnote{We have conducted some tests with the \textit{piratenets} proposed by \cite{wang2024piratenets} as an extension of the modified MLP, but we did not obtain better results, so we opted for the simpler architecture.}.
	
	The input of the neural network is a vector $x \in \mbR^{d \times 1}$ which is encoded into two higher dimensional vectors
	\begin{align*}
		\bs{U}(x) = \sigma(\bs{W}_{1} x + \bs{b}_{1}) & & \bs{V}(x) = \sigma(\bs{W}_{2} x + \bs{b}_{2})
	\end{align*}
	for $\bs{W}_{1}, \bs{W}_{2} \in \mbR^{r \times d}$ and $\bs{b}_{1}, \bs{b}_{2} \in \mbR^{r \times 1}$, and an activation function $\sigma: \mbR \to \mbR$ that is applied coordinate-wise to the respective vector. We consider $\sigma = \tanh$ and we may fix the bias terms $\bs{b}_{1},\bs{b}_{2} = 0$, as will be discussed later.
	
	Defining $\bs{g}^{(0)}(x) = x$ for all $x \in \mbR^{d \times 1}$, the modified MLP with $K$ hidden layers of $r$ nodes computes a function of $x$ iteratively for each hidden layer $l = 1,\dots,K$ as
	\begin{equation*}
		\begin{aligned}
			\bs{f}^{(l)}(x) &= \sigma(\bs{W}^{(l)} \bs{g}^{(l-1)}(x) + \bs{b}^{(l)})\\
			\bs{g}^{(l)}(x) &= \bs{f}^{(l)}(x) \odot \bs{U}(x) + (1 - \bs{f}^{(l)}(x)) \odot \bs{V}(x)
		\end{aligned}
	\end{equation*}
	where $\odot$ means coordinate-wise multiplication, $\bs{W}^{(1)} \in \mbR^{r \times d}$, $\bs{W}^{(l)} \in \mbR^{r \times r}, l > 1$, and $\bs{b}^{(l)} \in \mbR^{r \times 1}$, which may be fixed to zero. The output of the neural network is then given by
	\begin{align*}
			u_{\theta}(x) \coloneqq \bs{W}^{(K + 1)} \bs{g}^{(K)}(x) + \bs{b}^{(K + 1)}
	\end{align*}
	in which $\bs{W}^{(K+1)} \in \mbR^{1 \times r}$ and $\bs{b}^{(K + 1)} \in \mbR$, that may also be fixed to zero, and
	\begin{align*}
		\theta \coloneqq \left\{\bs{W}_{1},\bs{W}_{2},(\bs{W}^{(l)},\bs{b}^{(l)})_{l= 1}^{K + 1}\right\}
	\end{align*}
	are the parameters of the architecture. Observe that a residual connection is given as the multiplication by the encodings $\bs{U}$ and $\bs{V}$ of the input. The parameters are initialised with the Glorot scheme \cite{glorot2010understanding}.
	
	\subsubsection{Domain-aware Fourier features and hard boundary constraints}
	
	Neural networks are known to have training spectral bias towards low-frequency functions which yields frequency-dependent learning speeds \cite{rahaman2019spectral}. This fact has also been observed in PINNs through an analysis of its neural tangent kernel, and an encoding of the input into random Fourier features was proposed to allow recovering high-frequency, multi-scale solutions \cite{wang2021eigenvector}. Recently, \cite{calero2026enhancing} proposed the domain-aware Fourier features (DAFF), an encoding by the eigenfunctions of the Dirichlet Laplacian in $\Omega$ as a manner of changing the spectral bias while enforcing zero boundary conditions in the neural network architecture. This encoding is especially suitable for solving \eqref{basic_problem} and will be considered in the SV-PINNs.
	
	For a vector $\bs{k} = (k_{1},\dots,k_{d}) \in \mbN^{d}$, let $\phi_{\bs{k}}$ be the eigenfunctions of the Dirichlet Laplacian that satisfy\footnote{We index the eigenfunctions and eigenvalues by vectors in $\mbN^{d}$ now to facilitate the application of the method to hypercube domains in $\mbR^{d}$.}
	\begin{align*}
		-\Delta \phi_{\bs{k}} = \lambda_{\bs{k}}\phi_{\bs{k}} & & \phi_{\bs{k}}|_{\partial\Omega} = 0
	\end{align*}
	with eigenvalues $\lambda_{\bs{k}}$. For a subset $\bs{I} = \{\bs{k}_{1},\dots,\bs{k}_{m}\}$ of $m$ vectors, we consider the encoding
	\begin{align*}
		\phi_{\bs{I}}(x) = \begin{pmatrix}
			\phi_{\bs{k}_{1}}(x) & \cdots & \phi_{\bs{k}_{m}}(x)
		\end{pmatrix}
	\end{align*}
	for $x \in \Omega$ which is a mapping $\phi: \Omega \subset \mbR^{d \times 1} \to \mbR^{m \times 1}$. For example, when $\Omega = (0,1)^{d}$, the eigenfunctions are given by the product of sines
	\begin{equation}
		\label{eigenf}
		\phi_{\bs{k}}(x) = 2^{d/2} \prod_{j = 1}^{d} \sin\left(k_{j} \, \pi \, x_{j}\right)
	\end{equation}
	and $\lambda_{\bs{k}} = \pi^{2} \sum_{j} k_{j}^{2}$. Then, instead of inputting $x$ to the modified MLP, we input the domain-aware feature vector $\phi_{\bs{I}}(x)$, making the obvious modifications in the architecture (e.g., exchanging $d$ by $m$ in the dimension of the matrices). 
	
	We note that DAFF differ from random Fourier features, in which the feature map is given by functions
	\begin{align*}
		\phi_{j}(x) = \begin{pmatrix}
			\sin(\bs{A}_{j}x) & \cos(\bs{A}_{j}x)
		\end{pmatrix}
	\end{align*}
	for a matrix $\bs{A}_{j}$ with entries sampled independently from a Gaussian distribution with mean zero and a standard deviation associated with the magnitude of the frequency desired on the features; see \cite{wang2021eigenvector} for more details. On the other hand, DAFF select explicitly the desired frequencies $\bs{I}$ without any randomness. Furthermore, since $\phi_{\bs{k}}$ equals zero on the boundary and $\sigma(0) = 0$ when $\sigma$ is the \textit{tanh}, if DAFF are considered and the bias terms are fixed to zero, then $u_{\theta}|_{\partial\Omega} = 0$ for all values of $\theta$.
	
	Indeed, since $\phi_{\bs{k}}(x) = 0$ for $x \in \partial\Omega$ for any vector $\bs{k}$, we conclude that $\sigma(\bs{W} \phi_{\bs{I}}(x)) = \bs{0}$  for any matrix $\bs{W}$ which implies that $u_{\theta}(x) = 0$ since it is given by the composition and multiplication of analogous expressions. Therefore, DAFF not only nudge the spectral bias of neural networks towards desired frequencies, but, by taking the bias terms as zero, they impose a constraint on the architecture such that zero boundary conditions are automatically satisfied. In order to approximate the solution of inhomogeneous boundary value problems, it is enough to consider DAFF encoding with the biases as learnable parameters or consider Fourier features. We refer to \cite{calero2026enhancing} for more details about DAFF.	
	
	\subsection{SV-PINNs formulation}
	\label{Sec32}
	
	The SV-PINN with DAFF hard constraints to approximate the solution of the homogeneous boundary value problem \eqref{basic_problem} is the solution of the optimisation problem
	\begin{align}
		\label{sol_SVPINNs}
		\min\limits_{\theta} \, \lVert Lu_{\theta} - f \rVert_{\Phi}^{2}
	\end{align}
	that is, finding a set of parameters of the fixed neural network architecture that minimises the $\Phi$-stochastically weak norm of the residual of the PDE \eqref{basic_problem}. We note that the solution of \eqref{sol_SVPINNs} does not depend on the scale parameter $\tau$ since it is a multiplicative factor on $\lVert \cdot \rVert_{\Phi}$ (cf. \eqref{multi_tau}).
	
	For the inhomogeneous boundary value problem \eqref{inh_problem}, the SV-PINN is the solution of 
	\begin{align*}
		\min\limits_{\theta} \, \lVert Lu_{\theta} - f \rVert_{\Phi}^{2} + \lambda \ \lVert u_{\theta} - g\rVert_{L^{2}(\partial\Omega)}^{2}
	\end{align*}
	for a weight $\lambda > 0$, in which the optimisation is over the parameters of a neural network architecture without hard boundary constraints (e.g., with DAFF, but trainable bias terms). We observe that, in this case, the solution depends on $\tau$, more specifically, on the ratio $\tau^{2}/\lambda$. 
	
	In contrast, for $f \in L^{2}(\Omega)$, the standard PINN formulation solves
	\begin{align*}
		\min\limits_{\theta} \, \lVert Lu_{\theta} - f \rVert_{L^{2}}^{2} + \lambda \ \lVert u_{\theta} - g\rVert_{L^{2}(\partial\Omega)}^{2}
	\end{align*}
	in which the stochastically weak norm of the residual is replaced by the $L^{2}(\Omega)$ norm, and $\lambda = 0$ if hard boundary constraints are inserted into the architecture. In practice, the $L^2$ norms are replaced by empirical versions and PINNs are trained by minimising the loss function
	\begin{align*}
		\mcL_{PINNs}(\theta) &= \mcL_{c}(\theta) + \lambda \, \mcL_{b}(\theta) \\
		&\coloneqq \frac{1}{N_{c}} \sum_{i = 1}^{N_{c}} (Lu_{\theta}(x_{c}^{(i)}) - f(x_{c}^{(i)}))^{2} + \frac{\lambda}{N_{b}} \sum_{i = 1}^{N_{b}} (u_{\theta}(x_{b}^{(i)}) - g(x_{b}^{(i)}))^{2}
	\end{align*}
	in which $x_{c}^{(i)} \in \Omega$ are collocation points where the PDE is strongly enforced and $x_{b}^{(i)} \in \partial\Omega$ are boundary points in which the boundary condition is strongly enforced.
	
	In order to train SV-PINNs, the $\Phi$ norm of the residuals can be approximated by an empirical average. This can be accomplished by sampling test functions $\varphi_{1},\dots,\varphi_{N}$ from an approximation of the truncated random map $\Phi^{(n)}$, evaluating them at the collocation points $x_{c}^{(i)}$ and approximating $\lVert Lu_{\theta} - f \rVert_{\Phi}^{2}$ by
	\begin{align}
		\label{app_sw}
		\mcL_{\Phi^{(n)}}(\theta) \coloneqq \frac{1}{N} \sum_{j = 1}^{N} \left(\frac{1}{N_{c}} \sum_{i=1}^{N_{c}} (Lu_{\theta}(x_{c}^{(i)}) - f(x_{c}^{(i)}))\varphi_{j}(x_{c}^{(i)})\right)^{2}
	\end{align}
	in which the inner sum is an approximation of the integral \eqref{res_L2}. The SV-PINNs are then trained by minimising
	\begin{align}
		\label{loss_SVPINNs}
		\mcL_{SV}(\theta) = \mcL_{\Phi^{(n)}}(\theta) + \lambda \, \mcL_{b}(\theta)
	\end{align}
	with $\lambda = 0$ when the architecture is constrained to satisfy the boundary conditions.
	
	\subsection{Sampling test functions from the Whittle-Matérn process}
	\label{Sec34}
	
	The computational feasibility of the SV-PINNs relies on efficiently sampling the test functions and evaluating them at the collocation points so that \eqref{loss_SVPINNs} can be computed and minimised. We discuss how this can be done in hypercube, circular and L-shaped domains.
	
	\subsubsection{Hypercube domains}
		
	In hypercube domains, test functions can be sampled by solving a discretised version of a particular stochastic partial differential equation (SPDE) with the discrete sine transform (DST-I), as follows.
	
	In the special case when $\{w_{\bs{k}}\}$ are Gaussian random variables, for $s > 0$
	\begin{align}
		\label{expansion}
		\Phi_{s}(\xi,\cdot) \coloneqq \tau \sum_{\bs{k}} (1 + \lambda_{\bs{k}})^{-s/2} w_{\bs{k}}(\xi) \phi_{\bs{k}}(\cdot)
	\end{align}
	is known as the Whittle-Matérn process and is the unique solution of the homonymous SPDE
	\begin{align}
		\label{SPDE}
		(1 - \Delta)^{s/2}\Phi_{s} = \tau \, \mcW & & \Phi_{s}|_{\partial\Omega} = 0
	\end{align} 
	in which $\mcW: \Xi \times \Omega \to \mbR$ is white noise, i.e., $\mcW(x)$ is a mean zero Gaussian random variable with unit variance and $\mcW(x)$ is independent of $\mcW(y)$ for all $x,y \in \Omega$. Expression \eqref{expansion} is the Karhunen-Loève expansion \cite{wang2008karhunen} of the solution of \eqref{SPDE}.
	
	The solution $\Phi_{s}$ is a Gaussian random field, meaning that $(\Phi_{s}(x^{(1)}), \dots, \\ \Phi_{s}(x^{(n)}))$ has a multivariate Gaussian distribution for all $x^{(1)},\dots,x^{(n)} \in \Omega$. When $s > d/2$, it is a Gaussian process, meaning there exists a positive definite symmetric covariance function $\mc{K}: \Omega \times \Omega \to \mbR$, known as Matérn kernel, such that $(\Phi_{s}(x^{(1)}), \dots, \Phi_{s}(x^{(n)}))$ has a multivariate Gaussian distribution with mean zero and covariance matrix $\Sigma = [\mc{K}(x^{(i)},x^{(j)})]_{i,j}$. We note this is usually not the case in our framework since $s = 1 > d/2$ only for $d = 1$. This SPDE was first studied by Whittle \cite{whittle1954stationary,whittle1963stochastic} and we refer to \cite{bolin2023equivalence,korte2025smoothness,lindgren2022spde,lindgren2011explicit} for a modern analysis.
	
	When the domain is $\Omega = (0,1)^{d}$, in order to efficiently sample from an approximation of the truncated random map $\Phi^{(n)}$ at points in a grid of $\Omega$, we solve a discretised version of \eqref{SPDE} with $s = 1$ by applying the DST-I, solving on the discrete Fourier space, and then applying its inverse. This is an efficient algorithm, with complexity $\mathcal{O}(n^{d} \log n)$ for a grid of $n^{d}$ points, and the error between the approximation of the $\Phi$ norm of the residual functional by \eqref{app_sw} can be controlled in expectation, as will be done in Section \ref{Sec_error}.
	
	Fix a grid of $n^{d}$ points $x^{(\bs{k})} = (x_{k_{1}},\dots,x_{k_{d}}) \in \Omega$ with $x_{k_{j}} = k_{j}/(n + 1)$ for $k_{j} = 1,\dots,n$ and $j = 1,\dots,d$. Denote by $h = 1/(n + 1)$ the distance $x_{k_{j} + 1} - x_{k_{j}}$ and let $\Delta_{h}$ be the discretised Laplace operator in this grid defined as
	\begin{align*}
		\Delta_{h} u(x^{(\bs{k})}) = \frac{1}{h^{2}} \sum_{j = 1}^{d} u(x^{(\bs{k} + \bs{e}_{j})}) - 2 \, u(x^{(\bs{k})}) + u(x^{(\bs{k} - \bs{e}_{j})})
	\end{align*}
	in which $(\bs{e}_{j})_{i} = \delta_{ij}$ and $u$ is a function satisfying $u(x) = 0$ for $x \in \partial\Omega$. We consider the discretised SPDE
	\begin{equation}
		\label{discrete_WMP}
		(1 - \Delta_{h})^{1/2} \Phi^{(h)}(x^{(\bs{k})}) = \tau \ \mcW(x^{(\bs{k})}) 
	\end{equation}
	which is actually a system of $n^{d}$ stochastic equations that can be solved with the DST-I, as follows.
	
	The DST-I with orthonormal basis applied to a discretisation $u(x^{(\bs{k})})$ of a function $u$ is given by
	\begin{align}
		\label{DST} \nonumber
		DST(u)(x^{(\bs{k}^{\prime})}) &\coloneqq \hat{u}(x^{(\bs{k}^{\prime})}) = (2h)^{d/2} \sum_{|\bs{k}| \leq n} u(x^{(\bs{k})}) \, \prod_{j = 1}^{d} \sin(\pi \, k_{j} \, k^{\prime}_{j} \, h)\\
		&= h^{d/2} \sum_{|\bs{k}| \leq n} u(x^{(\bs{k})}) \, \phi_{\bs{k}}(x^{(\bs{k}^{\prime})})
	\end{align}
	in which $|\bs{k}| = \max_{j} k_{j}$ and the last equality is due to \eqref{eigenf}. This is a self-inverse transformation, so $DST^{-1} = DST$, and it can be written in matrix form as $DST(\bs{u}) = S\bs{u}$ in which $\bs{u} = (u(x^{(\bs{k})}): |\bs{k}| \leq n)^{T} \in \mbR^{n^d \times 1}$ and $S \in \mbR^{n^{d} \times n^{d}}$ satisfies $S^{T}S = I$. The DST-I diagonalises $-\Delta_{h}$ which implies that
	\begin{align*}
		DST((1 - \Delta_{h})^{t} \, u(x^{(\bs{k})})) = (1 + \lambda_{\bs{k}}^{(h)})^{t} \, \hat{u}(x^{(\bs{k}^{\prime})})
	\end{align*}
	for all $t > 0$ in which
	\begin{align*}
		\lambda_{\bs{k}}^{(h)} = \frac{4}{h^2} \sum_{j = 1}^{d} \sin^{2}\left(\frac{\pi \, k_{j} \, h}{2}\right)
	\end{align*}
	are the eigenvalues of $-\Delta_{h}$. We refer to \cite{britanak2010discrete} for more details about the DST-I and fast algorithms to compute it.
		
	Applying the DST-I with orthonormal basis to both sides of \eqref{discrete_WMP} we obtain
	\begin{equation*}
		(1 + \lambda_{\bs{k}}^{(h)})^{1/2} \widehat{\Phi^{(h)}}(x^{(\bs{k})}) = \tau \, \widehat{\mcW}(x^{(\bs{k})})
	\end{equation*}
	and by applying its inverse it follows that
	\begin{equation}
		\label{DST_sol}
		\Phi^{(h)}(x^{(\bs{k})}) = h^{-d/2} \, \tau \, DST^{-1}\left[(1 + \lambda_{\bs{k}}^{(h)})^{-1/2} \, \widehat{\mcW}(x^{(\bs{k})})\right]
	\end{equation}
	in which the $h^{-d/2}$ normalisation factor is necessary to cancel with that of \eqref{DST} so the right-hand side of \eqref{DST_sol} equals the respective truncated process.\footnote{See \eqref{t2} for more details.}
	
	By sampling $\mcW(x^{(\bs{k})})$, independent standard Gaussian random variables, at each position $x^{(\bs{k})}$ in the grid, the value of the discretised process $\Phi^{(h)}(x^{(\bs{k})})$ is obtained by formula \eqref{DST_sol}. This represents a sampled test function $\varphi_{j}(x^{(\bs{k})}) = \Phi^{(h)}(x^{(\bs{k})})$ evaluated at the grid that can be substituted into \eqref{app_sw} by taking the collocation points in the grid with $N_{c} = n^{d}$. Drawing $N$ independent samples of $\mcW(x^{(\bs{k})})$ yields random test functions $\varphi_{1},\dots,\varphi_{N}$ that can be evaluated at the grid and used to compute \eqref{app_sw}.
	
	Examples of realisations of $\Phi$ for $d = 1,2,3$ are presented in Figure \ref{fig_ex} \textbf{(a)}. Observe that, as expected (cf. Lemma \ref{lemma_traj}), the roughness of the trajectories increases visually with the dimension.
	
	\begin{figure}[htbp]
		\centering
		\captionsetup{justification=centering}
		\begin{subfigure}{\linewidth}
			\centering
			\includegraphics[width=\linewidth]{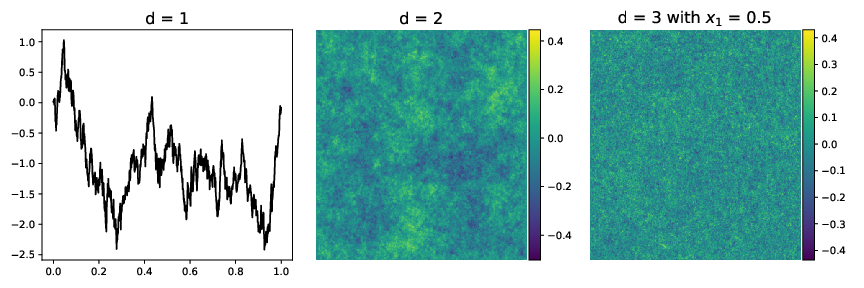}
			\caption{Hypercube domain $(0,1)^{d}$}
		\end{subfigure}
		\vspace{0.5cm} 		
		\begin{subfigure}{0.48\linewidth}
			\centering
			\includegraphics[width=0.75\linewidth]{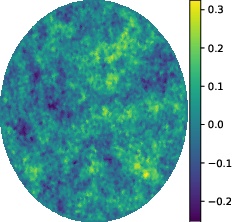}
			\caption{Radius 1 circle}
		\end{subfigure}
		\hfill
		\begin{subfigure}{0.48\linewidth}
			\centering
			\includegraphics[width=0.75\linewidth]{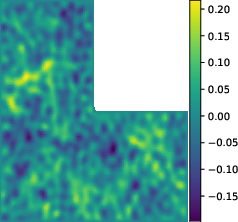}
			\caption{L-shaped domain $(-1,1)^2\setminus (0,1)^2$}
		\end{subfigure}
		
		\caption{Examples of realisations of $\Phi$. \textbf{(a)} Sampled by \eqref{DST_sol} with $\tau = 10$ for $d = 1,2,3$ in a grid with, respectively, $1024$, $1024^2$ and $512^3$ points. For $d = 3$ we present the slice $\varphi(0.5,\cdot,\cdot)$ of the sampled realisation. \textbf{(b)} Sampled via the truncated sum \eqref{Phi_circle} with $M = 4,096$, $\tau = 0.1$ and $22,500$ points sampled uniformly. \textbf{(c)} Sampled by truncating the sum \eqref{Phi_L} with $M = 1,024$, $\tau = 0.1$ and $7,105$ points in a grid.} 
		\label{fig_ex}
	\end{figure}
	
	\subsubsection{Circular domains in 2D}
	
	In the circular domain $\Omega = \{x \in \mbR^2: x_{1}^{2} + x_{2}^{2} < R^{2}\}$ for $R > 0$, the eigenvalues of the Dirichlet Laplacian (cf. \eqref{DLoper}) are
	\begin{align*}
		\lambda_{n,m} = \left(\frac{j_{n,m}}{R}\right)^2, \, \, n = 0, 1, 2, \cdots \, \, m = 1, 2, 3, \cdots
	\end{align*}
	in which $j_{n,m}$ is the $m$-th positive zero of the Bessel function of the first kind
	\begin{align*}
		J_{n}(r) = \sum_{l = 0}^{\infty} \frac{(-1)^{l}}{l! \, \Gamma(l + n + 1)} \left(\frac{r}{2}\right)^{2l + n}
	\end{align*}
	and the eigenfunctions associated with them are
	\begin{equation}
		\label{eigen_circle}
		\begin{aligned}
			\phi_{n,m,0}(r,\alpha) &= C_{n,m,0} \, J_{n}\left(\sqrt{\lambda_{n,m}} r\right) \cos(n\alpha) \text{ and } \\
			\phi_{n,m,1}(r,\alpha) &= C_{n,m,1} \, J_{n}\left(\sqrt{\lambda_{n,m}} r\right) \sin(n\alpha)
		\end{aligned}
	\end{equation}	
	for constants $C_{n,m,0},C_{n,m,1}$ that make the eigenfunctions an orthonormal basis of $L^{2}(\Omega)$. We note that $\phi_{n,m,\ell}(R,\alpha) = 0$ since $j_{n,m}$ is a root of $J_{n}$ so the Dirichlet boundary condition is indeed satisfied. We refer to the classical \cite[Section~5.5]{courant2024methods} for more details.
	
	Although algorithms somewhat analogous to the DST-I exist, for example, those based on the discrete Hankel transform \cite{johnson1987improved}, which could in principle be applied to solve \eqref{SPDE} on a grid, we instead sample the test functions from the truncated process
	\begin{align}
		\label{Phi_circle}
		\Phi^{(M)}(\xi,\cdot) = \tau \sum_{n = 0}^{M - 1} \sum_{m = 1}^{M} \sum_{\ell = 0,1} (1 + \lambda_{n,m})^{-1/2} w_{n,m,\ell}(\xi) \phi_{n,m,\ell}(\cdot)
	\end{align}
	for $M > 0$, in which $w_{n,m,\ell}$ are independent random variables with standard Gaussian distribution. Instead of considering a grid of the circle, we will sample the collocation points $x_{c}^{(i)}$ with uniform distribution in $\Omega$, so the inner sum in \eqref{app_sw} represents a Monte Carlo integration approximation of $R_{u_{\theta}}(\varphi_{j})$ (cf. \eqref{res_L2}). An example of a realisation of $\Phi^{(M)}$ is presented in Figure \ref{fig_ex} \textbf{(b)}.
	
	\subsubsection{L-shaped domain in 2D}
	\label{SecL}
	
	In the L-shaped domain $\Omega = (-1,1)^{2}\setminus (0,1)^{2}$, the eigenfunctions of the Dirichlet Laplacian do not have an explicit formula and should be approximated numerically. For this, we can apply the Finite Element Method (FEM) to approximate the eigenvalues and the eigenfunctions in a grid of $\Omega$.  We refer to \cite[Chapter~6]{strang1974analysis} for more details.
	
	For a grid of $\Omega$ with collocation points $x_{c}^{(i)}, i = 1,\dots,N_{c}$,  we sample test functions at these points from the truncated process
	\begin{align}
		\label{Phi_L}
		\hat{\Phi}^{(M)}(\xi,x^{(i)}_{c}) = \sum_{k = 1}^{M} (1 + \hat{\lambda}_{k})^{-1/2} w_{k}(\xi) \hat{\phi}_{k}(x^{(i)}_{c})
	\end{align} 
	in which $\hat{\lambda}_{k}$ and $\hat{\phi}_{k}(x^{(i)}_{c})$ are, respectively, approximations of the $M$ smallest eigenvalues and corresponding eigenfunctions of $-\Delta$ on the grid, and $w_{k}$ are independent random variables with standard Gaussian distribution. Since the collocation points are in a grid, the inner sum in \eqref{app_sw} represents a quadrature rule to approximate the integral $R_{u_{\theta}}(\varphi_{j})$ (cf. \eqref{res_L2}). We refer to Appendix \ref{app_details} for more detail about the FEM implementation. An example of a realisation of $\hat{\Phi}^{(M)}$ is presented in Figure \ref{fig_ex} \textbf{(c)}.
		
	\subsection{Optimisation algorithms}
	\label{Sec33}
	
	In this paper, we will consider two algorithms to train SV-PINNs by minimising the loss \eqref{loss_SVPINNs}: gradient descent and \lbs \cite{nocedal1980updating}. The former, through the ADAM \cite{kingma2014adam} implementation, can be easily applied by computing $Lu_{\theta}$ with automatic differentiation \cite{baydin2018automatic}. The latter has had great success on training neural networks by minimising variational loss functions, such as the Deep Ritz framework \cite{yu2025natural}, when the objective function is smoother than the $L^2$ norm of the residuals, and might be especially suitable for SV-PINNs. Although \lbs takes longer per step, it may converge faster and could be preferable for SV-PINNs, as will be illustrated in the experiments in Section \ref{Sec_exp}.
	
	We note that both algorithms can converge to sub-optimal local minima of \eqref{loss_SVPINNs} and there is no theoretical guarantee that the local minimum has a loss close to that of the global minimum. Nevertheless, as is also the case with PINNs, we can hope for this to be the case in practice.
	
	\subsection{The pragmatic view}
	
	The empirical weak loss \eqref{app_sw} can be rewritten as
	\begin{equation}
		\label{loosW}
		\begin{aligned}
			\mcL_{\Phi^{(n)}}(\theta) = &\frac{1}{N N_{c}^{2}} \Bigg[\sum_{i = 1}^{N_c} \lambda_{c}^{(i)} (Lu_{\theta}(x_{c}^{(i)}) - f(x_{c}^{(i)}))^{2} + \\
			&\sum_{\substack{i,i^\prime = 1 \\ i \neq i^\prime}}^{N_{c}} \lambda_{c}^{(i,i^\prime)} (Lu_{\theta}(x_{c}^{(i)}) - f(x_{c}^{(i)}))(Lu_{\theta}(x_{c}^{(i^\prime)}) - f(x_{c}^{(i^\prime)}))\Bigg]
		\end{aligned}
	\end{equation}
	in which
	\begin{align*}
		\lambda_{c}^{(i)} = \sum_{j = 1}^{N} |\varphi_{j}(x_{c}^{(i)})|^2 & & \lambda_{c}^{(i,i^\prime)} = \sum_{j = 1}^{N} \varphi_{j}(x_{c}^{(i)})\varphi_{j}(x_{c}^{(i^\prime)})
	\end{align*}
	for a fixed realisation of the sampled test functions $\varphi_{j}$, which are held fixed during training. Therefore, the SV-PINNs can be interpreted as PINNs with a loss function given by pointwise weights, akin to, for example, the self-adaptive PINNs \cite{mcclenny2023self}, but with weights $\lambda_{c}^{(i)}$ that are random and non-trainable and with an extra loss given by cross-term residuals with weights $\lambda_{c}^{(i,i^\prime)}$.
	
	Notwithstanding the mathematical reasoning that led to the SV-PINNs via the stochastically weak solutions, SV-PINNs can be pragmatically understood as PINNs with a particular weighting scheme. Their main differences with other well-established methods, such as self-adaptive \cite{mcclenny2023self}, gradient normalisation \cite{wang2021understanding}, neural tangent kernel-based weights \cite{wang2022and} and causal weights \cite{wang2022respecting}, are the cross-term residuals and the fact that the weights are random and non-trainable. To the best of our knowledge, weights are either fixed deterministically or are trainable, evolving deterministically with the training steps\footnote{Of course this evolution is random when stochastic GD is applied, but, conditioned on the randomness of the training process, the evolution of the weights is deterministic.}. We note that the variational PINNs \cite{kharazmi2019variational,kharazmi2021hp} have a loss function analogous to \eqref{loosW}, but with deterministic weights given by the fixed test functions.
	
	Another pragmatic view is to interpret the SV-PINNs as training the neural network enforcing that the vector of residuals $(R_{u_{\theta}}(x_{c}^{(i)}): i = 1,\dots,N_{c}) \in \mbR^{N_{c}}$ evaluated at the collocation points is orthogonal to the space
	\begin{align*}
		\hat{V}_{N}^{N_c} = \text{span}\{\varphi_{1}(\bs{x}),\dots,\varphi_{N}(\bs{x})\} \subset \mbR^{N_{c}},
	\end{align*} 
	in which $\varphi_{j}(\bs{x}) = (\varphi_{j}(x_{c}^{(i)}): i = 1,\dots,N_{c})$ denotes the vector of pointwise evaluations of the test function $\varphi_{j}$ at the collocation points interpreted as vectors in $\mbR^{N_c}$ endowed with the Euclidean norm. In view of the randomness of the test functions, it is expected that\footnote{This fact needs to be formally proved, but this proof is outside the scope of this paper.} $\text{dim}(\hat{V}_{N}^{N_c}) = N_{c}$ if $N$ is \textit{large enough}, so $(\hat{V}_{N}^{N_c})^{\perp} = \{0\}$ and the residual is actually being minimised pointwise. 
	
	This heuristic implies that, if the neural network can be properly trained, the pointwise residuals at the collocation points should be close to zero even by training with a weak formulation instead of minimising the strong empirical $L^2(\Omega)$ norm. On the other hand, if $N \ll N_{c}$, then $(\hat{V}_{N}^{N_c})^{\perp}$ might be \textit{too large} and the pointwise residuals will not be properly minimised. This is something that can be clearly observed in experiments. Future studies should rigorously prove (or disprove) these heuristics.
	
	While these interpretations are useful for intuition and comparison with other methods in the literature, the defining feature of SV-PINNs remains their variational formulation as the minimisation of a stochastically weak residual norm equivalent to the $H^{-1}(\Omega)$ norm, which underpins their mathematical basis.

	\section{Empirical approximation of the $\Phi$-norm}
	\label{Sec_error}
	
	In this section, we will analyse the approximation error incurred by replacing $\lVert R_{u_{\theta}} \rVert_{\Phi}^{2}$ with the empirical version $\mc{L}_{\Phi^{(n)}}(\theta)$ and prove that, for $\theta$ fixed, the error converges to zero in expectation over the samples of the test functions as the grid size and the number of sampled test functions go to infinity. 
	
	We assume from now on that $\Omega = (0,1)^{d}$ with $d \leq 3$, that $f \in H^{2}(\Omega)$, and that the coefficients of $L$ are regular enough so that $Lu_{\theta} \in H^2(\Omega)$. We recall that $h = 1/(n+1)$.
	
	\begin{theorem}
		\label{main_theorem}
		Let $\Omega = (0,1)^{d}$ with $d \leq 3$ and $f \in H^{2}(\Omega)$, and assume that $Lu_{\theta} \in H^2(\Omega)$ for all $\theta$. Then, 
		\begin{align}
			\label{main1}
			\mu\left[\left|(nh)^{2d} \mcL_{\Phi^{(n)}}(\theta) - \lVert R_{u_{\theta}} \rVert_{\Phi}^{2}\right|\right] \leq C \, \left[h^{2 - d/2} + N^{-1/2}\right] \, \lVert R_{u_{\theta}} \rVert_{H^{2}}^{2}.
		\end{align}
		In particular, $\lim\limits_{h \to 0} \lim\limits_{N \to \infty} \mu\left[\left|(nh)^{2d} \mcL_{\Phi^{(n)}}(\theta) - \lVert R_{u_{\theta}} \rVert_{\Phi}^{2}\right|\right] = 0$ for all $\theta$.
	\end{theorem}
	
	Before proving Theorem \ref{main_theorem}, we will deduce an expression for $\mu\left[\mcL_{\Phi^{(n)}}(\theta)\right]$ in terms of the discretised inner product
	\begin{align}
		\label{discrete_IN}
		\langle u,v \rangle_{h} \coloneqq h^{d} \sum_{|\bs{k}| \leq n} u(x^{(\bs{k})}) \ v(x^{(\bs{k})}), \ \ \ u,v \in H^2(\Omega)
	\end{align}
	with the associated discrete seminorm $\lVert u \rVert_{h} = (\langle u,u \rangle_h)^{1/2}$. For this, let
	\begin{align}
		\label{t2}
		\Phi^{(h)}(\xi,\cdot) = \tau \, \sum_{|\bs{k}| \leq n} (1 + \lambda_{\bs{k}}^{(h)})^{-1/2} \, w_{\bs{k}}(\xi) \, \phi_{\bs{k}}(\cdot)
	\end{align}
	where $w_{\bs{k}}$ are mean zero independent Gaussian random variables with unit variance. This expression satisfies \eqref{DST_sol} for all points $x^{(\bs{k})}$ in the grid. Indeed, $\widehat{\mcW}(x^{(\bs{k})})$ are mean zero independent Gaussian random variables with unit variance since the DST-I is an orthogonal transformation. Applying the DST-I formula \eqref{DST} to $\tau (1 + \lambda_{\bs{k}}^{(h )})^{-1/2}\widehat{\mcW}(x^{(\bs{k})})$ we obtain \eqref{t2} multiplied by $h^{d/2}$, which cancels with the $h^{-d/2}$ factor of \eqref{DST_sol}.
	
	Therefore, denoting $R_{u_{\theta}}(x^{(\bs{k})}) \coloneqq (Lu_{\theta} - f)(x^{(\bs{k})})$,
	\begin{align}
		\label{inExp} \nonumber
		\mu\left[\mcL_{\Phi^{(n)}}(\theta)\right] &= \mu\left[\left(\frac{1}{n^{d}} \sum_{|\bs{k}| \leq n} R_{u_{\theta}}(x^{(\bs{k})}) \tau \sum_{|\bs{k}^\prime| \leq n} \frac{w_{\bs{k}^\prime} \phi_{\bs{k}^{\prime}}(x^{(\bs{k})})}{(1 + \lambda^{(h)}_{\bs{k}^{\prime}})^{1/2}}\right)^2\right]\\ \nonumber
		&= \frac{\tau^2 h^{-2d}}{n^{2d}} \mu\left[\left(\sum_{|\bs{k}^\prime| \leq n} (1 + \lambda^{(h)}_{\bs{k}^{\prime}})^{-1/2} w_{\bs{k}^\prime} \sum_{|\bs{k}| \leq n} h^d R_{u_{\theta}}(x^{(\bs{k})}) \phi_{\bs{k}^{\prime}}(x^{(\bs{k})})\right)^2\right]\\ \nonumber
		&= \frac{\tau^2 h^{-2d}}{n^{2d}} \mu\left[\left(\sum_{|\bs{k}^\prime| \leq n} (1 + \lambda^{(h)}_{\bs{k}^{\prime}})^{-1/2} w_{\bs{k}^\prime} \langle R_{u_{\theta}},\phi_{\bs{k}^{\prime}} \rangle_{h}\right)^2\right]\\
		&= \frac{\tau^2 h^{-2d}}{n^{2d}} \sum_{|\bs{k}^\prime| \leq n} \frac{|\langle R_{u_{\theta}},\phi_{\bs{k}^\prime} \rangle_{h}|^2}{1 + \lambda_{\bs{k}^\prime}^{(h)}}
	\end{align}
	in which the last equality follows since $w_{\bs{k}}$ are independent with mean zero and unit variance. We first bound $\left|(nh)^{2d}\mu[\mcL_{\Phi^{(n)}}(\theta)] - \lVert R_{u_{\theta}} \rVert_{\Phi}^{2}\right|$.
	
	\begin{lemma}
		\label{L1}
		Under the assumptions of Theorem \ref{main_theorem},
		\begin{align}
			\label{res_L1}
			\left|(nh)^{2d}\mu[\mcL_{\Phi^{(n)}}(\theta)] - \lVert R_{u_{\theta}} \rVert_{\Phi}^{2}\right| \leq C \, h^{2 - d/2} \, \lVert R_{u_{\theta}} \rVert_{H^{2}}^{2}
		\end{align}
		for all $\theta$. In particular, the right-hand side converges to zero as $h \to 0$ since $d \leq 3$.
	\end{lemma}
	\begin{proof}
		It follows from \eqref{inExp} and \eqref{multi_tau} that the left-hand side of \eqref{res_L1} is bounded from above by $\tau^2$ times
		\begin{align*}
			 \sum_{|\bs{k}| \leq n} \left|\frac{|\langle R_{u_{\theta}},\phi_{\bs{k}} \rangle_{h}|^2}{1 + \lambda_{\bs{k}}^{(h)}} - \frac{|\langle R_{u_{\theta}},\phi_{\bs{k}} \rangle|^2}{1 + \lambda_{\bs{k}}}\right| + \sum_{|\bs{k}| > n} \frac{|\langle R_{u_{\theta}},\phi_{\bs{k}} \rangle|^2}{1 + \lambda_{\bs{k}}}.
		\end{align*}
		The second sum is bounded by
		\begin{align*}
			C \, \lVert R_{u_{\theta}} \rVert_{L^{2}}^{2} \, h^{2} \leq C \, \lVert R_{u_{\theta}} \rVert_{H^{2}}^{2} \, h^{2}
		\end{align*}
		since $(1 + \lambda_{\bs{k}})^{-1} \leq C \, h^{2}$ for $|\bs{k}| > n$. It remains to bound the first sum.
		
		Denoting $a = a^{(\bs{k})} \coloneqq \langle R_{u_{\theta}},\phi_{\bs{k}} \rangle$ and $a_h = a_h^{(\bs{k})} \coloneqq \langle R_{u_{\theta}},\phi_{\bs{k}} \rangle_{h}$, each element in the first sum above can be written as
		\begin{align*}
			&\left|\frac{(1 + \lambda_{\bs{k}}) (a_h^2 - a^2) + a^2 (\lambda_{\bs{k}} - \lambda_{\bs{k}}^{(h)})}{(1 + \lambda_{\bs{k}}^{(h)})(1 + \lambda_{\bs{k}})}\right|\\
			&\leq C \, \frac{(1 + \lambda_{\bs{k}}) (|a| \, |a - a_{h}| + |a_h||a - a_{h}|) + a^2 |\lambda_{\bs{k}} - \lambda_{\bs{k}}^{(h)}|}{(1 + \lambda_{\bs{k}})^2}\\
			&= C \, \left[\frac{|a|}{(1 + \lambda_{\bs{k}})} |a_{h} - a| + \frac{|a_{h}|}{(1 + \lambda_{\bs{k}})} |a_{h} - a| + \frac{a^2}{(1 + \lambda_{\bs{k}})^2} |\lambda_{\bs{k}} - \lambda_{\bs{k}}^{(h)}|\right]
		\end{align*}
		where the first inequality follows from the fact that $c\lambda_{\bs{k}} \leq \lambda_{\bs{k}}^{(h)} \leq C \lambda_{\bs{k}}$. 
		
		Summing over $\bs{k}$ the expression above, the first term of the sum is bounded by
		\begin{align*}
			&\sum_{|\bs{k}| \leq n} \frac{|\langle R_{u_{\theta}},\phi_{\bs{k}} \rangle|}{(1 + \lambda_{\bs{k}})} |\langle R_{u_{\theta}},\phi_{\bs{k}} \rangle_{h} - \langle R_{u_{\theta}},\phi_{\bs{k}} \rangle| \\
			&\leq \left(\sum_{|\bs{k}| \leq n} |\langle R_{u_{\theta}},\phi_{\bs{k}} \rangle|^{2}\right)^{1/2}\left(\sum_{|\bs{k}| \leq n} \frac{|\langle R_{u_{\theta}},\phi_{\bs{k}} \rangle_{h} - \langle R_{u_{\theta}},\phi_{\bs{k}} \rangle|^{2}}{(1 + \lambda_{\bs{k}})^{2}}\right)^{1/2} \\
			&\leq C \, h^{2} \, \lVert R_{u_{\theta}} \rVert_{L^{2}} \, \lVert R_{u_{\theta}} \rVert_{H^{2}}  \left(\sum_{|\bs{k}| \leq n} \frac{\lambda_{\bs{k}}^{2}}{(1 + \lambda_{\bs{k}})^{2}}\right)^{1/2} \leq C \, h^{2 - d/2} \, \lVert R_{u_{\theta}} \rVert_{H^{2}}^{2}
		\end{align*}
		where the first inequality is due to Cauchy-Schwarz and the second to Corollary \ref{corollary_norms} and Parseval's identity. Analogously, the sum of the second term is bounded by
		\begin{align*}
			&\sum_{|\bs{k}| \leq n} \frac{|\langle R_{u_{\theta}},\phi_{\bs{k}} \rangle_{h}|}{(1 + \lambda_{\bs{k}})} |\langle R_{u_{\theta}},\phi_{\bs{k}} \rangle_{h} - \langle R_{u_{\theta}},\phi_{\bs{k}} \rangle| \\
			&\leq \left(\sum_{|\bs{k}| \leq n} |\langle R_{u_{\theta}},\phi_{\bs{k}} \rangle_{h}|^{2}\right)^{1/2}\left(\sum_{|\bs{k}| \leq n} \frac{|\langle R_{u_{\theta}},\phi_{\bs{k}} \rangle_{h} - \langle R_{u_{\theta}},\phi_{\bs{k}} \rangle|^{2}}{(1 + \lambda_{\bs{k}})^{2}}\right)^{1/2} \\
			&\leq C \, h^{2 - d/2} \, \lVert R_{u_{\theta}} \rVert_{H^{2}} \left(\sum_{|\bs{k}| \leq n} |\langle R_{u_{\theta}},\phi_{\bs{k}} \rangle_{h}|^{2}\right)^{1/2}
		\end{align*}
		in which again the first inequality is due to Cauchy-Schwarz and the second to Corollary \ref{corollary_norms}. Since $\{\phi_{\bs{k}}\}$ is orthonormal under $\langle \cdot,\cdot \rangle_{h}$, and by Sobolev embedding, for $d \leq 3$, $H^{2}(\Omega)$ can be continuously embedded into $C^{0}(\bar{\Omega})$, the space of continuous functions equipped with the $L^\infty$ norm, we have
		\begin{align}
			\label{a1}
			\sum_{|\bs{k}| \leq n} |\langle R_{u_{\theta}},\phi_{\bs{k}} \rangle_{h}|^{2} = \lVert R_{u_{\theta}} \rVert_{h}^{2} \leq \lVert R_{u_{\theta}} \rVert_{L^\infty}^{2} \leq C \lVert R_{u_{\theta}} \rVert_{H^{2}}^{2}
		\end{align}
		and we conclude that the sum of the second term is bounded by $C \, h^{2 - d/2} \, \lVert R_{u_{\theta}} \rVert_{H^{2}}^{2}$. Finally, applying Lemma \ref{lemma_eigenvalue}, the sum of the third term is
		\begin{align*}
			\sum_{|\bs{k}| \leq n} \frac{|\langle R_{u_{\theta}},\phi_{\bs{k}} \rangle|^{2}}{(1 + \lambda_{\bs{k}})^{2}} |\lambda_{\bs{k}} - \lambda_{\bs{k}}^{(h)}| \leq C \, h^2 \, \sum_{|\bs{k}| \leq n} \frac{|\langle R_{u_{\theta}},\phi_{\bs{k}} \rangle|^{2}}{(1 + \lambda_{\bs{k}})^{2}} \lambda_{\bs{k}}^{2} \leq C \, h^2 \, \lVert R_{u_{\theta}} \rVert_{H^2}^{2}
		\end{align*}
		and the result follows.
	\end{proof}
	
	We now prove Theorem \ref{main_theorem}.
	
	\begin{proof}[Proof of Theorem \ref{main_theorem}]
		By the triangle inequality and Lemma \ref{L1},
		\begin{align*}
			\left|(nh)^{2d} \mcL_{\Phi^{(n)}}(\theta) - \lVert R_{u_{\theta}} \rVert_{\Phi}^{2}\right| \leq (nh)^{2d} \left|\mcL_{\Phi^{(n)}}(\theta) - \mu[\mcL_{\Phi^{(n)}}(\theta)]\right| + C \, h^{2 - d/2} \, \lVert R_{u_{\theta}} \rVert_{H^{2}}^{2}.
		\end{align*}
		By Cauchy-Schwarz inequality,
		\begin{align*}
			\mu\left[\left|\mcL_{\Phi^{(n)}}(\theta) - \mu[\mcL_{\Phi^{(n)}}(\theta)]\right|\right] \leq \left(\text{Var}(\mcL_{\Phi^{(n)}}(\theta))\right)^{1/2}
		\end{align*}
		in which, by an analogous computation to \eqref{inExp},
		\begin{align*}
			\text{Var}(\mcL_{\Phi^{(n)}}(\theta)) &= \frac{\tau^{4} \, h^{-4d}}{N \, n^{4d}} \text{Var}\left(\left[\sum_{|\bs{k}| \leq n} (1 + \lambda^{(h)}_{\bs{k}})^{-1/2} w_{\bs{k}} \langle R_{u_{\theta}},\phi_{\bs{k}} \rangle_{h}\right]^2\right).
		\end{align*}
		Observe that
		\begin{align*}
			Z_{n} \coloneqq \sum_{|\bs{k}| \leq n} (1 + \lambda^{(h)}_{\bs{k}})^{-1/2} w_{\bs{k}} \langle R_{u_{\theta}},\phi_{\bs{k}} \rangle_{h}
		\end{align*}
		follows a Gaussian distribution with mean zero and
		\begin{align*}
			\text{Var}(Z_{n}) = \sum_{|\bs{k}| \leq n} \frac{|\langle R_{u_{\theta}},\phi_{\bs{k}} \rangle_{h}|^{2}}{(1 + \lambda^{(h)}_{\bs{k}})} \leq \sum_{|\bs{k}| \leq n} |\langle R_{u_{\theta}},\phi_{\bs{k}} \rangle_{h}|^{2}  \leq C \, \lVert R_{u_{\theta}} \rVert_{H^{2}}^{2}
		\end{align*}
		in which the last inequality follows by the same arguments as in \eqref{a1}. Now, it holds for mean-zero Gaussian random variables
		\begin{align*}
			\text{Var}(Z_{n}^{2}) = 2 \, (\text{Var}(Z_{n}))^2 \leq C \, \lVert R_{u_{\theta}} \rVert_{H^{2}}^{4} 
		\end{align*}
		therefore
		\begin{align*}
			(nh)^{2d} \mu\left[\left|\mcL_{\Phi^{(n)}}(\theta) - \mu[\mcL_{\Phi^{(n)}}(\theta)]\right|\right] \leq C \, N^{-1/2} \, \lVert R_{u_{\theta}} \rVert_{H^{2}}^{2} 
		\end{align*}
		and the proof is complete.
	\end{proof}
	
	\begin{remark}
		The correcting factor $(nh)^{2d}$ appears since we are considering the DST-I to sample the test functions. If the truncated process were computed directly, this term would disappear. 	Although we focused on the hypercube domain $(0,1)^d$, analogous results could be obtained for other domains as long as there is a control over the quadrature error.
	\end{remark}
	
	\begin{remark}
		We emphasise that the $h^{2 - d/2}$ term is the trade-off between the quadrature rule, that is the Trapezoid rule, which is second order, and the irregularity of the test functions, for which the $H^{2}(\Omega)$ norm diverges as $h \to 0$. In particular, if the mode of truncation $n$ of the random test function is not selected as $\sim h^{-1}$, then a bound that depends on both quantities can be obtained analogously. See Appendix \ref{appAux} for more details. For higher dimensions ($d > 3$), higher-order quadrature rules need to be considered to guarantee convergence.
	\end{remark}
	
	\subsection{Convergence of SV-PINNs}
	
	Let $u_{\theta_{h}}$ be a minimiser\footnote{We assume a minimiser exists without loss of generality, as otherwise we could consider $u_{\theta_{h}}$ as satisfying $\mcL_{\Phi^{(n)}}(\theta_{h}) \leq \inf_{\theta} \mcL_{\Phi^{(n)}}(\theta) + \delta$ for $\delta > 0$ small enough.} of \eqref{loss_SVPINNs} with $\lambda = 0$ represented by a modified MLP with DAFF encoding using all eigenfunctions $\phi_{\bs{k}}$ with $|\bs{k}| \leq M$, $K$ hidden layers with $r$ nodes, and zero biases. A natural notion of convergence for SV-PINNs would be
	\begin{align}
	 	\label{converge_SV}
	 	\lim\limits_{M \to \infty} \lim\limits_{r \to \infty} \lim\limits_{h \to 0} \lim_{N \to \infty} \mu\left[\lVert u_{\theta_{h}} - u \rVert_{H_{0}^{1}}\right] = 0
	\end{align}
	for some fixed $K$. An analogous convergence result holding in high probability would also be desirable.
	
	One possible route towards proving convergence of SV-PINNs would be to establish concentration inequalities for the $\Phi$-norm in high probability
	\begin{align*}
		\mu\left(\sup_{\theta} \left|(nh)^{2d} \mcL_{\Phi^{(n)}}(\theta) - \lVert R_{u_{\theta}} \rVert_{\Phi}^{2}\right| > \epsilon\right) \leq \delta_{1}(\eps,h,N,r,M,K,L)
	\end{align*}
	or in expectation
	\begin{align*}
		\mu\left[\sup_{\theta} \left|(nh)^{2d} \mcL_{\Phi^{(n)}}(\theta) - \lVert R_{u_{\theta}} \rVert_{\Phi}^{2}\right|\right] \leq \delta_{2}(h,N,r,M,K,L)
	\end{align*}
	in which $\delta_{1}$ and $\delta_{2}$ converge to zero when the limit in \eqref{converge_SV} is applied. Observe that Theorem \ref{main_theorem} provides bounds in expectation, and analogous results could be obtained in high probability, for a fixed $\theta$. However, Theorem \ref{main_theorem} cannot be easily adapted for the supremum over $\theta$ inside the expectation.
	
	This kind of bound is not straightforward in this case and might not hold in general. In fact, it depends on the complexity of the space of functions generated by the modified MLP and on the structure of the \textit{empirical loss} $\mcL_{\Phi^{(n)}}(\cdot)$. Substantial research has been devoted to obtaining bounds for this kind of concentration inequality when the empirical loss is the mean squared error, and even in that case many open problems remain. See \cite{mendelson2015learning} for a discussion about the shortcomings of concentration inequalities. 
	
	Nevertheless, it is well-known that the rate of convergence of concentration probabilities and expectations such as those above depends on the geometry of the space of functions $u_{\theta}$ induced by the empirical loss $\mcL_{\Phi^{(n)}}(\theta)$, which can be characterised by \textit{complexity measures}. We refer to \cite{bartlett2002rademacher,concentration} for more details in the context of empirical processes. In our particular case, since for a fixed $\theta$ the empirical loss is an average of squared Gaussian random variables, the analysis might be more straightforward, especially if the space of trial functions is replaced by something more regular, for instance a convex set of functions. Talagrand's generic chaining method \cite{talagrand1996majorizing,talagrand2005generic}, which is a powerful tool to control the supremum of Gaussian random variables, and by extension other empirical processes, could be especially suitable to obtain concentration inequalities in this case. We leave the study of concentration inequalities for the $\Phi$-norm for future research as it is outside the scope of this paper. In the particular case of convex trial spaces, concentration inequalities could be bypassed entirely and the error could be controlled with the small ball method of \cite{mendelson2015learning}.
	
	In the case	of the inhomogeneous problem \eqref{inh_problem}, in which \eqref{loss_SVPINNs} is minimised with $\lambda > 0$, other analytical tools can be used, such as complexity-dependent error bounds analogous to those in \cite{marcondes2025complexity}, without relying on concentration inequalities. We also leave this analysis for future research.	
	
	\section{Experiments in $(0,1)^d$}
	\label{Sec_exp}
	
	In this section, we perform experiments to illustrate SV-PINNs and compare their performance with PINNs. We consider $6$ challenging boundary value problems associated with second-order linear elliptic operators in the hypercube domain $\Omega = (0,1)^d$: the Poisson equation with multi-scale high-frequency solutions in 1D (Section \ref{Sec_exp1}) and 2D (Section \ref{Sec_exp2}), the Poisson equation with variable high-frequency coefficients and high-frequency solution (Section \ref{Sec_exp3}), the Helmholtz equation with high-frequency solution in 2D (Section \ref{Sec_exp4}) and 3D (Section \ref{Sec_exp6}), and the inhomogeneous Poisson equation in 2D with a solution containing multiple Gaussian bumps (Section \ref{Sec_exp5}).
	
	The SV-PINNs are trained by gradient descent (GD) and \lbs for $5,000$ steps, and the PINNs are trained by GD also for $5,000$ steps. In all experiments with $d \leq 2$ and homogeneous boundary conditions we consider an architecture with $64 \times d$ DAFF computed by the $64 \times d$ eigenfunctions of the respective Dirichlet Laplacian with the smallest eigenvalues among those with indices $k_{1},\dots,k_{d} \leq 64$. In the 3D case we consider $12 \times d$ DAFF analogously.
	
	These DAFF are then fed to a $3$-layer modified MLP with $512$ nodes in each layer and \textit{tanh} activation function. In cases with non-zero boundary conditions, we consider Fourier features. Other architectures are also considered in the experiment in Section \ref{Sec_exp1}, where their specification can be found. More details about the hyperparameters considered in each case, such as the size of the collocation grid, the test grid and the sample size for computing the stochastically weak norm, are presented in Table \ref{tab_hyper}.
	
	For cases with zero boundary conditions we take $\tau = 0.1$ for $d = 1$, $\tau = 1$ for $d = 2$ and $\tau = 10$ for $d = 3$ in the stochastically weak norm so the initial value of the empirical loss lies in the range $10^{-2}-10^{1}$ for numerical stability. For non-zero boundary conditions, we choose $\tau$ so the two terms in the SV-PINNs loss are equal at initialisation. Letting $\theta_{0}$ denote the initial parameters of the neural network, we take $\lambda = 1$ and
	\begin{equation}
		\label{def_tau}
		\tau^{2} = \frac{\mcL_{b}(\theta_{0})}{\mcL_{1,\Phi^{(n)}}(\theta_{0})},
	\end{equation}
	in which $\mcL_{1,\Phi^{(n)}}$ is the loss $\mcL_{\Phi^{(n)}}$ computed when $\varphi_{i}$ is sampled by \eqref{DST_sol} with $\tau = 1$. Proceeding in this way, the two terms in the loss \eqref{loss_SVPINNs} are equal for $\theta = \theta_{0}$ as verified by substituting \eqref{DST_sol} into \eqref{loss_SVPINNs} with $\tau$ given by \eqref{def_tau}. This avoids loss imbalance at initialisation and should yield a more stable optimisation problem.
	
	Each case is trained three times with different random initialisation and sampling for stochastic norm calculation. The results are presented as the average over these three repetitions accompanied by the standard deviation in tables and standard error in plots. For each case we present the total training time and the $L^2$ relative error after $1,000$ and $5,000$ steps in tables, and for all steps in figures, defined as
	\begin{align}
		\label{L2}
		L^2 \text{ relative error of } u_{\theta} = \sqrt{\frac{\sum_{i = 1}^{N_{test}} |u_{\theta}(x^{(i)}_{t}) - u(x^{(i)}_{t})|^{2}}{\sum_{i = 1}^{N_{test}} |u(x^{(i)}_{t})|^{2}}}
	\end{align}
	where $u_{\theta}$ is the neural network with parameter $\theta$, $u$ is the solution of the respective boundary value problem and $x^{(i)}_{t}$ are test points in a grid of $(0,1)^{d}$ defined in Table \ref{tab_hyper}. We say that $u_{\theta}$ recovers the solution with $L^{2}$ relative error $< 1\%$ when \eqref{L2} is less than $0.01$. This threshold is marked as a dashed line in the plots of the $L^2$ relative error for a better comparison. We also present plots with the solution, the approximation by the SV-PINN trained by \lbs in one of the repetitions, and the pointwise error (true minus predicted) between them on the test grid.
	
	The experiments were performed in Python with \textit{jax} \cite{jax2018github} on an NVIDIA H200 141GB GPU with double precision\footnote{Except for the experiment in $3D$, where single precision was used instead.} and the code will be made available at \url{https://github.com/dmarcondes/JINNAX}.
	
	\subsection{Experiment 1: Multi-scale high-frequency Poisson equation in 1D}
	\label{Sec_exp1}
	
	As a first experiment, we consider the following Poisson equation in one dimension
	\begin{align}
		\label{ex1}
		\begin{cases}
			u_{xx}(x) = -4 \ \!  \pi^2 \sin(2 \pi x) - 0.1\ \!(a \pi)^{2} \sin(a \pi x), & \text{ for } x \in (0,1) \\
			u(0) = u(1) = 0
		\end{cases}
	\end{align}
	with the solution
	\begin{equation*}
		u(x) = \sin(2 \pi x) + 0.1 \sin(a \pi x)
	\end{equation*}
	for $a \in \{1,25,50,100,150,200,250,300\}$. The solution has two main frequencies, given respectively by the terms $\sin(2 \pi x)$ and $0.1 \sin(a \pi x)$, which are of different orders of magnitude for large values of $a$. This multi-scale behaviour makes solving \eqref{ex1} a challenging problem for PINNs.
	
	\begin{figure}[htbp]
		\centering
		\begin{subfigure}[b]{\textwidth}
			\centering
			\stackinset{l}{-20pt}{t}{5pt}{\textbf{(a)}}{\includegraphics[width=0.9\textwidth]{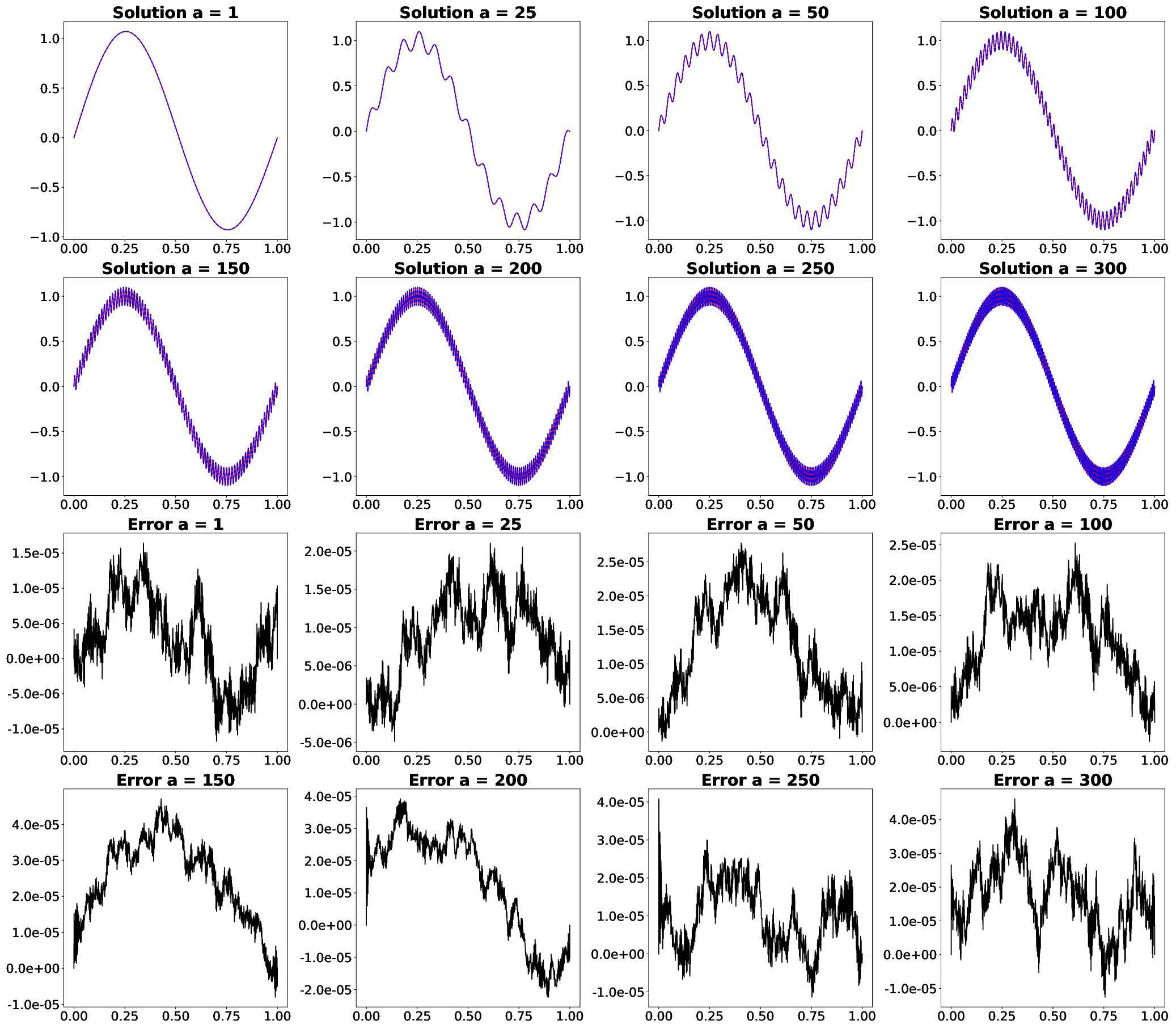}}
		\end{subfigure}
		
		\begin{subfigure}[b]{\textwidth}
			\centering
			\stackinset{l}{-20pt}{t}{5pt}{\textbf{(b)}}{\includegraphics[width=0.9\textwidth]{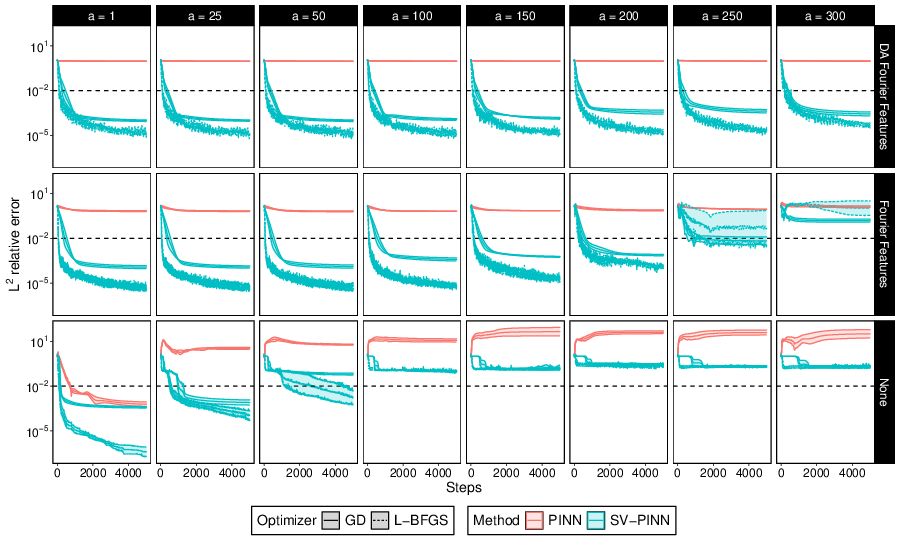}}
		\end{subfigure}
		
		\caption{\textbf{(a)} Solutions of the boundary value problems in Experiment 1 \eqref{ex1} (solid line in red) with the approximation by the SV-PINNs with DAFF trained with \lbs for $5,000$ steps (dashed line in blue), and the pointwise error between them. \textbf{(b)} Average $L^{2}$ relative error with $\pm$ one standard error over the $3$ repetitions at all training steps for each method, optimiser and architecture features for solving the boundary value problems \eqref{ex1}.} \label{fig_ex1}
	\end{figure}
	
	Apart from the PINNs and SV-PINNs with DAFF, we consider in each case the same architecture, 3 layers with 512 nodes, but with Fourier Features (FF) and with no features as detailed in Table \ref{tab_hyper}. This allows one to compare the performance of minimising the stochastically weak norm and the strong norm under different architectures. In the case of FF and no features, the boundary condition is enforced by a soft penalty to the loss function. In these cases, for PINNs and SV-PINNs trained by GD, we consider self-adaptive weights \cite{mcclenny2023self} with degree four polynomial mask. The results are presented in Figure \ref{fig_ex1} and Table \ref{tab_ex1}. We note that this problem was solved by PINNs with Fourier features for $a = 50$ in \cite{wang2021eigenvector} considering $40,000$ training steps and obtaining an $L^2$ relative error of $1.36 \times 10^{-3}$.
	
	When no features are considered, the problem cannot be properly solved either way for high values of $a$ ($a \geq 100$), which is expected due to the spectral bias of neural networks. For low values of $a$ (1 and 25) the solution was recovered by the SV-PINNs with a relative error $< 1\%$ training both by gradient descent and \lb, although the error was orders of magnitude lower with \lbs for $a = 1$. For $a = 50$, the solution was recovered with a relative error $< 1\%$ by the SV-PINNs trained with \lb. The PINNs could only recover the solution with no features for $a = 1$.
	
	The solution was recovered by the SV-PINNs with FF for all $a \leq 200$, with the result being in general better by \lbs. On the low end of $a$, the result of SV-PINNs with FF was better than with DAFF, but as $a$ increased their performance became comparable. For $a = 250,300$, the performance of SV-PINNs with FF decreased, and it could not recover the solution for any training method with $a = 300$, while that with DAFF remained with a relative error $< 1\%$. The performance of SV-PINNs with DAFF trained by GD and \lbs, respectively, were similar for all values of $a$. An interesting feature is that the pointwise error exhibits low regularity, visually resembling a function with fractional Sobolev regularity. PINNs could not recover the solution in any case with DAFF and FF.
	
	We see that, for SV-PINNs, while FF provides a better result for lower values of $a$, it is not as good as DAFF for extreme values, while DAFF recovers the solution very well for all values of $a$. We note that better results could be obtained in all cases, even with PINNs, by tailoring the architecture to the solution by, for example, increasing or decreasing the number of layers, nodes and frequencies of FF to fit simpler or more complex solutions. However, the point of this experiment was to outline that SV-PINNs with DAFF are robust to the hyperparameters of the architecture, and the solution can be properly recovered without fine-tuning for each problem.
	
	\subsection{Experiment 2: Multi-scale high-frequency Poisson equation in 2D}
	\label{Sec_exp2}
	
	For the second experiment, we consider again a Poisson equation with multi-scale high-frequency solution, but now in two dimensions:
	\begin{equation}
		\label{ex2}
		\begin{aligned}
			-\Delta u &= f, & & \text{ in } \Omega = (0,1)^2\\
				u &= 0, & & \text{ on } \partial \Omega		
		\end{aligned}
	\end{equation}
	where $f$ is chosen so that the solution of the boundary value problem is
	\begin{align*}
		u(x,y) = \sin(\pi x) \sin(\pi y)\left[\sin(a (x + y)) + \sin(2 \pi x) + \cos(3 \pi y) \right]
	\end{align*}
	for $a \in \{1,10,25,50,75\}$. The results of training SV-PINNs and PINNs with DAFF to solve this problem are presented in Figures \ref{fig_ex2_sol} and \ref{fig_ex2_res} and Table \ref{tab_ex2}.
	
	\begin{figure}[htbp]
		\centering
		\includegraphics[width=\linewidth]{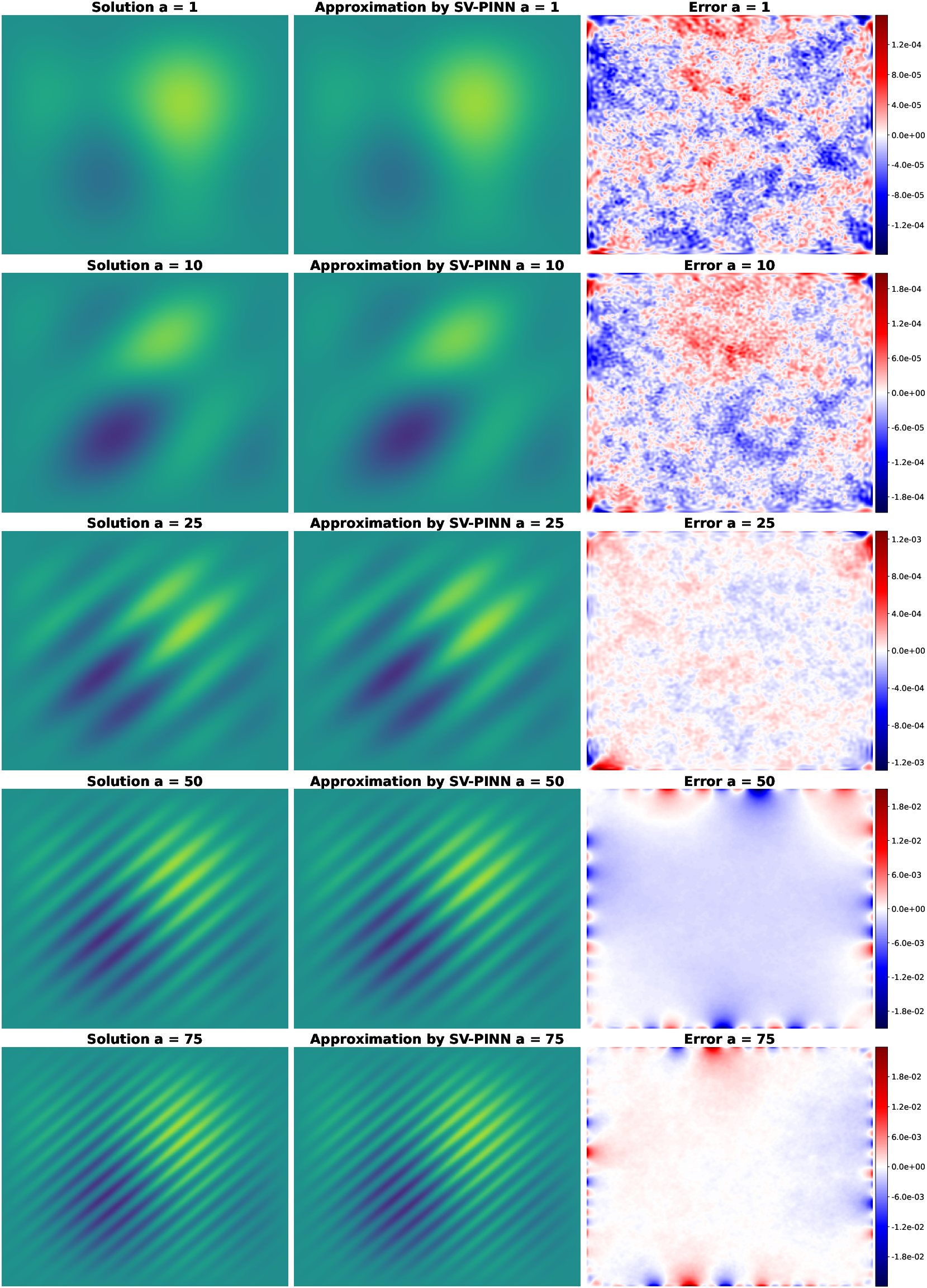}
		\caption{Solution of the boundary value problems in Experiment 2 \eqref{ex2}, the approximation by an SV-PINN trained with \lbs for $5,000$ steps and its pointwise error.}\label{fig_ex2_sol}
	\end{figure}
	
	\begin{figure}[htbp]
		\centering
		\includegraphics[width=\linewidth]{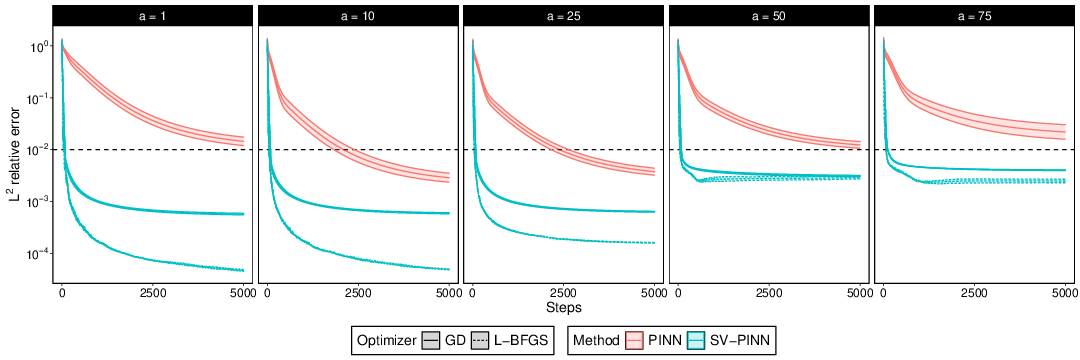}
		\caption{Average $L^{2}$ relative error with $\pm$ one standard error over the $3$ repetitions at all training steps for each method and optimiser for solving the boundary value problems in Experiment 2 \eqref{ex2}.} \label{fig_ex2_res}
	\end{figure}
	
	\begin{table}[htbp]
		\centering
		\caption{$L^2$ relative error and training time of solving the boundary value problems in Experiment 2 \eqref{ex2} with PINNs and SV-PINNs trained by gradient descent (GD) and \lb. The values are the average over $3$ repetitions after $1,000$ and $5,000$ training steps with the respective standard deviation in parentheses. The number of steps to achieve $L^2$-relative error $< 0.01$ is also presented for the cases in which all repetitions achieved this threshold in $5,000$ steps.}\label{tab_ex2}
		\resizebox{\linewidth}{!}{\begin{tabular}{lllllllc}
				\hline
				\multirow{2}{*}{a} & \multirow{2}{*}{Method} & \multirow{2}{*}{Optimiser} & \multicolumn{2}{c}{$1,000$ steps} & \multicolumn{2}{c}{$5,000$ steps} & Steps to \\
				\cline{4-7}& & & $L^2$ relative error & Time (m) &  $L^2$ relative error & Time (m) & $L^2$-RE $< 0.01$ \\
				\hline
				1 & PINN & GD & 1.745e-01 (5.195e-02) & 0.73 (0) & 1.496e-02 (4.392e-03) & 3.08 (0) &  \\ 
				\cline{2-8} & SV-PINN & GD & 9.986e-04 (5.419e-05) & 0.77 (0.01) & 5.750e-04 (3.635e-05) & 3.28 (0) & 87 (2.6) \\ 
				&  & L-BFGS & 1.735e-04 (1.016e-05) & 4.02 (0.01) & 4.707e-05 (3.795e-06) & 7.93 (0.02) & 48.7 (13.3) \\ 
				\hline 10 & PINN & GD & 3.767e-02 (1.689e-02) & 0.74 (0.03) & 3.005e-03 (1.100e-03) & 3.1 (0.04) & 2133.3 (511.8) \\ 
				\cline{2-8} & SV-PINN & GD & 1.011e-03 (4.951e-05) & 0.77 (0) & 5.955e-04 (2.634e-05) & 3.27 (0) & 82 (3.6) \\ 
				&  & L-BFGS & 1.735e-04 (2.362e-06) & 4.02 (0.02) & 4.966e-05 (1.590e-06) & 7.93 (0.02) & 39.3 (1.5) \\ 
				\hline 25 & PINN & GD & 4.201e-02 (1.347e-02) & 0.73 (0) & 3.873e-03 (1.116e-03) & 3.09 (0.01) & 2420.3 (442.5) \\ 
				\cline{2-8} & SV-PINN & GD & 1.031e-03 (4.888e-05) & 0.77 (0) & 6.375e-04 (2.292e-05) & 3.27 (0) & 81.7 (2.1) \\ 
				&  & L-BFGS & 2.879e-04 (5.703e-06) & 4 (0.01) & 1.621e-04 (1.798e-06) & 7.9 (0.01) & 38.7 (2.5) \\ 
				\hline 50 & PINN & GD & 6.479e-02 (1.990e-02) & 0.72 (0) & 1.264e-02 (3.114e-03) & 3.08 (0) &  \\ 
				\cline{2-8} & SV-PINN & GD & 3.807e-03 (2.111e-04) & 0.77 (0) & 3.116e-03 (1.653e-04) & 3.27 (0) & 87.3 (4.5) \\ 
				&  & L-BFGS & 2.718e-03 (4.052e-04) & 4.01 (0.01) & 2.926e-03 (3.201e-04) & 7.9 (0.02) & 44.3 (4.7) \\ 
				\hline 75 & PINN & GD & 8.602e-02 (3.591e-02) & 0.73 (0) & 2.431e-02 (1.315e-02) & 3.08 (0) &  \\ 
				\cline{2-8} & SV-PINN & GD & 4.638e-03 (6.903e-05) & 0.77 (0) & 4.024e-03 (1.198e-04) & 3.28 (0) & 110.3 (2.1) \\ 
				&  & L-BFGS & 2.473e-03 (4.923e-05) & 4.01 (0.01) & 2.509e-03 (3.346e-04) & 7.9 (0.01) & 66.3 (4.5) \\
				\hline
		\end{tabular}}
	\end{table}
	
	For all values of $a$, the SV-PINNs recovered the solution with an $L^2$ relative error $< 1\%$ after, on average, no more than $110$ steps, with \lbs achieving notably lower errors after $1,000$ steps: $1.735 \times 10^{-4}$, $1.735 \times 10^{-4}$, $2.879 \times 10^{-4}$, $2.718 \times 10^{-3}$ and $2.473 \times 10^{-3}$ for increasing values of $a$, respectively. Remarkably, training SV-PINNs with \lbs for $1,000$ steps achieved a better performance than training with GD for $5,000$ steps. PINNs were able to recover the solution with an error $< 1\%$ for $a = 10, 25$ after $5,000$ training steps. The performance of SV-PINNs was more stable, i.e., smaller standard deviations, over the three repetitions than that of PINNs.
	
	The SV-PINNs converged with significantly fewer steps than PINNs, attaining after $1,000$ steps an $L^2$ relative error at least one order of magnitude lower than that of PINNs, trained by both GD and \lb. This difference is especially striking for small values of $a$ (1, 10 and 25) in which SV-PINNs trained by \lbs reached an error approximately one order of magnitude lower than SV-PINNs trained by GD and between two and three orders of magnitude lower than that of PINNs. Although training by \lbs takes more time per step than GD, when comparing SV-PINNs trained by \lbs for $1,000$ steps and PINNs trained for $5,000$ steps, the former achieved in all cases an $L^2$ relative error at least one order of magnitude lower with only around 30\% more training time.
	
	We see in Figure \ref{fig_ex2_sol} that the SV-PINNs trained by \lbs recovered the solution well for all values of $a$. However, while for $a \leq 25$ the pointwise error is irregularly spread over the domain, for $a = 50, 75$ the solution was better recovered away from the boundary. 
	
	As in the first experiment, results could be improved by fine-tuning the architecture to each solution by adjusting, for instance, the number of DAFF, layers and nodes. Nevertheless, these experiments illustrate again how SV-PINNs are robust to the hyperparameters of the architecture over problems with varying complexities.
	
	\subsection{Experiment 3: Poisson equation with high-frequency coefficients}
	\label{Sec_exp3}
	
	In the third experiment, we consider a Poisson equation with rapidly oscillating variable coefficients and a high-frequency solution in two dimensions:
	\begin{equation}
		\label{ex3}
		\begin{aligned}
			-\nabla \cdot \!\left(a(x,y)\,\nabla u(x,y)\right) &= f(x,y), & &(x,y) \in \Omega = (0,1)^2 \\
			u(x,y) &= 0, & &(x,y) \in \partial \Omega
		\end{aligned}
	\end{equation}
	with 
	\begin{align*}
		a(x,y) = 1 + \beta \sin(k_a \pi x) \sin(k_a \pi y)
	\end{align*}
	and $f$ such that the solution is
	\begin{align*}
		u(x,y) = \sin(k_u \pi x)\sin(k_u \pi y)
	\end{align*}
	for $k_a = 20, k_u = 10$ and $\beta = 0.75$. Unlike Experiments 1 and 2, this problem involves rapidly oscillating coefficients, which significantly complicate the residual evaluation arising from the divergence-form structure of the operator and is known to be challenging for PINNs. The results of training SV-PINNs and PINNs with DAFF to solve this problem are presented in Figure \ref{fig_ex3} and Table \ref{tab_ex3}.
	
	\begin{table}[htbp]
		\centering
		\caption{$L^2$ relative error and training time of solving the boundary value problem in Experiment 3 \eqref{ex3} with PINNs and SV-PINNs trained by gradient descent (GD) and \lb. The values are the average over $3$ repetitions after $1,000$ and $5,000$ training steps with the respective standard deviation in parentheses. The number of steps to achieve $L^2$-relative error $< 0.01$ is also presented for the cases in which all repetitions achieved this threshold in $5,000$ steps.} \label{tab_ex3}
		\resizebox{\linewidth}{!}{\begin{tabular}{llllllc}
				\hline
				\multirow{2}{*}{Method} & \multirow{2}{*}{Optimiser} & \multicolumn{2}{c}{1,000 steps} & \multicolumn{2}{c}{5,000 steps} & Steps to \\ \cline{3-6}
				& & $L^2$ relative error & Time (m) & $L^2$ relative error & Time (m) & $L^2$-RE $< 0.01$ \\ 
				\hline
				PINN & GD & 6.844e-02 (2.379e-02) & 1.06 (0.02) & 1.946e-02 (7.183e-03) & 4.38 (0.04) &  \\ 
				\hline SV-PINN & GD & 1.734e-03 (5.887e-05) & 1.1 (0.01) & 1.007e-03 (2.827e-05) & 4.59 (0.03) & 108.7 (3.1) \\ 
				 & L-BFGS & 3.620e-04 (3.045e-05) & 4.92 (0.02) & 9.468e-05 (1.283e-05) & 9.89 (0.05) & 56 (2) \\ 
				\hline
		\end{tabular}}
	\end{table}
	
	\begin{figure}[htbp]
		\centering
		\begin{subfigure}[b]{\textwidth}
			\centering
			\stackinset{l}{-20pt}{t}{5pt}{\textbf{(a)}}{\includegraphics[width=\textwidth]{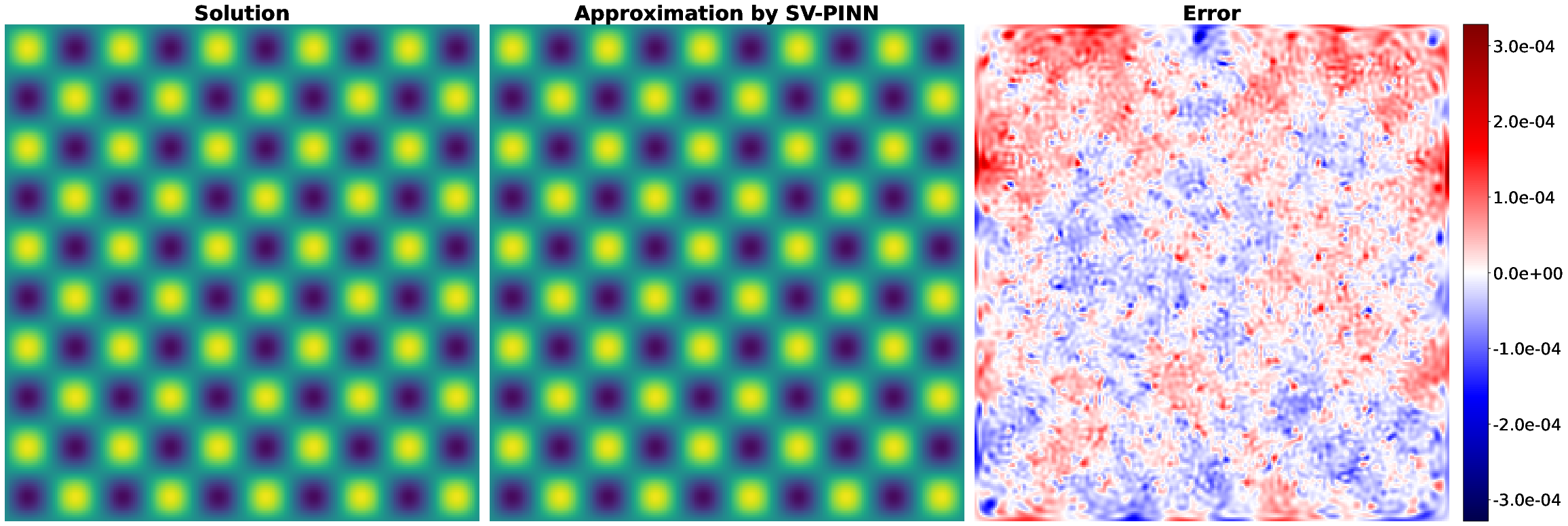}}
		\end{subfigure}
		
		\begin{subfigure}[b]{\textwidth}
			\centering
			\stackinset{l}{-20pt}{t}{5pt}{\textbf{(b)}}{\includegraphics[width=\textwidth]{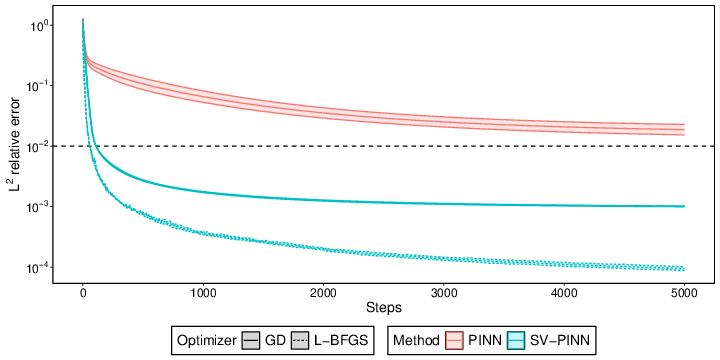}}
		\end{subfigure}
		
		\caption{\textbf{(a)} Solution of the boundary value problem in Experiment 3 \eqref{ex3}, the approximation by an SV-PINN trained with \lbs for $5,000$ steps and its pointwise error. \textbf{(b)} Average $L^{2}$ relative error with $\pm$ one standard error over the $3$ repetitions at all training steps for each method and optimiser.} \label{fig_ex3}
	\end{figure}
	
	The SV-PINNs trained by \lbs had an $L^2$ relative error after $1,000$ steps of $3.620 \times 10^{-4}$, just under half order of magnitude lower than that of SV-PINNs trained by GD for $5,000$ steps and between one and two orders of magnitude lower than that of PINNs. The training time of SV-PINNs trained by \lbs for $1,000$ steps was slightly greater than that of both PINNs and SV-PINNs trained by GD for $5,000$ steps.
	
	We see in Figure \ref{fig_ex3} that the convergence of SV-PINNs was also significantly faster in this case, especially when training by \lb, and the result is more stable over the repetitions. Furthermore, the solution has been well recovered by the SV-PINNs with the pointwise error irregularly scattered over the domain.
	
	\subsection{Experiment 4: Helmholtz equation in 2D}
	\label{Sec_exp4}
	
	To further increase the complexity, we consider the two-dimensional Helmholtz equation, which involves an indefinite operator and highly oscillatory solutions, making it substantially harder than the Poisson equations considered in the previous experiments:
	\begin{equation}
		\label{ex4}
		\begin{aligned}
			\Delta u + k^{2} u &= f, & & \text { in } \Omega = (0,1)^2 \\
			u &= 0, & & \text { on } \partial \Omega
		\end{aligned}
	\end{equation}
	in which $f$ is such that the solution is
	\begin{align*}
		u(x,y) = \sin(k \pi x)\sin(k \pi y)
	\end{align*}
	for $k \in \{5,10\}$. The results of training SV-PINNs and PINNs with DAFF to solve this problem are presented in Figure \ref{fig_ex4} and Table \ref{tab_ex4}.
	
	\begin{figure}[htbp]
		\centering
		\begin{subfigure}[b]{\textwidth}
			\centering
			\stackinset{l}{-20pt}{t}{5pt}{\textbf{(a)}}{\includegraphics[width=\textwidth]{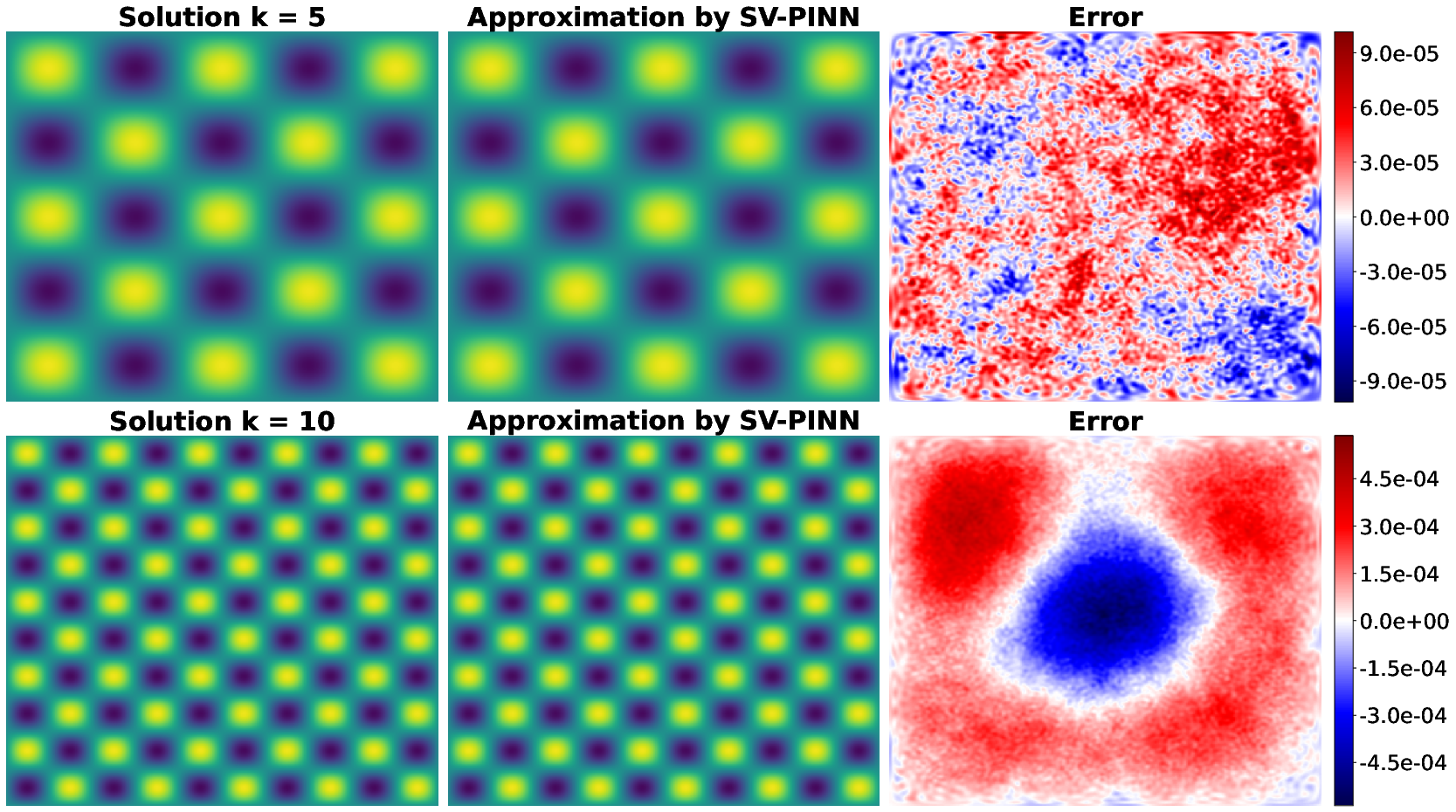}}
		\end{subfigure}
		
		\begin{subfigure}[b]{\textwidth}
			\centering
			\stackinset{l}{-20pt}{t}{5pt}{\textbf{(b)}}{\includegraphics[width=\textwidth]{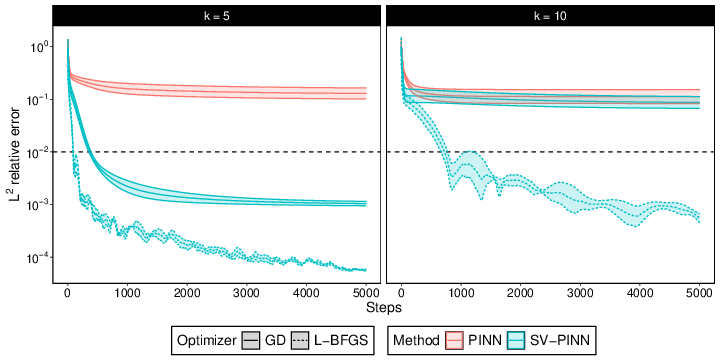}}
		\end{subfigure}
		
		\caption{\textbf{(a)} Solution of the boundary value problem in Experiment 4 \eqref{ex4}, the approximation by an SV-PINN trained with \lbs for $5,000$ steps and its pointwise error. \textbf{(b)} Average $L^{2}$ relative error with $\pm$ one standard error over the $3$ repetitions at all training steps for each method and optimiser.} \label{fig_ex4}
	\end{figure}
	
	\begin{table}[htbp]
		\centering
		\caption{$L^2$ relative error and training time of solving the boundary value problems in Experiment 4 \eqref{ex4} with PINNs and SV-PINNs trained by gradient descent (GD) and \lb. The values are the average over $3$ repetitions after $1,000$ and $5,000$ training steps with the respective standard deviation in parentheses. The number of steps to achieve $L^2$-relative error $< 0.01$ is also presented for the cases in which all repetitions achieved this threshold in $5,000$ steps.}\label{tab_ex4}
		\resizebox{\linewidth}{!}{\begin{tabular}{lllllllc}
				\hline
				\multirow{2}{*}{k} & \multirow{2}{*}{Method} & \multirow{2}{*}{Optimiser} & \multicolumn{2}{c}{$1,000$ steps} & \multicolumn{2}{c}{$5,000$ steps} & Steps to \\
				\cline{4-7}& & & $L^2$ relative error & Time (m) &  $L^2$ relative error & Time (m) & $L^2$-RE $< 0.01$ \\
				\hline
				5 & PINN & GD & 1.678e-01 (6.897e-02) & 0.99 (0.01) & 1.371e-01 (5.885e-02) & 4.31 (0.01) &  \\ 
				\cline{2-8} & SV-PINN & GD & 2.346e-03 (9.192e-04) & 1.05 (0) & 1.051e-03 (1.817e-04) & 4.61 (0) & 370.7 (33.5) \\ 
				&  & L-BFGS & 3.217e-04 (7.369e-05) & 5.48 (0.02) & 5.761e-05 (3.575e-06) & 10.47 (0.01) & 90.7 (6) \\ 
				\hline 10 & PINN & GD & 1.251e-01 (5.654e-02) & 0.99 (0) & 1.227e-01 (5.757e-02) & 4.31 (0) &  \\ 
				\cline{2-8} & SV-PINN & GD & 1.148e-01 (4.802e-02) & 1.05 (0) & 9.268e-02 (3.849e-02) & 4.61 (0) &  \\ 
				&  & L-BFGS & 6.602e-03 (4.677e-03) & 5.47 (0.01) & 5.731e-04 (1.861e-04) & 10.44 (0.02) & 886 (384) \\ 
				\hline
		\end{tabular}}
	\end{table}
	
	As in the previous experiments, the SV-PINNs trained with \lbs were able to recover the solution with an $L^2$ relative error $< 1\%$ in both cases, whereas gradient descent succeeded only for $k = 5$, although the training with \lbs converged significantly faster for $k = 5$: $L^2$ relative error of $3.217 \times 10^{-4}$ with \lbs versus $2.346 \times 10^{-3}$ with GD after $1,000$ steps. 
	
	However, the convergence of the $L^2$ relative error was less stable in this case compared with the first three experiments, especially after it reached lower values ($< 0.01$). This is due to the fact that this problem is harder to solve compared with the Poisson equation. For $k = 5$ the pointwise error was again irregularly scattered over the domain, but for $k = 10$ it was negative in the centre of the domain and positive closer to the boundary. The PINNs were not able to recover the solution.
	
	\subsection{Experiment 5: Gaussian bumps}
	\label{Sec_exp5}
		
	In the final two-dimensional experiment, we consider a Poisson equation in which the solution has inhomogeneous boundary conditions and is not built from sines and cosines:
	\begin{equation}
		\label{ex5}
		\begin{aligned}
			-\Delta u &= f, & & \text{ in } \Omega = (0,1)^2\\
			u &= g, & & \text{ on } \partial \Omega		
		\end{aligned}
	\end{equation}
	in which $f$ and $g$ are such that the solution of the boundary value problem is
	\begin{align*}
		u(x,y) = \sum_{i = 1}^{k} \sum_{j = 1}^{k} \exp\left(-50 \left[\left(x - \frac{i}{k + 1}\right)^{2} + \left(y - \frac{j}{k + 1}\right)^{2}\right]\right)
	\end{align*}
	for $k \in \{1,2,3,4,5\}$.
	
	The solution is formed by $k^2$ Gaussian bumps centred at points equally spaced in a square contained in $(0,1)^{2}$ as can be seen in Figure \ref{fig_ex5_sol}. In this case we consider an architecture with Fourier features, and the boundary condition is enforced by a soft penalty added to the loss function. The results of training SV-PINNs and PINNs to solve this problem are presented in Figures \ref{fig_ex5_sol} and \ref{fig_ex5_res} and Table \ref{tab_ex5}.
	
	\begin{figure}[htbp]
		\centering
		\includegraphics[width=\linewidth]{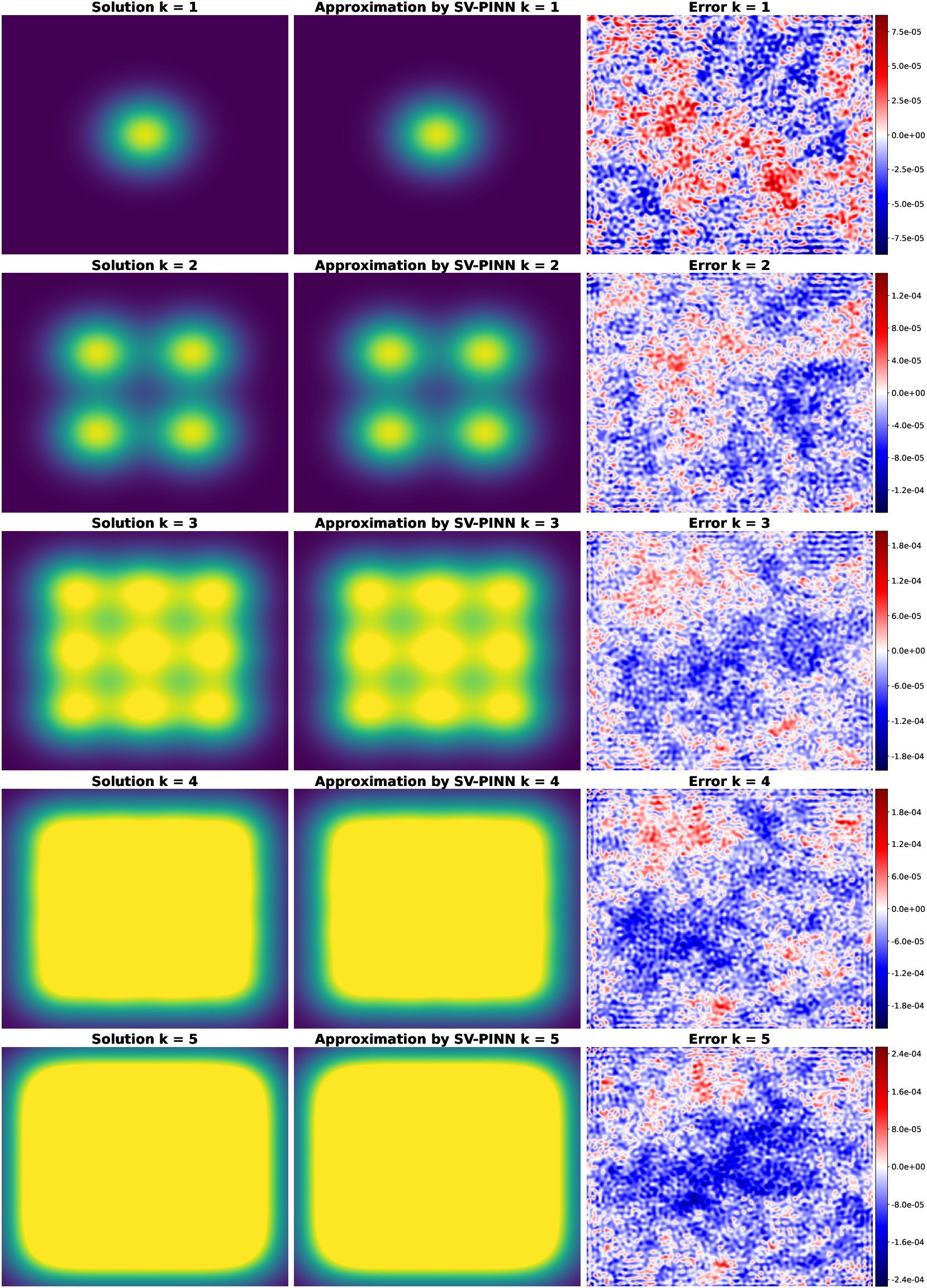}
		\caption{Solution of the boundary value problems in Experiment 5 \eqref{ex5}, the approximation by an SV-PINN trained with \lbs for $5,000$ steps and its pointwise error.}\label{fig_ex5_sol}
	\end{figure}
	
	\begin{figure}[htbp]
		\centering
		\includegraphics[width=\linewidth]{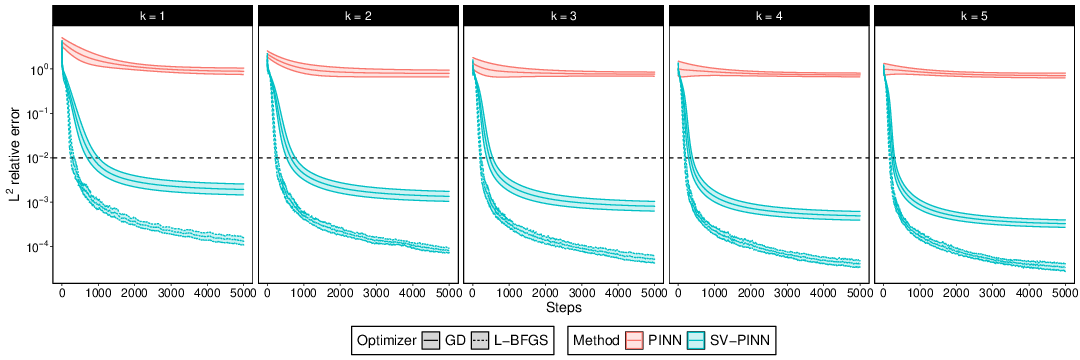}
		\caption{Average $L^{2}$ relative error with $\pm$ one standard error over the $3$ repetitions at all training steps for each method and optimiser for solving the boundary value problems in Experiment 5 \eqref{ex5}.} \label{fig_ex5_res}
	\end{figure}
	
	\begin{table}[htbp]
		\centering
		\caption{$L^2$ relative error and training time of solving the boundary value problems in Experiment 5 \eqref{ex5} with PINNs and SV-PINNs trained by gradient descent (GD) and \lb. The values are the average over $3$ repetitions after $1,000$ and $5,000$ training steps with the respective standard deviation in parentheses. The number of steps to achieve $L^2$-relative error $< 0.01$ is also presented for the cases in which all repetitions achieved this threshold in $5,000$ steps.} \label{tab_ex5}
		\resizebox{\linewidth}{!}{\begin{tabular}{lllllllc}
			\hline
			\multirow{2}{*}{k} & \multirow{2}{*}{Method} & \multirow{2}{*}{Optimiser} & \multicolumn{2}{c}{$1,000$ steps} & \multicolumn{2}{c}{$5,000$ steps} & Steps to \\
			\cline{4-7}& & & $L^2$ relative error & Time (m) &  $L^2$ relative error & Time (m) & $L^2$-RE $< 0.01$ \\
			\hline
			1 & PINN & GD & 1.657e+00 (6.915e-01) & 0.86 (0) & 9.098e-01 (2.623e-01) & 3.61 (0) &  \\ 
			\cline{2-8} & SV-PINN & GD & 8.029e-03 (5.204e-03) & 0.92 (0) & 2.111e-03 (1.128e-03) & 3.92 (0) & 859.3 (304.7) \\ 
			&  & L-BFGS & 9.185e-04 (3.666e-04) & 5.31 (0.01) & 1.363e-04 (4.461e-05) & 9.79 (0.04) & 290.3 (93.7) \\ 
			\hline 2 & PINN & GD & 1.090e+00 (4.792e-01) & 0.86 (0) & 8.142e-01 (2.326e-01) & 3.62 (0) &  \\ 
			\cline{2-8} & SV-PINN & GD & 4.874e-03 (2.708e-03) & 0.93 (0) & 1.451e-03 (6.891e-04) & 3.93 (0) & 650.7 (222.9) \\ 
			&  & L-BFGS & 5.051e-04 (1.450e-04) & 5.29 (0.02) & 8.401e-05 (1.699e-05) & 9.82 (0.02) & 251.3 (52.5) \\ 
			\hline 3 & PINN & GD & 9.211e-01 (4.003e-01) & 0.88 (0.01) & 7.723e-01 (1.475e-01) & 3.65 (0.01) &  \\ 
			\cline{2-8} & SV-PINN & GD & 2.683e-03 (1.500e-03) & 0.96 (0.03) & 8.659e-04 (3.940e-04) & 3.97 (0.03) & 465 (157) \\ 
			&  & L-BFGS & 3.437e-04 (1.130e-04) & 5.33 (0.01) & 5.436e-05 (1.730e-05) & 9.84 (0.02) & 215.7 (42.6) \\ 
			\hline 4 & PINN & GD & 8.862e-01 (2.776e-01) & 0.9 (0.01) & 7.404e-01 (1.313e-01) & 3.66 (0.01) &  \\ 
			\cline{2-8} & SV-PINN & GD & 1.472e-03 (7.480e-04) & 0.96 (0) & 5.203e-04 (2.131e-04) & 3.98 (0) & 337.3 (103.2) \\ 
			&  & L-BFGS & 2.349e-04 (7.610e-05) & 5.34 (0.02) & 4.176e-05 (1.119e-05) & 9.85 (0.03) & 186.3 (33.5) \\ 
			\hline 5 & PINN & GD & 8.729e-01 (2.175e-01) & 0.94 (0.01) & 7.190e-01 (1.520e-01) & 3.7 (0.01) &  \\ 
			\cline{2-8} & SV-PINN & GD & 9.304e-04 (3.962e-04) & 1.01 (0.02) & 3.416e-04 (1.178e-04) & 4.03 (0.02) & 262 (74.1) \\ 
			&  & L-BFGS & 1.768e-04 (4.885e-05) & 5.38 (0.04) & 3.501e-05 (1.073e-05) & 9.89 (0.04) & 169.7 (27.5) \\
			\hline
		\end{tabular}}
	\end{table}
	
	This is a problem that is known to be challenging for PINNs and could not be solved by them for any value of $k$ in this case. On the other hand, the SV-PINNs consistently solved this problem with the usual robustness over the values of $k$, with an $L^2$ relative error around one order of magnitude lower when training with \lbs compared with GD. This illustrates that SV-PINNs can have a superior performance even in inhomogeneous problems in which the boundary condition is enforced through a soft penalty to the loss.
	
	In contrast to DAFF, which did not attain satisfactory results in this case ($L^2$ relative error between $1\%$ and $5\%$ in preliminary experiments) and therefore were not considered, Fourier features are more sensitive to the frequency of the solution (see \cite{wang2021eigenvector} for more details). For all values of $k$, we considered Fourier features (Gaussian with standard deviation $5$), and we see that the performance of SV-PINNs improved with increasing $k$. This is likely because these frequencies are more compatible with the solution for higher values of $k$.
	
	As with the other examples, the convergence of the SV-PINNs trained by \lbs is significantly faster, achieving an $L^2$ relative error lower than $1\%$ in no more than $300$ steps on average for all values of $k$. We see in Figure \ref{fig_ex5_sol} an irregular pattern on the pointwise error, especially for lower values of $k$, indicating that the solution was recovered well by the SV-PINNs trained by \lbs for $5,000$ steps for all values of $k$. Again, the SV-PINNs were able to recover the solution of a challenging problem with varying complexities without fine-tuning the architecture, even in the presence of inhomogeneous boundary conditions.
	
	\subsection{Experiment 6: Helmholtz equation in $3D$}
	\label{Sec_exp6}
	
	In the last experiment in hypercube domains, we consider the Helmholtz equation in three dimensions
	\begin{equation}
		\label{ex6}
		\begin{aligned}
			\Delta u + k^{2} u &= f, & & \text { in } \Omega = (0,1)^3 \\
			u &= 0, & & \text { on } \partial \Omega
		\end{aligned}
	\end{equation}	
	in which $f$ is such that the solution to the boundary value problem is
	\begin{align*}
		u(x,y,z) = \sin(\pi x) \sin(2\pi y) \sin(3 \pi z) (x^2 + y^2  + z^2)
	\end{align*}
	for $k \in \{100,250,500\}$. The results of training SV-PINNs and PINNs to solve this problem are presented in Figures \ref{fig_ex6_sol} and \ref{fig_ex6_res} and Table \ref{tab_ex6}.
	
	In higher dimensions, memory constraints limit grid refinement, so a grid of only $32$ points on each coordinate was considered. Such a coarse discretisation makes it difficult to recover highly oscillatory solutions. To mitigate this, we adopt a relatively simple solution while considering large values of $k$. Furthermore, as illustrated in the experiments of \cite{calero2026enhancing}, the PINNs with DAFF perform particularly well in this case (simple solution, high values of $k$), providing a competitive baseline for SV-PINNs. This setting allows us to assess how the method scales beyond two dimensions and to demarcate potential limitations that are not apparent in lower-dimensional experiments.

	\begin{figure}[htbp]
		\centering
		\includegraphics[width=\linewidth]{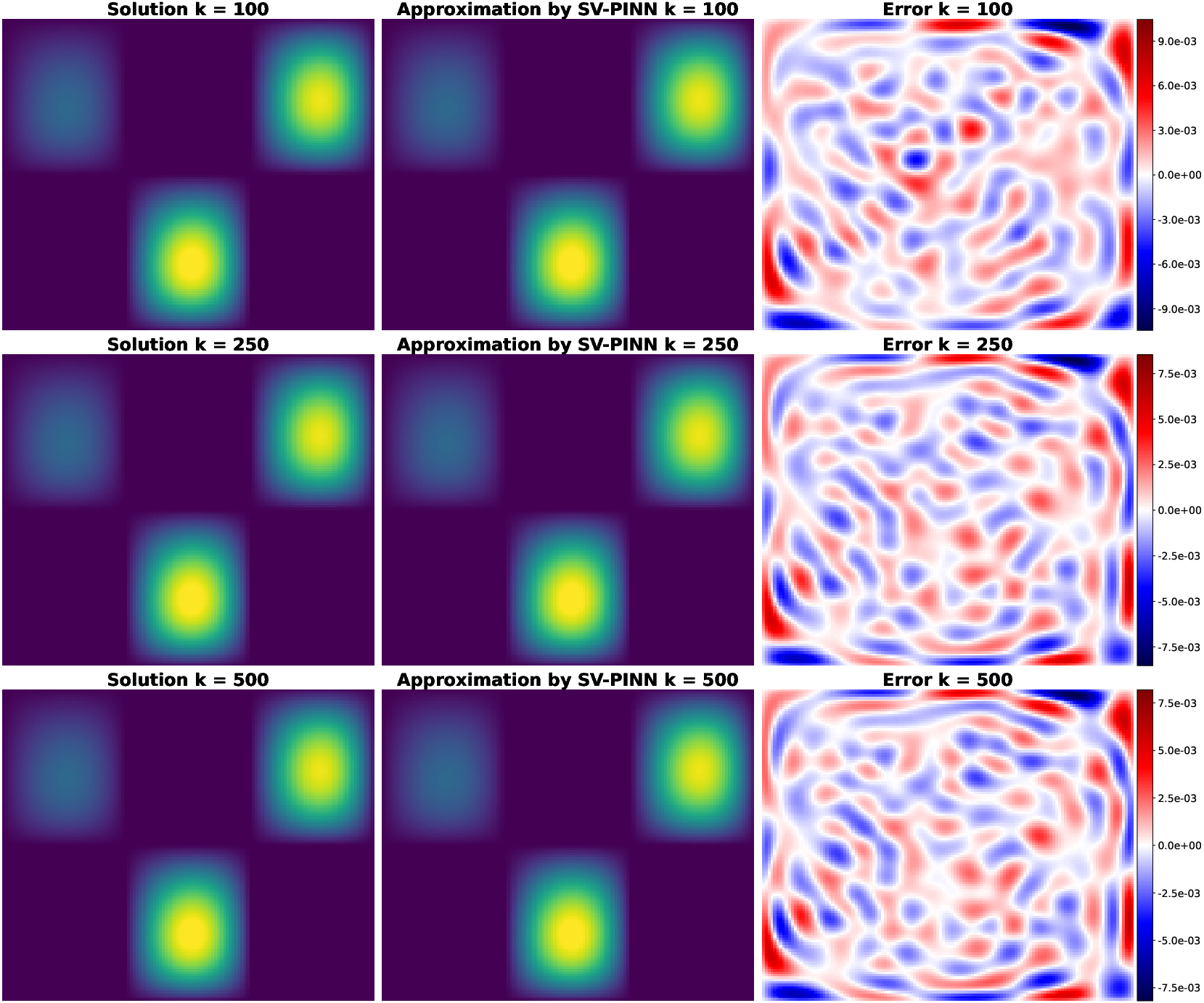}
		\caption{Solution of the boundary value problems in Experiment 6 \eqref{ex6}, the approximation by an SV-PINN trained with \lbs for $5,000$ steps and its pointwise error. We consider a slice of the respective functions with the first variable $x = 0.5$.}\label{fig_ex6_sol}
	\end{figure}
	
	\begin{figure}[htbp]
		\centering
		\includegraphics[width=\linewidth]{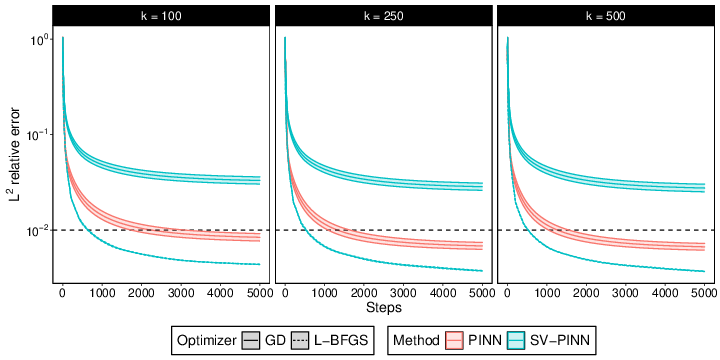}
		\caption{Average $L^{2}$ relative error with $\pm$ one standard error over the $3$ repetitions at all training steps for each method and optimiser for solving the boundary value problems in Experiment 6 \eqref{ex6}.} \label{fig_ex6_res}
	\end{figure}
	
	\begin{table}[htbp]
		\centering
		\caption{$L^2$ relative error and training time of solving the boundary value problems in Experiment 6 \eqref{ex6} with PINNs and SV-PINNs trained by gradient descent (GD) and \lb. The values are the average over $3$ repetitions after $1,000$ and $5,000$ training steps with the respective standard deviation in parentheses. The number of steps to achieve $L^2$-relative error $< 0.01$ is also presented for the cases in which all repetitions achieved this threshold in $5,000$ steps.} \label{tab_ex6}
		\resizebox{\linewidth}{!}{\begin{tabular}{lllllllc}
				\hline
				\multirow{2}{*}{k} & \multirow{2}{*}{Method} & \multirow{2}{*}{Optimiser} & \multicolumn{2}{c}{$1,000$ steps} & \multicolumn{2}{c}{$5,000$ steps} & Steps to \\
				\cline{4-7}& & & $L^2$ relative error & Time (m) &  $L^2$ relative error & Time (m) & $L^2$-RE $< 0.01$ \\
				\hline
				100 & PINN & GD & 1.443e-02 (2.615e-03) & 1.24 (0.05) & 8.465e-03 (1.350e-03) & 5.28 (0.04) &  \\ 
				\cline{2-8} & SV-PINN & GD & 4.881e-02 (6.986e-03) & 1.26 (0.01) & 3.333e-02 (5.186e-03) & 5.46 (0.01) &  \\ 
				&  & L-BFGS & 7.701e-03 (1.479e-04) & 6.35 (0.07) & 4.370e-03 (2.516e-05) & 11.98 (0.05) & 642 (21.2) \\ 
				\hline 250 & PINN & GD & 1.171e-02 (2.166e-03) & 1.22 (0) & 6.872e-03 (1.031e-03) & 5.26 (0.01) & 1418.3 (569.5) \\ 
				\cline{2-8} & SV-PINN & GD & 4.370e-02 (6.496e-03) & 1.26 (0.01) & 2.868e-02 (4.514e-03) & 5.46 (0.01) &  \\ 
				&  & L-BFGS & 6.755e-03 (1.182e-04) & 6.31 (0.01) & 3.749e-03 (4.407e-05) & 11.93 (0.03) & 531.7 (18) \\ 
				\hline 500 & PINN & GD & 1.143e-02 (2.084e-03) & 1.24 (0.05) & 6.723e-03 (9.803e-04) & 5.27 (0.04) & 1342.7 (516.8) \\ 
				\cline{2-8} & SV-PINN & GD & 4.278e-02 (6.600e-03) & 1.3 (0.07) & 2.782e-02 (4.600e-03) & 5.49 (0.06) &  \\ 
				&  & L-BFGS & 6.684e-03 (5.667e-05) & 6.33 (0.01) & 3.699e-03 (3.682e-05) & 11.96 (0.03) & 524 (4.6) \\
				\hline
		\end{tabular}}
	\end{table}
	
	As expected, the PINNs perform particularly well, especially for higher values of $k$, attaining $L^2$ relative errors on average between $6.7 \times 10^{-3}$ and $8.5 \times 10^{-3}$, outperforming the SV-PINNs trained with GD (on average no less than $2.7 \times 10^{-2}$ in all cases). Nevertheless, the SV-PINNs trained with \lbs converged significantly faster to $L^2$ relative errors less than $1\%$, on average after around 500-650 steps, attaining relative errors roughly half those of PINNs. We see in Figure \ref{fig_ex6_sol} that the pointwise error has a structured pattern, unlike the irregular scatter observed in previous experiments.
	
	This higher dimensional case remains challenging, and may require a more tailored choice of the DAFF and a finer grid to tackle more complex problems, indicating that the SV-PINNs are not immune to the curse of dimensionality, as can also be concluded from Theorem \ref{main_theorem}.
	
	We note that for lower values of $k < 100$, the performance of the methods in preliminary experiments was not as good, and SV-PINNs could not recover the solution with an $L^2$ relative error below $1\%$, but only below $5\%$ without tailoring the architecture (e.g., decreasing the number of DAFF). Interestingly, we see in Table \ref{tab_ex6} that a better result is obtained as $k$ increases, a counter-intuitive phenomenon that may be related to numerical effects or optimisation dynamics, warranting further investigation. Slightly more complex solutions could be recovered with $L^2$ relative errors below $5\%$ by tailoring the architecture. Overall, the robustness observed in two dimensions is not replicated in three dimensions, and a systematic study of SV-PINNs in 3D is deferred to future work.
	
	\section{Experiment on the circular domain}
	\label{Sec_exp_circle}
	
	As a further increase in complexity, we study the Helmholtz equation on a circular domain of radius one, in which both the indefinite operator and the domain geometry pose significant challenges for PINN-based methods:
	\begin{equation}
		\label{ex7}
		\begin{aligned}
			\Delta u + k^{2} u &= f, & & \text { in } \Omega = \{(x,y) \in \mbR^{2}: x^{2} + y^{2} < 1\} \\
			u &= 0, & & \text { on } \partial \Omega
		\end{aligned}
	\end{equation}
	for $k \in \{1,3,5\}$, in which $f$ is such that the solution to the boundary value problem is
	\begin{align*}
		u(x,y) = \phi_{0,1,0}(r(x,y),\theta(x,y))  \, \sin(k \pi x/2) \, \sin(k \pi y/2)
	\end{align*}
	where $\phi_{0,1,0}(r,\theta) = C_{0,1,0} J_{0}(\sqrt{\lambda_{0,1}}r)$ is the respective eigenfunction associated with the eigenvalue $\lambda_{0,1}$ (cf. \eqref{eigen_circle}). This choice combines radial eigenmodes of the circle with Cartesian oscillations, yielding a solution that is not separable in polar or Cartesian coordinates. Since $\phi_{0,1,0}|_{\partial\Omega} = 0$, the solution $u$ indeed satisfies the zero boundary condition.
	
	We consider an architecture with $64$ DAFF, given by sine and cosine eigenfunctions associated with the $32$ least eigenvalues among $\lambda_{n,m}$ with $n < 150$ and $m \leq 150$, and three hidden layers with $512$ nodes. We take $\tau = 0.1$ for numerical stability, $22,500$ collocation points sampled uniformly and apply the same training methods as in the previous section. More details about the implementation are in Appendix \ref{app_details}. The results of training SV-PINNs and PINNs with DAFF to solve this problem are presented in Figures \ref{fig_ex7_sol} and \ref{fig_ex7_res}, and Table \ref{tab_ex7}.
	
	For all values of $k$, the SV-PINNs trained with \lbs achieved a relative error between half and one order of magnitude lower than both PINNs and the SV-PINNs trained with GD after $5,000$ steps. Furthermore, for $k = 3,5$ the SV-PINNs trained by \lbs for $1,000$ steps achieved a significantly lower relative error than the PINNs trained for $5,000$ steps, taking around $35\%$ less time: $2.019 \times 10^{-3}$ versus $3.112  \times 10^{-3}$ for $k = 3$ and $5.036  \times 10^{-3}$ versus $3.522  \times 10^{-2}$ for $k = 5$. The PINN algorithm performed relatively well in this experiment, obtaining a better result than SV-PINNs trained by GD for $k = 1,3$. We note that the convergence of SV-PINNs trained by \lbs for all values of $k$ showed a large variation across repetitions, as reflected by the standard deviations of 183 to 312 steps around the respective mean number of steps to achieve $L^2$ relative error below $1\%$.
	
	The SV-PINNs trained by \lbs achieved a relative error lower than $1\%$ significantly faster than the other methods for $k = 1,3$, and it was the only method to achieve this mark for $k = 5$. We see in Figure \ref{fig_ex7_sol} that, for $k = 1,3$, the pointwise error of the SV-PINNs trained by \lbs is scattered over the domain without a clear pattern, while in the more challenging case $k = 5$ a clear error pattern emerges. The results on the circular domain demonstrate that SV-PINNs continue to outperform PINNs even after a substantial change in geometry, indicating that the benefits of SV-PINNs are not restricted to hypercube domains.
	
	\begin{figure}[htbp]
		\centering
		\includegraphics[width=\linewidth]{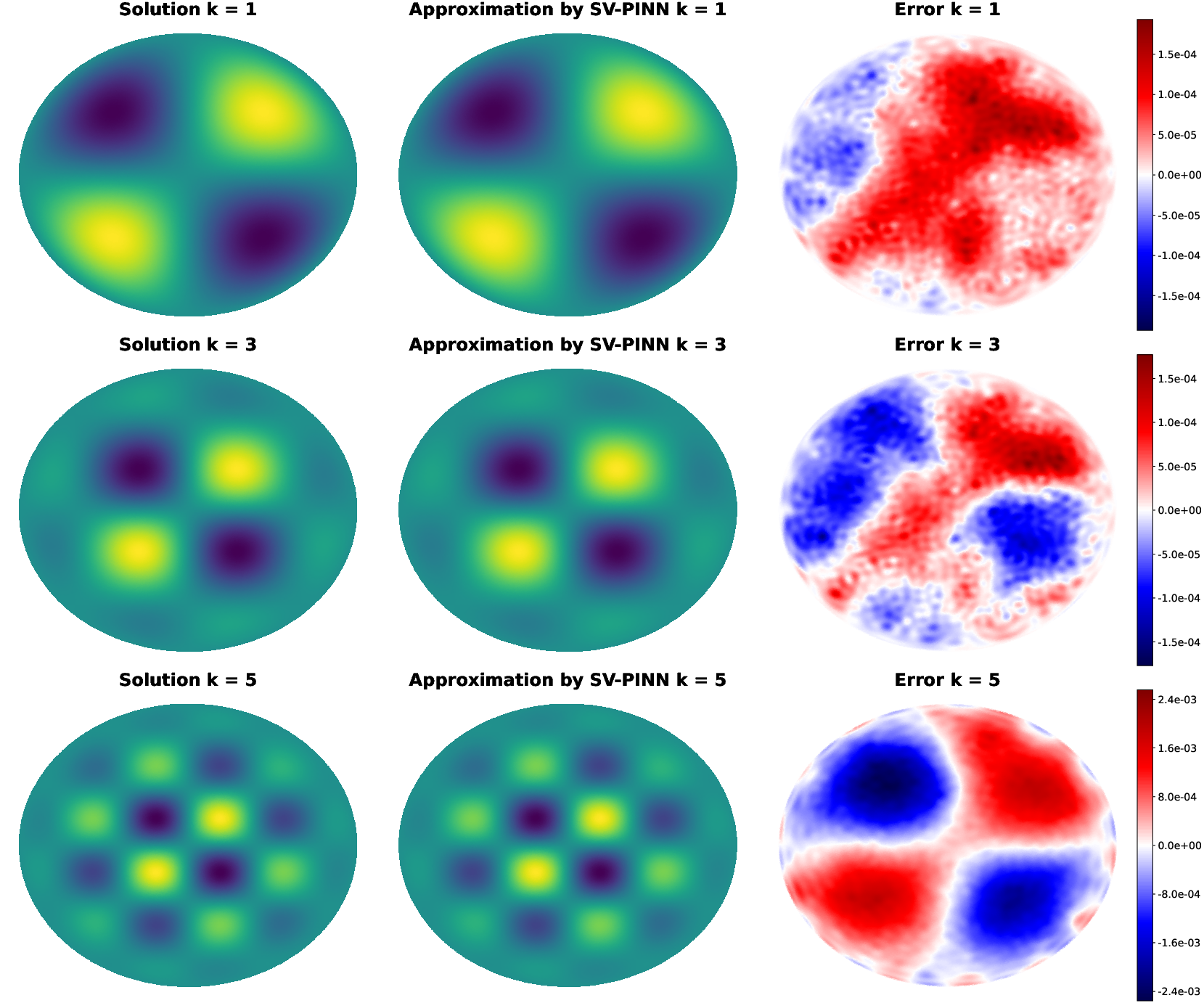}
		\caption{Solution of the boundary value problems in Experiment 7 \eqref{ex7}, the approximation by an SV-PINN trained with \lbs for $5,000$ steps and its pointwise error.}\label{fig_ex7_sol}
	\end{figure}
	
	\begin{figure}[htbp]
		\centering
		\includegraphics[width=\linewidth]{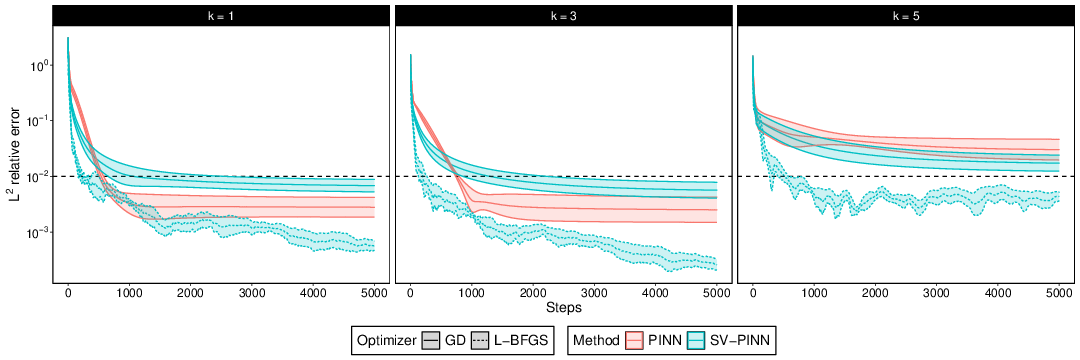}
		\caption{Average $L^{2}$ relative error with $\pm$ one standard error over the $3$ repetitions at all training steps for each method and optimiser for solving the boundary value problems in Experiment 7 \eqref{ex7}.} \label{fig_ex7_res}
	\end{figure}
	
	\begin{table}[htbp]
		\centering
		\caption{$L^2$ relative error and training time of solving the boundary value problems in Experiment 7 \eqref{ex7} with PINNs and SV-PINNs trained by gradient descent (GD) and \lb. The values are the average over $3$ repetitions after $1,000$ and $5,000$ training steps with the respective standard deviation in parentheses. The number of steps to achieve $L^2$-relative error $< 0.01$ is also presented for the cases in which all repetitions achieved this threshold in $5,000$ steps.}\label{tab_ex7}
		\resizebox{\linewidth}{!}{\begin{tabular}{lllllllc}
				\hline
				\multirow{2}{*}{k} & \multirow{2}{*}{Method} & \multirow{2}{*}{Optimiser} & \multicolumn{2}{c}{$1,000$ steps} & \multicolumn{2}{c}{$5,000$ steps} & Steps to \\
				\cline{4-7}& & & $L^2$ relative error & Time (m) &  $L^2$ relative error & Time (m) & $L^2$-RE $< 0.01$ \\
				\hline
				1 & PINN & GD & 3.816e-03 (2.351e-03) & 1.83 (0.01) & 3.203e-03 (1.704e-03) & 8.8 (0.07) & 547.3 (53.2) \\ 
				\cline{2-8} & SV-PINN & GD & 1.199e-02 (7.404e-03) & 2.63 (0.02) & 7.294e-03 (3.405e-03) & 9.73 (0.05) &  \\ 
				&  & L-BFGS & 3.460e-03 (7.596e-04) & 5.7 (0.09) & 5.930e-04 (1.998e-04) & 17.06 (0.06) & 317.3 (254.7) \\ 
				\hline 3 & PINN & GD & 4.517e-03 (2.985e-03) & 1.84 (0) & 3.112e-03 (1.928e-03) & 8.82 (0.03) & 769.3 (40.1) \\ 
				\cline{2-8} & SV-PINN & GD & 1.294e-02 (6.272e-03) & 2.63 (0.02) & 6.336e-03 (3.929e-03) & 9.74 (0.04) &  \\ 
				&  & L-BFGS & 2.019e-03 (3.436e-04) & 5.68 (0.04) & 2.700e-04 (1.062e-04) & 17.04 (0.06) & 235 (183.8) \\ 
				\hline 5 & PINN & GD & 5.705e-02 (3.116e-02) & 1.84 (0.01) & 3.522e-02 (1.949e-02) & 8.81 (0.05) &  \\ 
				\cline{2-8} & SV-PINN & GD & 4.168e-02 (2.555e-02) & 2.63 (0.01) & 1.917e-02 (1.025e-02) & 9.75 (0.02) &  \\ 
				&  & L-BFGS & 5.036e-03 (2.285e-03) & 5.71 (0.01) & 4.515e-03 (1.493e-03) & 17.09 (0.04) & 530.3 (312.6) \\
				\hline
		\end{tabular}}
	\end{table}
	
	\section{Experiment in L-shaped domain}
	\label{Sec_exp_L}
	
	As a final experiment, we consider the Helmholtz equation in the L-shaped domain
	\begin{equation}
		\label{ex8}
		\begin{aligned}
			\Delta u + k^{2} u &= f, & & \text { in } \Omega = (-1,1)^{2}\setminus(0,1)^{2} \\
			u &= 0, & & \text { on } \partial \Omega
		\end{aligned}
	\end{equation}
	for $k \in \{1, 5,10\}$, in which $f$ is such that the solution to the boundary value problem is
	\begin{align}
		\label{sol8}
		u(x,y) = \phi_{1}(x,y) \, \phi_{10}(x,y) \, \sin(2 \pi x) \, \sin(3 \pi y)
	\end{align}
	where $\phi_{1}$ and $\phi_{10}$ are the eigenfunctions associated with the first and $10$-th least eigenvalue of the Dirichlet Laplacian (cf. \eqref{DLoper}). Since $\phi_{1}|_{\partial\Omega} = \phi_{10}|_{\partial\Omega} = 0$, $u$ satisfies the zero boundary condition.
	
	We consider an architecture with $64$ DAFF and one hidden layer with $1,024$ nodes, $\tau = 0.1$ for numerical stability, the collocation points in a grid with $7,105$ points, and the same training methods as in the previous sections. In order to compute the DAFF, we interpolate the values of the eigenfunctions $\hat{\phi}_{k}(x_{c}^{(i)})$ in the grid computed with FEM (cf. Section \ref{SecL}) with radial basis functions so they are $C^2$. The same interpolation is used to compute the solution \eqref{sol8} and $f$. More details about the implementation are in Appendix \ref{app_details}. The results of training SV-PINNs and PINNs with DAFF to solve this problem are presented in Figures \ref{fig_ex8_sol} and \ref{fig_ex8_res}, and Table \ref{tab_ex8}.
	
	For $k = 1$, although both SV-PINNs, with \lbs and GD, and PINNs attained comparable performance after $5,000$ steps, around $2.5 \times 10^{-3}$, the SV-PINNs converged in significantly fewer steps. The SV-PINNs with GD and \lbs took on average $150$ and $85$ steps, respectively, to attain a relative error $< 1\%$, while the PINNs took 641 steps. We see that SV-PINNs trained by \lbs after $1,000$ steps attained a relative error around $17\%$ lower than that of PINNs after $5,000$ steps with a running time only $6\%$ greater.
	
	For $k = 5$, the SV-PINNs converged significantly faster, attaining a relative error $< 1\%$ after 126 steps with \lbs and 242 steps with GD, on average, with comparable errors with both \lbs and GD after $1,000$ and $5,000$ steps. Comparing SV-PINNs after $1,000$ steps of \lbs with PINNs after $5,000$ steps, we see that the former achieved a relative error $33\%$ lower again with only a $6\%$ greater running time.
	
	Finally, for $k = 10$, on average the relative error of SV-PINNs trained with \lbs was marginally above $1\%$ after $1,000$ ($1.043  \times 10^{-2}$) and $5,000$ ($1.073 \times 10^{-2}$) steps, the former being around $37\%$ lower than the error of PINNs after $5,000$ steps, with comparable running time. In this case, the instability of SV-PINNs, represented by the standard deviation of the relative error, was greater than for $k = 1,5$, underscoring the higher difficulty of this problem.
	
	Although PINNs with DAFF achieved better performance in this example compared with the previous ones, the performance of SV-PINNs trained by \lbs within $5,000$ steps was still superior, and they converged in significantly fewer steps. We see in Figure \ref{fig_ex8_sol} that, as $k$ increases, more patterns emerge in the pointwise error, demonstrating the difficulty of the SV-PINNs in recovering some features of the solution, even though the $L^2$ relative error remains close to $1\%$. 
	
	\begin{figure}[htbp]
		\centering
		\includegraphics[width=\linewidth]{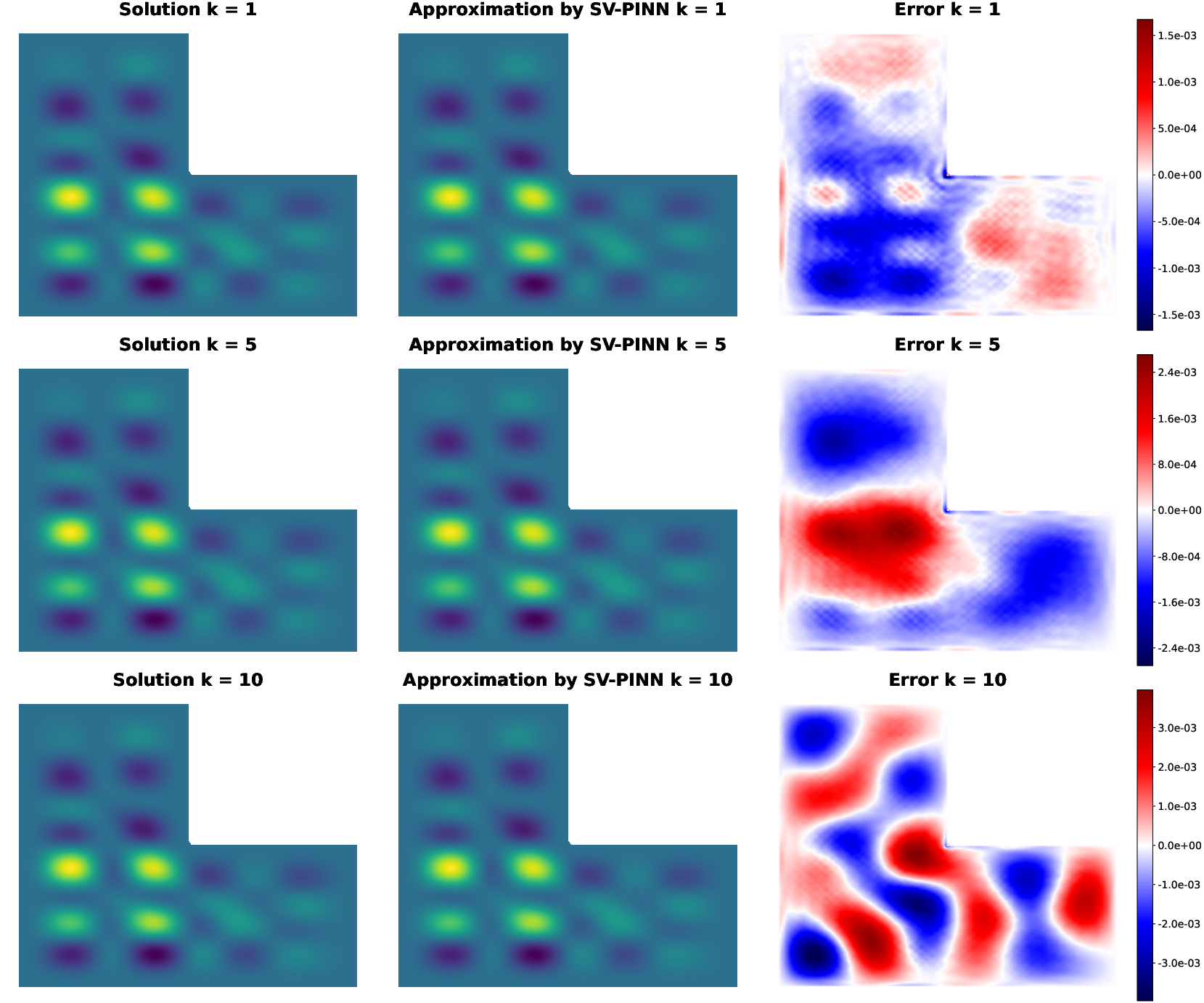}
		\caption{Solution of the boundary value problems in Experiment 8 \eqref{ex8}, the approximation by an SV-PINN trained with \lbs for $5,000$ steps and its pointwise error.}\label{fig_ex8_sol}
	\end{figure}
	
	\begin{figure}[htbp]
		\centering
		\includegraphics[width=\linewidth]{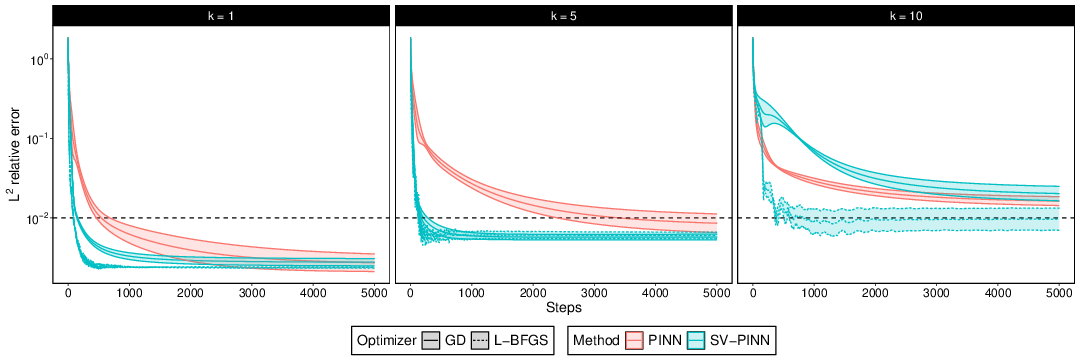}
		\caption{Average $L^{2}$ relative error with $\pm$ one standard error over the $3$ repetitions at all training steps for each method and optimiser for solving the boundary value problems in Experiment 8 \eqref{ex8}.} \label{fig_ex8_res}
	\end{figure}
	
	\begin{table}[htbp]
		\centering
		\caption{$L^2$ relative error and training time of solving the boundary value problems in Experiment 8 \eqref{ex8} with PINNs and SV-PINNs trained by gradient descent (GD) and \lb. The values are the average over $3$ repetitions after $1,000$ and $5,000$ training steps with the respective standard deviation in parentheses. The number of steps to achieve $L^2$-relative error $< 0.01$ is also presented for the cases in which all repetitions achieved this threshold in $5,000$ steps.}\label{tab_ex8}
		\resizebox{\linewidth}{!}{\begin{tabular}{lllllllc}
				\hline
				\multirow{2}{*}{k} & \multirow{2}{*}{Method} & \multirow{2}{*}{Optimiser} & \multicolumn{2}{c}{$1,000$ steps} & \multicolumn{2}{c}{$5,000$ steps} & Steps to \\
				\cline{4-7}& & & $L^2$ relative error & Time (m) &  $L^2$ relative error & Time (m) & $L^2$-RE $< 0.01$ \\
				\hline
				1 & PINN & GD & 6.650e-03 (2.919e-03) & 0.42 (0.01) & 2.918e-03 (1.317e-03) & 1.82 (0.01) & 641.3 (284) \\ 
				\cline{2-8} & SV-PINN & GD & 3.337e-03 (5.460e-04) & 0.76 (0) & 2.807e-03 (5.075e-04) & 2.22 (0) & 150.3 (14.7) \\ 
				&  & L-BFGS & 2.425e-03 (1.248e-05) & 1.94 (0) & 2.385e-03 (7.436e-05) & 4.51 (0) & 85.7 (3.1) \\ 
				\hline 5 & PINN & GD & 2.737e-02 (5.810e-03) & 0.42 (0) & 9.152e-03 (3.606e-03) & 1.82 (0.01) &  \\ 
				\cline{2-8} & SV-PINN & GD & 6.006e-03 (1.028e-03) & 0.76 (0) & 5.790e-03 (8.319e-04) & 2.22 (0.01) & 242.7 (123.3) \\ 
				&  & L-BFGS & 6.096e-03 (1.025e-03) & 1.93 (0.01) & 6.065e-03 (9.839e-04) & 4.51 (0.01) & 126.3 (11.2) \\ 
				\hline  10 & PINN & GD & 2.920e-02 (4.055e-03) & 0.41 (0) & 1.653e-02 (3.441e-03) & 1.82 (0.01) &  \\ 
				\cline{2-8} & SV-PINN & GD & 6.942e-02 (6.781e-03) & 0.76 (0.01) & 2.115e-02 (7.799e-03) & 2.22 (0.01) &  \\ 
				&  & L-BFGS & 1.043e-02 (7.006e-03) & 1.94 (0.01) & 1.073e-02 (6.380e-03) & 4.53 (0.03) &  \\
				\hline
		\end{tabular}}
	\end{table}
	
	\section{Generalisations and extensions}
	\label{SecGen}
	
	Even though we introduced SV-PINNs and performed illustrative experiments with second-order elliptic operators, the framework can be generalised and extended to other settings, which we discuss in this section. These extensions indicate that SV-PINNs are not restricted to second-order elliptic PDEs, but form an abstract framework based on the $\Phi$-stochastically weak norm applicable to broader classes of PDEs. We emphasise that the following extensions are presented at a conceptual level and detailed analytical and numerical investigations are left for future work.
	
	\subsection{High-order elliptic operators and poly-harmonic operators}
	
	Let $L_{m}$ be a $2m$-order elliptic operator and consider the problem
	\begin{equation}
		\label{high_order}
		\begin{aligned}
			L_{m}u = f & \text{ in } \Omega \\
			\frac{\partial^{i}u}{\partial\nu^{i}} = 0 & \text{ on } \partial\Omega \text{ for } i = 0,1,\dots,m-1
		\end{aligned}
	\end{equation}	
	in which $\frac{\partial^{i}u}{\partial\nu^{i}}$ is the $i$-th normal derivative. For instance, $\frac{\partial^{0}u}{\partial\nu^{0}} = u|_{\partial\Omega}$ and $\frac{\partial^{1}u}{\partial\nu^{1}} = \nabla u \cdot \nu$ in which $\nu$ is the normal vector. To characterise the weak solutions of \eqref{high_order}, we consider the eigenvalues and eigenfunctions of the poly-harmonic operator $(-\Delta)^{m}$ that satisfy the boundary conditions:
	\begin{align*}
		&(-\Delta)^{m} \phi_{k}^{(m)} = \lambda_{k}^{(m)} \phi_{k}^{(m)}\\
		&\frac{\partial^{i}\phi_{k}^{(m)}}{\partial\nu^{i}}\Big|_{\partial\Omega} = 0 \text{ for } i = 0,1,\dots,m-1.
	\end{align*}
	Assuming that $\Omega$ is bounded with sufficiently smooth boundary, the operator $(-\Delta)^m$ admits a discrete spectrum and $\{\phi_{k}^{(m)}\}$ is an orthonormal basis of $L^2(\Omega)$. In this case, let
	\begin{align*}
		H_{0}^{m}(\Omega) = \left\{u \in L^{2}(\Omega): \lVert u \rVert_{H_{0}^{m}}^{2} \coloneqq \sum_{k = 1}^{\infty} \lambda_{k}^{(m)} \, |\langle u,\phi_{k}^{(m)}\rangle|^{2} < \infty\right\}
	\end{align*}
	and its dual\footnote{We note that the norm $\lVert \cdot \rVert_{-s}$ defined in \eqref{norms} and the $\lVert \cdot \rVert_{-m}$ defined here are in general not the same. We use analogous notations to make the presentation clearer. The same applies for positive-index norms $\lVert \cdot \rVert_{\dot{H}^{s}}$ and $\lVert \cdot \rVert_{m}$.}
	\begin{align*}
		H^{-m}(\Omega) = \left\{R: H_{0}^{m}(\Omega) \to \mbR \text{ bounded linear}: \sum_{k = 1}^{\infty} (\lambda_{k}^{(m)})^{-1} \, |R(\phi_{k}^{(m)})|^{2} < \infty\right\}.
	\end{align*}
	For $f \in H^{-m}(\Omega)$, we can then define a weak solution as $u \in H_{0}^{m}(\Omega)$ such that
	\begin{align*}
		R_{u}^{(m)}(\varphi) = (L_mu - f)(\varphi) = 0, \, \, \forall \varphi \in H_{0}^{m}(\Omega), \text{ or equivalently, } \lVert R_{u}^{(m)} \rVert_{-m} = 0
	\end{align*}
	in which
	\begin{align*}
		\lVert R_{u}^{(m)} \rVert_{-m} = \sup\limits_{\substack{\varphi \in H_{0}^{m}(\Omega) \\ \varphi \neq 0}} \frac{|R_{u}^{(m)}(\varphi)|}{\lVert \varphi \rVert_{H_{0}^{m}}}.
	\end{align*}
	We refer to \cite[Chapter~7]{hackbusch2017elliptic} for sufficient conditions for the weak solution to be unique.
	
	By replacing $\lambda_{k}$ and $\phi_{k}$ with $\lambda_{k}^{(m)}$ and $\phi_{k}^{(m)}$ in the definition of $\Phi$ in \eqref{random_map}, a result analogous to Proposition \ref{prop_equiv} holds, and therefore the weak and $\Phi$-stochastically weak solutions coincide. Consequently, in this case, approximating the solution of a $2m$-order elliptic equation reduces to (a) computing eigenvalues and eigenfunctions of the poly-harmonic operator and (b) training an SV-PINN with DAFF given by the respective eigenfunctions.
	
	\subsection{Parabolic and hyperbolic equations}
	
	Fix $T > 0$ and let 
	\begin{equation*}
		L^2(0, T; H_0^1(\Omega)) \coloneqq \left\{v : [0, T] \to H_0^1(\Omega) : \int_0^T \| v(t) \|_{H_0^1(\Omega)}^2 \, dt < \infty \right\},
	\end{equation*}
	and analogously $L^2(0, T; L^2(\Omega))$ and $L^2(0, T; H^{-1}(\Omega))$, be Bochner spaces, and define
	\begin{align*}
		W_{\text{par}}(0, T) &\coloneqq \left\{ u \in L^2(0, T; H_0^1(\Omega)) : \partial_{t}u \in L^2(0, T; H^{-1}(\Omega)) \right\}\\
		W_{\text{hyp}}(0, T) &\coloneqq \left\{ u \in L^2(0, T; H_0^1(\Omega)) : \partial_{t}u \in L^2(0, T; L^2(\Omega)), \ \partial_{tt}u \in L^2(0, T; H^{-1}(\Omega)) \right\}.
	\end{align*}	
	
	Consider the parabolic equation in $[0,T] \times \Omega$
	\begin{equation}
		\label{parabolic}
	 	\begin{aligned}
	 		\partial_{t}u & = Lu + f, & & \text{ in } (0,T) \times \Omega \\ 
	 		u(0,x) & = g(x), & & \text{ for } x \in \Omega\\
	 		u(t,x) &= 0, & & \text{ for }  (t,x) \in (0,T) \times \partial\Omega
	 	\end{aligned}
	\end{equation}
	in which $L$ is a second-order linear elliptic operator, $f \in L^2(0, T; H^{-1}(\Omega))$ and $g \in L^2(\Omega)$. A $u \in W_{\text{par}}(0, T)$ is a weak solution of \eqref{parabolic} if, for almost all $t \in (0,T)$,
	\begin{align*}
	 	&R_{u(t)}^{\text{par}}(\varphi) \coloneqq (\partial_{t}u(t) - Lu(t) - f(t))(\varphi) = 0, \, \forall \varphi \in H_{0}^{1}(\Omega), \text{ or equivalently, }\\
	 	&\lVert R_{u(t)}^{\text{par}} \rVert_{-1} \coloneqq \lVert \partial_{t}u(t) - Lu(t) - f(t) \rVert_{-1} = 0
	\end{align*}
	and $u(0) = g$ in $L^{2}(\Omega)$. We assume the weak solution of \eqref{parabolic} is unique and we refer to \cite[Chapter~7]{evans2022partial} for more details about this characterisation.
 
	Likewise, consider the hyperbolic equation in $[0,T] \times \Omega$
	\begin{equation}
		\label{hyperbolic}
		\begin{aligned}
			\partial_{tt}u & = Lu + f, & & \text{ in } (0,T) \times \Omega \\ 
			u(0,x) & = g_0(x), & & \text{ for } x \in \Omega\\
			\partial_t u(0,x) & = g_1(x), & & \text{ for } x \in \Omega\\
			u(t,x) & = 0, & & \text{ for } (t,x) \in (0,T) \times \partial\Omega
		\end{aligned}
	\end{equation}
	in which $L$ is a second-order linear elliptic operator, $f \in L^2(0, T; H^{-1}(\Omega))$, $g_0 \in H_{0}^{1}(\Omega)$ and $g_1 \in L^2(\Omega)$. A $u \in W_{\text{hyp}}(0, T)$ is a weak solution of \eqref{hyperbolic} if, for almost all $t \in (0,T)$,
	\begin{align*}
		&R_{u(t)}^{\text{hyp}}(\varphi) \coloneqq (\partial_{tt}u(t) - Lu(t) - f(t))(\varphi) = 0, \, \forall \varphi \in H_{0}^{1}(\Omega), \text{ or equivalently, }\\
		&\lVert R_{u(t)}^{\text{hyp}} \rVert_{-1} \coloneqq \lVert \partial_{tt}u(t) - Lu(t) - f(t) \rVert_{-1} = 0
	\end{align*}
	and the initial conditions are satisfied: $u(0) = g_0$ and $\partial_t u(0) = g_1$ in $L^2(\Omega)$. Again we assume the weak solution of \eqref{hyperbolic} is unique and we refer to \cite[Chapter~7]{evans2022partial} for more details.
	
	We can define the $\Phi$-stochastically weak solutions of parabolic and hyperbolic linear equations based on the random map $\Phi$, analogous to Definition \ref{def_mu_weak}. It is direct from Proposition \ref{prop_equiv} applied pointwise in time for almost every $t \in (0,T)$ that weak and $\Phi$-stochastically weak solutions, of the parabolic problem \eqref{parabolic} and hyperbolic problem \eqref{hyperbolic}, are equivalent for the random map $\Phi$ in \eqref{random_map}.
	
	\begin{definition}
		\label{def_mu_weak_par}
		A $\Phi$-stochastically weak solution of the parabolic homogeneous Dirichlet problem \eqref{parabolic} is any $u \in W_{\text{par}}(0, T)$ satisfying $\lVert R_{u(t)}^{\text{par}} \rVert_{\Phi} = 0$ for almost every $t \in (0,T)$ and $u(0) = g$ in $L^{2}(\Omega)$. A $\Phi$-stochastically weak solution of the hyperbolic homogeneous Dirichlet problem \eqref{hyperbolic} is any $u \in W_{\text{hyp}}(0, T)$ satisfying $\lVert R_{u(t)}^{\text{hyp}} \rVert_{\Phi} = 0$ for almost every $t \in (0,T)$, and $u(0) = g_0$ and $\partial_t u(0) = g_1$ in $L^2(\Omega)$.
	\end{definition}
	
	The equivalence between weak and $\Phi$-stochastically weak solutions yields respectively the parabolic SV-PINN optimisation problem
	\begin{align*}
		&\min\limits_{\theta} \int_{0}^{T} \lVert \partial_{t}u_{\theta}(t,\cdot) - Lu_{\theta}(t,\cdot) - f(t,\cdot) \rVert_{\Phi}^2 \, dt + \lambda \, \lVert u_{\theta}(0,\cdot) - g \rVert_{L^2}^{2}
	\end{align*}
	and the hyperbolic SV-PINN optimisation problem
	\begin{align*}
		\min\limits_{\theta} \int_{0}^{T} \lVert \partial_{tt}u_{\theta}(t,\cdot) - Lu_{\theta}(t,\cdot) - f(t,\cdot) \rVert_{\Phi}^2 \, dt + \lambda_{1} \, \lVert u_{\theta}(0,\cdot) - g_{0} \rVert_{L^2}^{2} + \lambda_{2} \, \lVert \partial_{t}u_{\theta}(0,\cdot) - g_{1} \rVert_{L^2}^{2}
	\end{align*}
	for $\lambda,\lambda_{1},\lambda_{2} > 0$, assuming the zero boundary conditions are hard constraints on the architecture by, for example, DAFF on the spatial coordinates.
	
	The norms and integrals above can be discretised as in \eqref{app_sw} and an SV-PINN trained by \lbs or GD. Non-homogeneous boundary conditions, such as periodic, can also be considered by adding an extra term to the loss or considering other hard constraints. The operator $L$ could be replaced by non-linear operators and the SV-PINNs could still be applied provided a suitable random map $\Phi$, which induces a norm that is equivalent to the respective weak norm, is available. SV-PINNs applied to parabolic and hyperbolic problems, with linear and non-linear space operators, will be deferred to a future paper.
	
	\subsection{Galerkin-type methods based on $\Phi$-stochastically weak norms}
	
	The $\Phi$-stochastically weak norm induces, given a trial space $\mathcal{U} \subset H_{0}^{1}(\Omega)$, a Galerkin-type approximation to the solution of the homogeneous problem \eqref{basic_problem}, obtained by solving
	\begin{align}
		\label{PG}
		\min_{u \in \mathcal{U}} \, \sum_{j = 1}^{N} |R_{u}(\varphi_{j})|^{2}
	\end{align}
	for test functions $\varphi_{j}$ sampled from $\Phi$. This is an empirical approximation of
	\begin{align*}
		\min_{u \in \mathcal{U}} \, \lVert R_{u} \rVert_{\Phi}^{2}
	\end{align*}
	and represents a shift from classical Galerkin methods in which the test space is the same as the trial space, or is selected in a deterministic PDE-operator-dependent way.
	
	An empirical comparison of the method induced by \eqref{PG} with classical Galerkin methods that use the same space $\mathcal{U}$ is a natural and important topic for future research. From a theoretical perspective, it is necessary to control the error of the solution of \eqref{PG} that, as discussed in Section \ref{Sec_error}, depends on the quadrature error and concentration inequalities for the empirical $\Phi$-stochastically weak norm. As mentioned, these concentration inequalities are more tractable for convex classes $\mathcal{U}$ than for general nonlinear classes, and therefore the former provides the natural baseline to address error analysis in this setting.
	
	\subsection{Abstract framework}
	
	The equivalence between weak and $\Phi$-stochastically weak solutions, which leads to an equivalent characterisation of the weak solution of an operator equation, can be stated in a more abstract manner and is related to well-established concepts. Let $H$ be a separable Hilbert space with inner product $\langle \cdot, \cdot \rangle_H$ and 
	\begin{align*}
		A : D(A) \subset H \to H
	\end{align*}
	be a densely defined, self-adjoint, non-negative operator with compact resolvent. By the spectral theorem\footnote{This follows by applying the spectral theorem for compact self-adjoint operators \cite[Theorem~6.11]{brezis2011functional} to $(A+I)^{-1}$, which is compact and self-adjoint since $A$ is self-adjoint and non-negative with compact resolvent. The eigenvectors of $(A+I)^{-1}$ coincide with those of $A$, and if $(A+I)^{-1}\phi_k = \mu_k \phi_k$ then $A\phi_k = (\mu_k^{-1} - 1)\phi_k \eqqcolon \lambda_k \phi_k$.}, there exists an orthonormal basis $\{\phi_k\}$ of $H$ and a non-decreasing sequence of eigenvalues $0 \le \lambda_1 \le \lambda_2 \le \cdots \to \infty$ such that $A \phi_k = \lambda_k \phi_k$  for all $k \in \mathbb{N}$.
	
	For $s > 0$, we define the Hilbert space\footnote{We add one to the eigenvalues since we are assuming that $A$ is non-negative, so $\lambda_{k} = 0$ is possible for some values of $k$ and the norm is not well-defined for $s < 0$. If $A$ is positive, the addition can be omitted.}
	\begin{align*}
		H_s \coloneqq \left\{u \in H : \lVert u \rVert_{H_s}^2 \coloneqq \sum_{k=1}^{\infty} (1 + \lambda_k)^{s} |\langle u, \phi_k \rangle_H|^2 < \infty
		\right\}
	\end{align*}
	and its dual space
	\begin{align*}
		H_{-s} \coloneqq \left\{R: H_{s} \to \mbR \text{ bounded linear} : \lVert R \rVert_{-s}^2 \coloneqq \sum_{k=1}^{\infty} (1 + \lambda_k)^{-s} |R(\phi_k)|^2 < \infty
		\right\}.
	\end{align*}
	Formally, the $\Phi$-stochastically weak norm defined as
	\begin{align*}
		\lVert R \rVert_{\Phi_{s}}^{2} \coloneqq \lim\limits_{n \to \infty} \mu\left[|R(\Phi_{s}^{(n)})|^{2}\right]
	\end{align*}
	with
	\begin{align*}
		\Phi_{s}^{(n)}(\xi) = \tau \sum_{k = 1}^{n} (1 + \lambda_{k})^{-s/2} w_{k}(\xi) \phi_{k}
	\end{align*}
	for standard independent Gaussian random variables $w_{k}$ is equivalent to the $H_{-s}$ norm by an analogous argument to Proposition \ref{prop_equiv}.
	
	In this abstract setting, the process $\Phi_{s}$, when it is a Gaussian process, is called the \textit{isonormal Gaussian} process associated with the Hilbert space $H_{s}$, which plays a central role in Malliavin calculus \cite{nualart2006malliavin}. We refer to \cite[Chapter~2]{nourdin2012normal}, see also \cite{tubaro2025introduction}, for a formal presentation of isonormal Gaussian processes, Gaussian Hilbert spaces and abstract results that are related to the aforementioned norm equivalence. Importantly, however, the main interest in Malliavin calculus is on computing derivatives of functionals applied to stochastic process, not on any aspect of numerical analysis. In particular, it usually studies Gaussian processes that generate a norm equivalent to that of a Hilbert space that is a subset of $L^{2}(\Omega)$, while in our case the random process $\Phi$, which might not even be a Gaussian process, induces a norm equivalent to that of the dual Sobolev spaces. 
	
	In this framework, let
	\begin{align*}
		\mathcal{L} : H_s \to H_{-s}
	\end{align*}
	be a bounded, invertible linear operator. Given $f \in H_{-s}$, we can characterise the weak solution of the equation
	\begin{align*}
		\mathcal{L} u = f
	\end{align*}
	as the element $u \in H_s$ satisfying
	\begin{align*}
		\lVert \mathcal{L}u - f \rVert_{-s} = 0, \text{ or, equivalently, } \lVert \mathcal{L}u - f \rVert_{\Phi_s} = 0.
	\end{align*}
	Making this abstract framework concrete for specific operators $A$ and $\mcL$ would yield numerical methods analogous to SV-PINNs across a wider range of problems.
	
	\section{Discussion}
	\label{SecDis}
	
	The numerical solution of PDEs involves a well-known tension: dual-norm characterisations of weak solutions are mathematically natural, yet the supremum that defines them is computationally intractable. This paper approaches this by replacing that supremum with an expectation over random test functions, converting the optimisation problem into a stochastic one while aiming to preserve the correct weak topology up to empirical approximation.
	
	A central implication of the norm equivalence is that the quality of the weak formulation is no longer tied to the choice of a finite-dimensional deterministic test space. In contrast to V-PINNs, and more broadly Petrov-Galerkin methods, where the optimisation problem depends in an essential way on the selected test functions, the stochastically weak norm induces the correct weak topology independently of any particular basis or operator-dependent construction up to a sampling error that is, in principle, controllable. 
	
	At the same time, the framework highlights an unconventional but fundamental role of rough test functions which, in deterministic Galerkin theory, would be deemed unsuitable. This paper indicates that this intuition does not extend to stochastic testing: stability and norm equivalence emerge only after averaging, and do not require regularity at the level of individual test functions. This separation between pointwise regularity and averaged stability is a defining feature of the approach and suggests a broader principle that may be relevant beyond the particular context of neural networks.
	
	The empirical results indicate that minimising a stochastically weak residual can significantly mitigate several known failure modes of strong-form PINNs, particularly in regimes characterised by high frequencies, multi-scale behaviour, or indefinite operators. Furthermore, the clear superiority of \lb suggests that the stochastically weak norm induces a smoother or better-conditioned optimisation landscape than strong residual losses, though a precise characterisation remains to be established. This leads to the recovery of the solution within an $L^2$ relative error of $1\%$ in hundreds of steps, whereas usual PINN methods take thousands of steps, typically followed by additional thousands of steps for fine-tuning with \lb.
	
	Domain-aware Fourier features substantially improve the convergence and stability of both PINNs and SV-PINNs in challenging scenarios. However, even with DAFF, PINNs consistently fall short of SV-PINNs, indicating that the improvement brought by the stochastically weak formulation is not due to architectural choices alone. An interesting feature of the results is that, when the SV-PINN is properly trained with \lb, the pointwise error becomes highly irregular. This appears to reflect the regularity of the random test functions (cf. Lemma \ref{lemma_traj}): since realisations of $\Phi$ belong almost surely to $\dot{H}^{1-d/2-\epsilon}(\Omega)$, the residual	is probed in a negative-order Sobolev sense, and the resulting error may be small in the weak norm while remaining visibly irregular pointwise. A formal characterisation of this phenomenon would be relevant.
	
	Finally, the approach also exposes a clear trade-off. While stochastic testing removes the dependence on operator-specific test spaces, it introduces reliance on spectral information of the domain and on probabilistic error control. In domains where the Laplacian spectrum is readily available, this trade-off is favourable; in more general geometries, this may not be the case. Moreover, since convergence analysis depends on uniform concentration over the trial class, classical Galerkin error estimates are insufficient, and new tools that blend numerical analysis with 
	empirical process theory are necessary. A first step in this direction is taken in Section \ref{Sec_error}, where pointwise-in-$\theta$ bounds are established; the uniform extension over the trial class remains the central open problem.
		
	\subsection{Limitations}
	
	The main limitation of this method is that it is memory heavy. The necessity of computing, via automatic differentiation, the second derivative of the function represented by the neural network, which is given by a large neural network architecture with DAFF, and then combining the result with thousands of test functions, all within the \lbs algorithm, makes the SV-PINN very memory-intensive. The experiments could only be performed on state-of-the-art GPUs and exceed the memory capacity of standard consumer hardware. This limitation is especially striking for problems in three dimensions, where, as illustrated in Experiment 6, these issues are compounded with the usual curse of dimensionality. 
	
	Furthermore, as opposed to V-PINNs, and also the Deep Ritz Methods, in which integration by parts is leveraged to compute the loss function by taking only first derivatives of the neural network and of the test function, in SV-PINNs, this is not theoretically possible because the test functions are not differentiable. Nevertheless, in practice, since the test functions are actually a truncation of the spectral sum that defines $\Phi$, they are differentiable and integration by parts could be applied to save one derivative from the neural network and improve the computational efficiency. Moreover, in future work the implementation can be greatly improved to decrease the memory burden since it has not been developed to be strictly memory efficient, but rather as a proof-of-concept for SV-PINNs.
	
	As discussed above, the necessity of computing the eigenfunctions and eigenvalues of the Laplace operator might be a limitation in complex geometries, even though they only have to be computed once. However, as illustrated with the L-shaped domain in Experiment 8, a radial basis function interpolation of eigenfunctions computed by a basic FEM method may suffice to properly compute the eigenvalues and eigenfunctions in a way that allows relevant problems to be solved. This suggests that, if sophisticated numerical methods are applied to compute these functions, a robust method may be achievable.
	
	\subsection{Future research}
	
	As already pointed out throughout the paper, there are many interesting topics for future research. For instance, we plan to apply the SV-PINNs to parabolic and hyperbolic equations, and to perform an error analysis by developing concentration inequalities for the $\Phi$-stochastically weak norm in subsequent papers.
	
	Other relevant topics for future research are performing an in-depth benchmark of SV-PINNs for second-order linear elliptic operators, applying the SV-PINNs to higher-order linear elliptic operators, further investigating their behaviour in three-dimensional problems, and implementing the Galerkin-type method in which the $\Phi$-stochastically weak norm is minimised in other trial spaces. Domain decomposition could also be leveraged in this context. In addition, SV-PINNs, and more broadly the minimisation of the $\Phi$-stochastically weak norm, could be studied in the context of inverse problems and operator learning, in particular within the DeepONets framework \cite{lu2021learning}.
	
	Several further generalisations of the SV-PINN framework merit future investigation. First, extensions to stochastic PDEs are conceivable by considering randomness over both the test functions and the solution, connecting to the theory of mild solutions in Hilbert spaces. Second, Neumann and Robin boundary conditions require a different eigenbasis, namely that of the Neumann or Robin Laplacian, and a corresponding adaptation of the function spaces, which is a concrete and practically important direction. Finally, since the abstract framework of Section \ref{SecGen} requires only a densely defined self-adjoint operator with compact resolvent on a separable Hilbert space, it applies immediately to PDEs posed on compact Riemannian manifolds via the Laplace-Beltrami operator, which is particularly relevant for applications in geometric deep learning.
	
	\subsection{Final remark}
	
	This work shows that stochastically weak formulations can bridge the gap between abstract dual norms and practical algorithms. In particular, it shifts the question from which test functions to choose to how well the stochastically weak norm can be approximated, a question that belongs to empirical process theory rather than classical numerical analysis. Whether this shift ultimately proves more tractable remains to be seen, but the theoretical guarantees and empirical results presented here suggest that stochastic testing is not merely a computational convenience, but a genuinely different way of operationalising weak formulations of PDEs.
		
	\section*{Acknowledgements}  
	
	This work was supported by computational resources provided by the Australian Government through the National Computational Infrastructure (NCI) under the ANU Startup Scheme. I thank Lucas Franceschini, Nelson Kuhl, Pierre Portal and Adilson Simonis for stimulating discussions. AI-assisted tools were used for language editing during the revision of this manuscript.
	
	\FloatBarrier
	\bibliographystyle{plain}
	\bibliography{ref}

\begin{thebibliography}{10}

\bibitem{ainsworth2021galerkin}
Mark Ainsworth and Justin Dong.
\newblock Galerkin neural networks: A framework for approximating variational
  equations with error control.
\newblock {\em SIAM Journal on Scientific Computing}, 43(4):A2474--A2501, 2021.

\bibitem{babuvska2007stochastic}
Ivo Babu{\v{s}}ka, Fabio Nobile, and Ra{\'u}l Tempone.
\newblock A stochastic collocation method for elliptic partial differential
  equations with random input data.
\newblock {\em SIAM Journal on Numerical Analysis}, 45(3):1005--1034, 2007.

\bibitem{bartlett2002rademacher}
Peter~L Bartlett and Shahar Mendelson.
\newblock Rademacher and gaussian complexities: Risk bounds and structural
  results.
\newblock {\em Journal of machine learning research}, 3(Nov):463--482, 2002.

\bibitem{baydin2018automatic}
Atilim~Gunes Baydin, Barak~A Pearlmutter, Alexey~Andreyevich Radul, and
  Jeffrey~Mark Siskind.
\newblock Automatic differentiation in machine learning: a survey.
\newblock {\em Journal of machine learning research}, 18(153):1--43, 2018.

\bibitem{berrone2022variational}
Stefano Berrone, Claudio Canuto, and Moreno Pintore.
\newblock Variational physics informed neural networks: the role of quadratures
  and test functions.
\newblock {\em Journal of Scientific Computing}, 92(3):100, 2022.

\bibitem{jaxopt_implicit_diff}
Mathieu Blondel, Quentin Berthet, Marco Cuturi, Roy Frostig, Stephan Hoyer,
  Felipe Llinares-L{\'o}pez, Fabian Pedregosa, and Jean-Philippe Vert.
\newblock Efficient and modular implicit differentiation.
\newblock {\em arXiv preprint arXiv:2105.15183}, 2021.

\bibitem{bolin2023equivalence}
David Bolin and Kristin Kirchner.
\newblock Equivalence of measures and asymptotically optimal linear prediction
  for gaussian random fields with fractional-order covariance operators.
\newblock {\em Bernoulli}, 29(2):1476--1504, 2023.

\bibitem{concentration}
Stephane Boucheron, Gabor Lugosi, and Pascal Massart.
\newblock {\em Concentration Inequalities: A Nonasymptotic Theory of
  Independence}.
\newblock Oxford University Press, 02 2013.

\bibitem{jax2018github}
James Bradbury, Roy Frostig, Peter Hawkins, Matthew~James Johnson, Chris Leary,
  Dougal Maclaurin, George Necula, Adam Paszke, Jake Vander{P}las, Skye
  Wanderman-{M}ilne, and Qiao Zhang.
\newblock {JAX}: composable transformations of {P}ython+{N}um{P}y programs,
  2018.

\bibitem{braess2001finite}
Dietrich Braess.
\newblock {\em Finite elements: Theory, fast solvers, and applications in solid
  mechanics}.
\newblock Cambridge University Press, 2001.

\bibitem{bramble1970estimation}
James~H Bramble and Stephen~Reginald Hilbert.
\newblock Estimation of linear functionals on sobolev spaces with application
  to fourier transforms and spline interpolation.
\newblock {\em SIAM Journal on Numerical Analysis}, 7(1):112--124, 1970.

\bibitem{brezis2011functional}
Haim Brezis.
\newblock {\em Functional analysis, Sobolev spaces and partial differential
  equations}.
\newblock Springer, 2011.

\bibitem{britanak2010discrete}
Vladimir Britanak, Patrick~C Yip, and Kamisetty~Ramamohan Rao.
\newblock {\em Discrete cosine and sine transforms: general properties, fast
  algorithms and integer approximations}.
\newblock Elsevier, 2010.

\bibitem{calero2026enhancing}
Alberto~Mi{\~n}o Calero, Luis Salamanca, and Konstantinos~E Tatsis.
\newblock Enhancing physics-informed neural networks with domain-aware fourier
  features: Towards improved performance and interpretable results.
\newblock {\em arXiv preprint arXiv:2603.02948}, 2026.

\bibitem{chen2021solving}
Yifan Chen, Bamdad Hosseini, Houman Owhadi, and Andrew~M Stuart.
\newblock Solving and learning nonlinear pdes with gaussian processes.
\newblock {\em Journal of Computational Physics}, 447:110668, 2021.

\bibitem{cockayne2019bayesian}
Jon Cockayne, Chris~J Oates, Timothy~John Sullivan, and Mark Girolami.
\newblock Bayesian probabilistic numerical methods.
\newblock {\em SIAM review}, 61(4):756--789, 2019.

\bibitem{courant2024methods}
Richard Courant and David Hilbert.
\newblock {\em Methods of mathematical physics, volume 1}.
\newblock Interscience Publishers, 1953.

\bibitem{cross2024spectrum}
Elizabeth~J Cross, Timothy~J Rogers, Daniel~J Pitchforth, Samuel~J Gibson,
  Sikai Zhang, and Matthew~R Jones.
\newblock A spectrum of physics-informed gaussian processes for regression in
  engineering.
\newblock {\em Data-Centric Engineering}, 5:e8, 2024.

\bibitem{cuomo2022scientific}
Salvatore Cuomo, Vincenzo~Schiano Di~Cola, Fabio Giampaolo, Gianluigi Rozza,
  Maziar Raissi, and Francesco Piccialli.
\newblock Scientific machine learning through physics--informed neural
  networks: Where we are and what's next.
\newblock {\em Journal of Scientific Computing}, 92(3):88, 2022.

\bibitem{deepmind2020jax}
DeepMind, Igor Babuschkin, Kate Baumli, Alison Bell, Surya Bhupatiraju, Jake
  Bruce, Peter Buchlovsky, David Budden, Trevor Cai, Aidan Clark, Ivo
  Danihelka, Antoine Dedieu, Claudio Fantacci, Jonathan Godwin, Chris Jones,
  Ross Hemsley, Tom Hennigan, Matteo Hessel, Shaobo Hou, Steven Kapturowski,
  Thomas Keck, Iurii Kemaev, Michael King, Markus Kunesch, Lena Martens, Hamza
  Merzic, Vladimir Mikulik, Tamara Norman, George Papamakarios, John Quan,
  Roman Ring, Francisco Ruiz, Alvaro Sanchez, Laurent Sartran, Rosalia
  Schneider, Eren Sezener, Stephen Spencer, Srivatsan Srinivasan, Milo\v{s}
  Stanojevi\'{c}, Wojciech Stokowiec, Luyu Wang, Guangyao Zhou, and Fabio
  Viola.
\newblock The {D}eep{M}ind {JAX} {E}cosystem, 2020.

\bibitem{demkowicz2017discontinuous}
Leszek Demkowicz and Jay Gopalakrishnan.
\newblock Discontinuous petrov--galerkin (dpg) method.
\newblock {\em Encyclopedia of Computational Mechanics}, 2:777--792, 2017.

\bibitem{devore2021neural}
Ronald DeVore, Boris Hanin, and Guergana Petrova.
\newblock Neural network approximation.
\newblock {\em Acta Numerica}, 30:327--444, 2021.

\bibitem{di2012hitchhikers}
Eleonora Di~Nezza, Giampiero Palatucci, and Enrico Valdinoci.
\newblock Hitchhiker's guide to the fractional sobolev spaces.
\newblock {\em Bulletin des sciences math{\'e}matiques}, 136(5):521--573, 2012.

\bibitem{durrett2019probability}
Rick Durrett.
\newblock {\em Probability: theory and examples}, volume~49.
\newblock Cambridge university press, 2019.

\bibitem{evans2022partial}
Lawrence~C Evans.
\newblock {\em Partial differential equations}, volume~19.
\newblock American mathematical society, 2022.

\bibitem{gao2022physics}
Han Gao, Matthew~J Zahr, and Jian-Xun Wang.
\newblock Physics-informed graph neural galerkin networks: A unified framework
  for solving pde-governed forward and inverse problems.
\newblock {\em Computer Methods in Applied Mechanics and Engineering},
  390:114502, 2022.

\bibitem{glorot2010understanding}
Xavier Glorot and Yoshua Bengio.
\newblock Understanding the difficulty of training deep feedforward neural
  networks.
\newblock In {\em Proceedings of the thirteenth international conference on
  artificial intelligence and statistics}, pages 249--256. JMLR Workshop and
  Conference Proceedings, 2010.

\bibitem{hackbusch2017elliptic}
Wolfgang Hackbusch.
\newblock {\em Elliptic differential equations: theory and numerical
  treatment}, volume~18.
\newblock Springer, 2017.

\bibitem{hu2025iterative}
Tianhao Hu, Bangti Jin, and Fengru Wang.
\newblock An iterative deep ritz method for monotone elliptic problems.
\newblock {\em Journal of Computational Physics}, 527:113791, 2025.

\bibitem{ivrii2016100}
Victor Ivrii.
\newblock 100 years of weyl's law.
\newblock {\em Bulletin of Mathematical Sciences}, 6(3):379--452, 2016.

\bibitem{jakeman2008stochastic}
John~Davis Jakeman and Stephen~Gwyn Roberts.
\newblock Stochastic galerkin and collocation methods for quantifying
  uncertainty in differential equations: a review.
\newblock {\em The Proceedings of ANZIAM}, 50:C815--C830, 2008.

\bibitem{johnson1987improved}
H~Fisk Johnson.
\newblock An improved method for computing a discrete hankel transform.
\newblock {\em Computer physics communications}, 43(2):181--202, 1987.

\bibitem{karniadakis2021physics}
George~Em Karniadakis, Ioannis~G Kevrekidis, Lu~Lu, Paris Perdikaris, Sifan
  Wang, and Liu Yang.
\newblock Physics-informed machine learning.
\newblock {\em Nature Reviews Physics}, 3(6):422--440, 2021.

\bibitem{kharazmi2019variational}
Ehsan Kharazmi, Zhongqiang Zhang, and George~Em Karniadakis.
\newblock Variational physics-informed neural networks for solving partial
  differential equations.
\newblock {\em arXiv preprint arXiv:1912.00873}, 2019.

\bibitem{kharazmi2021hp}
Ehsan Kharazmi, Zhongqiang Zhang, and George~Em Karniadakis.
\newblock hp-vpinns: Variational physics-informed neural networks with domain
  decomposition.
\newblock {\em Computer Methods in Applied Mechanics and Engineering},
  374:113547, 2021.

\bibitem{kim2020fractional}
Seungil Kim.
\newblock Fractional order sobolev spaces for the neumann laplacian and the
  vector laplacian.
\newblock {\em Journal of the Korean Mathematical Society}, 57(3):721--745,
  2020.

\bibitem{kingma2014adam}
Diederik~P Kingma and Jimmy Ba.
\newblock Adam: A method for stochastic optimization.
\newblock {\em arXiv preprint arXiv:1412.6980}, 2014.

\bibitem{korte2025smoothness}
Moritz Korte-Stapff, Toni Karvonen, and {\'E}ric Moulines.
\newblock Smoothness estimation for whittle--mat{\'e}rn processes on closed
  riemannian manifolds.
\newblock {\em Stochastic Processes and their Applications}, 189:104685, 2025.

\bibitem{lindgren2022spde}
Finn Lindgren, David Bolin, and H{\aa}vard Rue.
\newblock The spde approach for gaussian and non-gaussian fields: 10 years and
  still running.
\newblock {\em Spatial Statistics}, 50:100599, 2022.

\bibitem{lindgren2011explicit}
Finn Lindgren, H{\aa}vard Rue, and Johan Lindstr{\"o}m.
\newblock An explicit link between gaussian fields and gaussian markov random
  fields: the stochastic partial differential equation approach.
\newblock {\em Journal of the Royal Statistical Society Series B: Statistical
  Methodology}, 73(4):423--498, 2011.

\bibitem{lord2014introduction}
Gabriel~J Lord, Catherine~E Powell, and Tony Shardlow.
\newblock {\em An introduction to computational stochastic PDEs}, volume~50.
\newblock Cambridge University Press, 2014.

\bibitem{lu2021learning}
Lu~Lu, Pengzhan Jin, Guofei Pang, Zhongqiang Zhang, and George~Em Karniadakis.
\newblock Learning nonlinear operators via deeponet based on the universal
  approximation theorem of operators.
\newblock {\em Nature machine intelligence}, 3(3):218--229, 2021.

\bibitem{luo2025physics}
Kuang Luo, Jingshang Zhao, Yingping Wang, Jiayao Li, Junjie Wen, Jiong Liang,
  Henry Soekmadji, and Shaolin Liao.
\newblock Physics-informed neural networks for pde problems: A comprehensive
  review.
\newblock {\em Artificial Intelligence Review}, 58(10):323, 2025.

\bibitem{mahoney2011randomized}
Michael~W Mahoney.
\newblock Randomized algorithms for matrices and data.
\newblock {\em Foundations and Trends{\textregistered} in Machine Learning},
  3(2):123--224, 2011.

\bibitem{marcondes2025complexity}
Diego Marcondes.
\newblock Complexity dependent error rates for physics-informed statistical
  learning via the small-ball method.
\newblock {\em arXiv preprint arXiv:2510.23149}, 2025.

\bibitem{mcclenny2023self}
Levi~D McClenny and Ulisses~M Braga-Neto.
\newblock Self-adaptive physics-informed neural networks.
\newblock {\em Journal of Computational Physics}, 474:111722, 2023.

\bibitem{mendelson2015learning}
Shahar Mendelson.
\newblock Learning without concentration.
\newblock {\em Journal of the ACM (JACM)}, 62(3):1--25, 2015.

\bibitem{nocedal1980updating}
Jorge Nocedal.
\newblock Updating quasi-newton matrices with limited storage.
\newblock {\em Mathematics of computation}, 35(151):773--782, 1980.

\bibitem{nochetto2015pde}
Ricardo~H Nochetto, Enrique Ot{\'a}rola, and Abner~J Salgado.
\newblock A pde approach to fractional diffusion in general domains: a priori
  error analysis.
\newblock {\em Foundations of Computational Mathematics}, 15(3):733--791, 2015.

\bibitem{nourdin2012normal}
Ivan Nourdin and Giovanni Peccati.
\newblock {\em Normal approximations with Malliavin calculus: from Stein's
  method to universality}, volume 192.
\newblock Cambridge University Press, 2012.

\bibitem{nualart2006malliavin}
David Nualart.
\newblock {\em The Malliavin calculus and related topics}.
\newblock Springer, 2006.

\bibitem{python}
{Python Core Team}.
\newblock {\em {Python: A dynamic, open source programming language}}.
\newblock {Python Software Foundation}, 2024.
\newblock Version 3.9.2.

\bibitem{rahaman2019spectral}
Nasim Rahaman, Aristide Baratin, Devansh Arpit, Felix Draxler, Min Lin, Fred
  Hamprecht, Yoshua Bengio, and Aaron Courville.
\newblock On the spectral bias of neural networks.
\newblock In {\em International conference on machine learning}, pages
  5301--5310. PMLR, 2019.

\bibitem{rahimi2007random}
Ali Rahimi and Benjamin Recht.
\newblock Random features for large-scale kernel machines.
\newblock {\em Advances in neural information processing systems}, 20, 2007.

\bibitem{raissi2019physics}
Maziar Raissi, Paris Perdikaris, and George~E Karniadakis.
\newblock Physics-informed neural networks: A deep learning framework for
  solving forward and inverse problems involving nonlinear partial differential
  equations.
\newblock {\em Journal of Computational physics}, 378:686--707, 2019.

\bibitem{strang1974analysis}
Gilbert Strang, George~J Fix, and DS~Griffin.
\newblock {\em An analysis of the finite-element method}.
\newblock Prentice-Hall, 1974.

\bibitem{strohmer2009randomized}
Thomas Strohmer and Roman Vershynin.
\newblock A randomized kaczmarz algorithm with exponential convergence.
\newblock {\em Journal of Fourier Analysis and Applications}, 15(2):262--278,
  2009.

\bibitem{talagrand1996majorizing}
Michel Talagrand.
\newblock Majorizing measures: the generic chaining.
\newblock {\em The Annals of Probability}, 24(3):1049--1103, 1996.

\bibitem{talagrand2005generic}
Michel Talagrand.
\newblock {\em The generic chaining: upper and lower bounds of stochastic
  processes}.
\newblock Springer, 2005.

\bibitem{thomee2007galerkin}
Vidar Thom{\'e}e.
\newblock {\em Galerkin finite element methods for parabolic problems},
  volume~25.
\newblock Springer Science \& Business Media, 2007.

\bibitem{toscano2025pinns}
Juan~Diego Toscano, Vivek Oommen, Alan~John Varghese, Zongren Zou, Nazanin
  Ahmadi~Daryakenari, Chenxi Wu, and George~Em Karniadakis.
\newblock From pinns to pikans: Recent advances in physics-informed machine
  learning.
\newblock {\em Machine Learning for Computational Science and Engineering},
  1(1):15, 2025.

\bibitem{tubaro2025introduction}
Luciano Tubaro and Margherita Zanella.
\newblock An introduction to malliavin calculus.
\newblock {\em arXiv preprint arXiv:2502.07941}, 2025.

\bibitem{wang2008karhunen}
Limin Wang.
\newblock {\em Karhunen-Loeve expansions and their applications}.
\newblock London School of Economics and Political Science (United Kingdom),
  2008.

\bibitem{wang2024piratenets}
Sifan Wang, Bowen Li, Yuhan Chen, and Paris Perdikaris.
\newblock Piratenets: Physics-informed deep learning with residual adaptive
  networks.
\newblock {\em Journal of Machine Learning Research}, 25(402):1--51, 2024.

\bibitem{wang2022respecting}
Sifan Wang, Shyam Sankaran, and Paris Perdikaris.
\newblock Respecting causality is all you need for training physics-informed
  neural networks.
\newblock {\em arXiv preprint arXiv:2203.07404}, 2022.

\bibitem{wang2021understanding}
Sifan Wang, Yujun Teng, and Paris Perdikaris.
\newblock Understanding and mitigating gradient flow pathologies in
  physics-informed neural networks.
\newblock {\em SIAM Journal on Scientific Computing}, 43(5):A3055--A3081, 2021.

\bibitem{wang2021eigenvector}
Sifan Wang, Hanwen Wang, and Paris Perdikaris.
\newblock On the eigenvector bias of fourier feature networks: From regression
  to solving multi-scale pdes with physics-informed neural networks.
\newblock {\em Computer Methods in Applied Mechanics and Engineering},
  384:113938, 2021.

\bibitem{wang2022and}
Sifan Wang, Xinling Yu, and Paris Perdikaris.
\newblock When and why pinns fail to train: A neural tangent kernel
  perspective.
\newblock {\em Journal of Computational Physics}, 449:110768, 2022.

\bibitem{whittle1954stationary}
Peter Whittle.
\newblock On stationary processes in the plane.
\newblock {\em Biometrika}, pages 434--449, 1954.

\bibitem{whittle1963stochastic}
Peter Whittle.
\newblock Stochastic-processes in several dimensions.
\newblock {\em Bulletin of the International Statistical Institute},
  40(2):974--994, 1963.

\bibitem{woodruff2014sketching}
David~P Woodruff.
\newblock Sketching as a tool for numerical linear algebra.
\newblock {\em Foundations and Trends{\textregistered} in Theoretical Computer
  Science}, 10(1-2):1--157, 2014.

\bibitem{xu2026weak}
Zhihang Xu, Min Wang, and Zhu Wang.
\newblock Weak transnet: A petrov-galerkin based neural network method for
  solving elliptic pdes: Z. xu, m. wang, z. wang.
\newblock {\em Journal of Scientific Computing}, 107(2):60, 2026.

\bibitem{yu2018deep}
Bing Yu et~al.
\newblock The deep ritz method: a deep learning-based numerical algorithm for
  solving variational problems.
\newblock {\em Communications in Mathematics and Statistics}, 6(1):1--12, 2018.

\bibitem{yu2025natural}
Haijun Yu and Shuo Zhang.
\newblock A natural deep ritz method for essential boundary value problems.
\newblock {\em Journal of Computational Physics}, 537:114133, 2025.

\bibitem{zang2020weak}
Yaohua Zang, Gang Bao, Xiaojing Ye, and Haomin Zhou.
\newblock Weak adversarial networks for high-dimensional partial differential
  equations.
\newblock {\em Journal of Computational Physics}, 411:109409, 2020.

\end{thebibliography}
	
	\appendix
	\section{More details about the experiments}
	\label{app_details}
	
	Table \ref{tab_hyper} present the details about the implementation of each experiment.
	
	\begin{table}[H]
		\centering
		\caption{\footnotesize Hyperparameters of the experiments in hypercube domains 1D (Ex. 1), 2D (Ex. 2, 3, 4, 5) and 3D (Ex. 6), circular domain (Ex. 7) and L-shaped domain (Ex. 8).} \label{tab_hyper}
		\resizebox{\linewidth}{!}{\begin{tabular}{|l|c|c|c|c|c|}
			\hline
			& 1-Dimension & 2-Dimensions & 3-Dimensions  & Circle & L-shape \\
			\hline
			\multirow{3}{*}{\textbf{Architecture}} & 64 DAFF + 3 $\times$ 512 & 128 DAFF + 3 $\times$ 512 (Ex. 2-4) & \multirow{3}{*}{36 DAFF + 3 $\times$ 512} & \multirow{3}{*}{64 DAFF + 3 $\times$ 512} & \multirow{3}{*}{64 DAFF + 1 $\times$ 1,024} \\
			& 32 FF ($\sigma = 10$) + 3 $\times$ 512 & 32 FF ($\sigma = 5$) + 3 $\times$ 512 (Ex. 5) & & &  \\
			& None 3 $\times$ 512 & & & & \\
			\hline
			\multirow{2}{*}{\textbf{Collocation grid}} & \multirow{3}{*}{$1,024$} & $128 \times 128$ (Ex. 2-3) & \multirow{3}{*}{$32 \times 32 \times 32$} & \multirow{3}{*}{$22,500$ random points} & \multirow{3}{*}{$7,105$ points}  \\
			& & $150 \times 150$ (Ex. 4-5) & & & \\
			& & $604$ boundary (Ex. 5) & & & \\
			\hline 
			\multirow{2}{*}{\textbf{Test grid}} & \multirow{2}{*}{$2,048$} & $256 \times 256$ (Ex. 2-3) & \multirow{2}{*}{$64 \times 64 \times 64$} & \multirow{2}{*}{$90,000$ random points} & \multirow{2}{*}{$28,033$ points}\\
			& & $300 \times 300$ (Ex. 4-5) & & &  \\
			\hline 
			\textbf{Samples weak norm} & $25,000$ & $25,000$ & $25,000$ & $18,000$ & $25,000$ \\
			\hline 
			\textbf{Steps} & \multicolumn{5}{c|}{$5,000$} \\
			\hline
			\textbf{Activation} & \multicolumn{5}{c|}{tanh} \\
			\hline
			\textbf{ADAM} \cite{baydin2018automatic} & \multicolumn{5}{c|}{Learning rate $0.001$ with exponential decay of $0.9$ every $100$ steps, Glorot initialisation \cite{glorot2010understanding}} \\
			\hline 
			\textbf{\lb} \cite{nocedal1980updating} & \multicolumn{5}{c|}{tol = 1e-9, linesearch = zoom, history\_size = 200 and default settings of \textit{jaxopt} \cite{jaxopt_implicit_diff} Python library} \\
			\hline
			\textbf{Software} &  \multicolumn{5}{c|}{Python 3.12.1 \cite{python}, and libraries \textit{jax} \cite{jax2018github}, \textit{jaxopt} \cite{jaxopt_implicit_diff}, \textit{optax} \cite{deepmind2020jax}} \\
			\hline
			\textbf{Hardware} & \multicolumn{5}{c|}{NVIDIA H200 141GB GPU} \\
			\hline
		\end{tabular}}
	\end{table}
	
	Implementation notes:
	\begin{itemize}
		\item \textbf{Experiments 1-6}: we consider a slightly modified loss function, computing the average over the sampled test functions of
		\begin{align*}
			\sum_{i = 0}^{N - 1} (Lu_{\theta}(i/N) - f(i/N))\varphi_{j}((i + 1)/(N+1))
		\end{align*}
		instead of also evaluating the residuals at the internal points $(i + 1)/(N+1)$. This slight modification led to better results in some experiments since it introduces a small quadrature perturbation that might act as implicit regularisation.
		\item \textbf{Experiment 7}: The Bessel functions are approximated by pre-computing on a dense one-dimensional grid and reconstructing the function values via interpolation during model training and evaluation. Specifically, the functions $J_{n}(r)$ are computed by smooth interpolation (cubic spline for computing test functions and a smoothed cubic radial basis for computing DAFF) of pre-computed exact Bessel evaluations at fixed $r_{1},\dots,r_{n}$ points ($n = 4,096$). As a result, evaluations during training use only the interpolated values rather than direct Bessel function calls, providing a smooth, differentiable, and computationally efficient approximation suitable for repeated use in the Helmholtz solver.
		\item \textbf{Experiment 8}: In the L-shaped domain, the Dirichlet Laplacian eigenproblem is discretised using a finite-element method with continuous piecewise-cubic polynomials (P3) bases on a conforming triangular mesh. DAFF computation, and model training and evaluation, are performed by interpolating the value of the eigenfunctions in the mesh with smoothed cubic radial basis function.
	\end{itemize}
	
	{\footnotesize
		\setlength{\tabcolsep}{4pt}
		\setlength{\LTleft}{-3cm}
		\setlength{\LTright}{\fill}
	\begin{longtable}{llllllllc}
		\caption{$L^2$ relative error and training time of solving the boundary value problem in Experiment 1 \eqref{ex1} with PINNs and SV-PINNs with different features and optimisers. The values are the average over $3$ repetitions after $1,000$ and $5,000$ training steps with the respective standard deviation in parentheses. The number of steps to achieve $L^2$-relative error $< 0.01$ is also presented for the cases in which all repetitions achieved this threshold in $5,000$ steps.} \label{tab_ex1}\\
			\hline
			\multirow{2}{*}{a} & \multirow{2}{*}{Method} & \multirow{2}{*}{Features} & \multirow{2}{*}{Optimiser} & \multicolumn{2}{c}{$1,000$ steps} & \multicolumn{2}{c}{$5,000$ steps} & Steps to\\
			\cline{5-8}& & & & $L^2$ relative error & Time (m) &  $L^2$ relative error & Time (m) & $L^2$-RE $< 0.01$ \\
			    \hline
			  1 & PINN & DAFF & GD & 9.724e-01 (4.685e-02) & 0.14 (0.02) & 9.393e-01 (4.212e-02) & 0.23 (0.02) &  \\ 
			  &  & FF & GD & 7.919e-01 (7.120e-02) & 0.18 (0.02) & 6.805e-01 (8.503e-02) & 0.27 (0.02) &  \\ 
			  &  & None & GD & 4.902e-03 (2.273e-03) & 0.15 (0.01) & 6.858e-04 (3.355e-04) & 0.22 (0.01) & 715 (47.9) \\ 
			  \cline{2-9} & SV-PINN & DAFF & GD & 2.538e-04 (9.091e-05) & 0.14 (0) & 9.906e-05 (2.193e-05) & 0.23 (0) & 385.7 (145.1) \\ 
			  &  &  & L-BFGS & 1.702e-04 (6.288e-05) & 0.52 (0.01) & 9.239e-06 (1.962e-07) & 1.34 (0.01) & 99.3 (4) \\ 
			  \cline{3-9} &  & FF & GD & 3.428e-04 (3.386e-05) & 0.16 (0) & 1.249e-04 (4.931e-05) & 0.26 (0) & 462 (150.3) \\ 
			  &  &  & L-BFGS & 4.641e-05 (1.305e-05) & 0.72 (0.01) & 6.289e-06 (3.471e-06) & 2.17 (0.02) & 51.7 (7.8) \\ 
			  \cline{3-9} &  & None & GD & 6.738e-04 (1.604e-04) & 0.16 (0.01) & 3.903e-04 (7.576e-05) & 0.24 (0.01) & 170.7 (16.8) \\ 
			  &  &  & L-BFGS & 1.442e-05 (1.132e-05) & 0.67 (0.01) & 5.838e-07 (4.267e-07) & 1.93 (0.01) & 104.3 (17.8) \\ 
			  \hline 25 & PINN & DAFF & GD & 9.698e-01 (3.620e-02) & 0.13 (0) & 9.365e-01 (3.164e-02) & 0.22 (0) &  \\ 
			  &  & FF & GD & 7.859e-01 (6.160e-02) & 0.15 (0) & 6.704e-01 (7.186e-02) & 0.24 (0) &  \\ 
			  &  & None & GD & 2.107e+00 (7.852e-01) & 0.15 (0) & 3.494e+00 (8.501e-01) & 0.22 (0) &  \\ 
			  \cline{2-9} & SV-PINN & DAFF & GD & 2.010e-04 (7.370e-05) & 0.14 (0) & 1.005e-04 (1.884e-05) & 0.24 (0) & 344.7 (109.5) \\ 
			  &  &  & L-BFGS & 2.019e-04 (1.067e-04) & 0.69 (0) & 1.270e-05 (4.081e-06) & 2.15 (0) & 87.7 (20.8) \\ 
			  \cline{3-9} &  & FF & GD & 3.651e-04 (2.411e-05) & 0.16 (0) & 1.259e-04 (3.952e-05) & 0.26 (0) & 467 (156.9) \\ 
			  &  &  & L-BFGS & 3.337e-05 (1.181e-05) & 0.73 (0.01) & 5.366e-06 (1.990e-06) & 2.2 (0.01) & 51.3 (8.1) \\ 
			  \cline{3-9} &  & None & GD & 2.448e-02 (3.528e-02) & 0.15 (0) & 9.238e-04 (5.241e-04) & 0.24 (0) & 1068.7 (208.8) \\ 
			  &  &  & L-BFGS & 1.829e-03 (1.100e-03) & 0.66 (0) & 1.670e-04 (1.785e-04) & 1.91 (0) & 538.7 (94.9) \\ 
			  \hline 50 & PINN & DAFF & GD & 9.701e-01 (3.470e-02) & 0.13 (0) & 9.366e-01 (3.001e-02) & 0.22 (0) &  \\ 
			  &  & FF & GD & 7.740e-01 (6.091e-02) & 0.15 (0) & 6.500e-01 (7.318e-02) & 0.24 (0) &  \\ 
			  &  & None & GD & 1.253e+01 (3.864e+00) & 0.15 (0) & 6.066e+00 (8.865e-01) & 0.22 (0) &  \\ 
			  \cline{2-9} & SV-PINN & DAFF & GD & 2.059e-04 (8.463e-05) & 0.14 (0) & 9.886e-05 (2.234e-05) & 0.24 (0) & 350.7 (118.8) \\ 
			  &  &  & L-BFGS & 1.823e-04 (8.511e-05) & 0.69 (0) & 2.340e-05 (1.680e-05) & 2.15 (0.01) & 89.7 (21.2) \\ 
			  \cline{3-9} &  & FF & GD & 4.446e-04 (6.932e-05) & 0.16 (0) & 1.388e-04 (6.024e-05) & 0.26 (0) & 490.3 (144.9) \\ 
			  &  &  & L-BFGS & 3.821e-05 (8.171e-06) & 0.73 (0.01) & 6.926e-06 (5.917e-06) & 2.2 (0.01) & 53.3 (9.1) \\ 
			  \cline{3-9} &  & None & GD & 1.068e-01 (4.118e-03) & 0.15 (0) & 6.573e-02 (1.596e-02) & 0.24 (0) &  \\ 
			  &  &  & L-BFGS & 7.347e-02 (4.900e-02) & 0.66 (0.01) & 3.435e-03 (2.857e-03) & 1.89 (0.02) & 3134.3 (1900.5) \\ 
			  \hline 100 & PINN & DAFF & GD & 9.745e-01 (4.049e-02) & 0.13 (0) & 9.419e-01 (3.475e-02) & 0.22 (0) &  \\ 
			  &  & FF & GD & 7.590e-01 (3.853e-02) & 0.15 (0) & 6.590e-01 (3.641e-02) & 0.24 (0) &  \\ 
			  &  & None & GD & 1.568e+01 (7.755e+00) & 0.14 (0) & 1.238e+01 (5.482e+00) & 0.22 (0) &  \\ 
			  \cline{2-9} & SV-PINN & DAFF & GD & 2.221e-04 (8.165e-05) & 0.14 (0) & 1.201e-04 (2.723e-05) & 0.24 (0) & 362 (110) \\ 
			  &  &  & L-BFGS & 1.825e-04 (9.580e-05) & 0.69 (0) & 1.614e-05 (3.915e-06) & 2.13 (0.04) & 87 (23.6) \\ 
			  \cline{3-9} &  & FF & GD & 1.402e-03 (5.866e-04) & 0.16 (0) & 4.135e-04 (1.454e-04) & 0.26 (0) & 582.7 (179.7) \\ 
			  &  &  & L-BFGS & 4.680e-05 (2.240e-05) & 0.72 (0) & 6.653e-06 (2.189e-06) & 2.16 (0.03) & 72 (24.2) \\ 
			  \cline{3-9} &  & None & GD & 1.172e-01 (6.582e-03) & 0.15 (0) & 1.011e-01 (1.357e-02) & 0.23 (0) &  \\ 
			  &  &  & L-BFGS & 1.109e-01 (4.418e-03) & 0.72 (0.1) & 9.449e-02 (2.186e-02) & 1.97 (0.1) &  \\ 
			  \hline 150 & PINN & DAFF & GD & 9.715e-01 (3.823e-02) & 0.13 (0) & 9.401e-01 (3.334e-02) & 0.22 (0) &  \\ 
			  &  & FF & GD & 8.154e-01 (5.936e-02) & 0.15 (0) & 7.291e-01 (1.084e-02) & 0.24 (0) &  \\ 
			  &  & None & GD & 3.400e+01 (1.963e+01) & 0.15 (0) & 6.110e+01 (4.259e+01) & 0.22 (0) &  \\ 
			  \cline{2-9} & SV-PINN & DAFF & GD & 4.251e-04 (1.773e-04) & 0.14 (0) & 1.439e-04 (3.709e-05) & 0.24 (0) & 348.3 (119.7) \\ 
			  &  &  & L-BFGS & 1.907e-04 (6.547e-05) & 0.69 (0) & 2.440e-05 (1.375e-05) & 2.15 (0) & 88 (21.9) \\ 
			  \cline{3-9} &  & FF & GD & 1.972e-03 (9.858e-04) & 0.16 (0) & 5.891e-04 (8.430e-05) & 0.26 (0) & 571.7 (173.3) \\ 
			  &  &  & L-BFGS & 2.031e-04 (1.665e-04) & 0.72 (0) & 2.404e-05 (1.278e-05) & 2.17 (0.04) & 127 (41) \\ 
			  \cline{3-9} &  & None & GD & 1.469e-01 (2.527e-02) & 0.16 (0.01) & 1.385e-01 (3.027e-02) & 0.24 (0.01) &  \\ 
			  &  &  & L-BFGS & 1.469e-01 (1.449e-02) & 0.67 (0.01) & 1.713e-01 (8.501e-02) & 1.93 (0.01) &  \\ 
			  \hline 200 & PINN & DAFF & GD & 9.845e-01 (5.461e-02) & 0.13 (0) & 9.523e-01 (5.037e-02) & 0.22 (0) &  \\ 
			  &  & FF & GD & 9.771e-01 (2.086e-01) & 0.15 (0) & 7.555e-01 (9.036e-02) & 0.25 (0) &  \\ 
			  &  & None & GD & 2.695e+01 (1.250e+01) & 0.14 (0) & 4.205e+01 (1.377e+01) & 0.22 (0) &  \\ 
			  \cline{2-9} & SV-PINN & DAFF & GD & 6.929e-04 (1.761e-04) & 0.14 (0) & 3.959e-04 (1.762e-04) & 0.24 (0) & 368 (120.2) \\ 
			  &  &  & L-BFGS & 1.805e-04 (1.175e-04) & 0.69 (0) & 2.081e-05 (9.764e-06) & 2.13 (0.03) & 87.7 (28.4) \\ 
			  \cline{3-9} &  & FF & GD & 3.234e-03 (2.236e-03) & 0.16 (0) & 7.740e-04 (1.615e-04) & 0.26 (0) & 562 (171.8) \\ 
			  &  &  & L-BFGS & 6.399e-04 (2.534e-04) & 0.72 (0.01) & 1.152e-04 (1.921e-05) & 2.19 (0.01) & 263.3 (63.5) \\ 
			  \cline{3-9} &  & None & GD & 3.012e-01 (9.973e-02) & 0.15 (0) & 2.779e-01 (7.777e-02) & 0.23 (0) &  \\ 
			  &  &  & L-BFGS & 2.818e-01 (1.023e-01) & 0.67 (0.01) & 2.161e-01 (6.686e-02) & 1.92 (0.01) &  \\ 
			  \hline 250 & PINN & DAFF & GD & 1.001e+00 (5.541e-02) & 0.15 (0.03) & 9.657e-01 (5.548e-02) & 0.23 (0.03) &  \\ 
			  &  & FF & GD & 1.092e+00 (3.177e-02) & 0.16 (0.01) & 8.858e-01 (7.490e-02) & 0.26 (0.01) &  \\ 
			  &  & None & GD & 2.870e+01 (2.084e+01) & 0.15 (0.01) & 4.532e+01 (3.190e+01) & 0.22 (0.01) &  \\ 
			  \cline{2-9} & SV-PINN & DAFF & GD & 1.156e-03 (4.949e-04) & 0.14 (0) & 4.470e-04 (1.759e-04) & 0.24 (0.01) & 333.7 (146.1) \\ 
			  &  &  & L-BFGS & 2.966e-04 (6.614e-05) & 0.69 (0) & 2.910e-05 (1.021e-05) & 2.15 (0) & 108.3 (19.1) \\ 
			  \cline{3-9} &  & FF & GD & 7.287e-03 (2.668e-03) & 0.16 (0) & 1.015e-02 (1.097e-02) & 0.27 (0) & 716.7 (277.1) \\ 
			  &  &  & L-BFGS & 8.498e-01 (1.154e+00) & 0.72 (0.01) & 6.872e-01 (9.537e-01) & 2.19 (0.01) &  \\ 
			  \cline{3-9} &  & None & GD & 4.799e-01 (4.452e-01) & 0.15 (0.01) & 2.093e-01 (4.935e-02) & 0.24 (0) &  \\ 
			  &  &  & L-BFGS & 2.076e-01 (4.251e-02) & 0.67 (0.01) & 2.014e-01 (1.999e-02) & 1.93 (0) &  \\ 
			  \hline 300 & PINN & DAFF & GD & 9.908e-01 (5.692e-02) & 0.13 (0) & 9.606e-01 (5.523e-02) & 0.23 (0.01) &  \\ 
			  &  & FF & GD & 1.566e+00 (4.275e-01) & 0.15 (0) & 1.336e+00 (3.038e-01) & 0.24 (0) &  \\ 
			  &  & None & GD & 1.615e+01 (1.834e+01) & 0.15 (0) & 4.530e+01 (3.825e+01) & 0.22 (0.01) &  \\ 
			  \cline{2-9} & SV-PINN & DAFF & GD & 9.426e-04 (4.674e-04) & 0.14 (0) & 3.050e-04 (1.666e-04) & 0.24 (0) & 394 (68.5) \\ 
			  &  &  & L-BFGS & 8.127e-04 (4.405e-04) & 0.69 (0.01) & 5.318e-05 (3.462e-05) & 2.16 (0.01) & 160 (65.2) \\ 
			  \cline{3-9} &  & FF & GD & 1.850e-01 (6.646e-02) & 0.16 (0) & 1.624e-01 (6.249e-02) & 0.26 (0) &  \\ 
			  &  &  & L-BFGS & 1.858e+00 (4.762e-01) & 0.73 (0.02) & 2.291e+00 (2.155e+00) & 2.18 (0.03) &  \\ 
			  \cline{3-9} &  & None & GD & 5.883e-01 (3.995e-01) & 0.15 (0.01) & 2.078e-01 (6.886e-02) & 0.24 (0.01) &  \\ 
			  &  &  & L-BFGS & 1.953e-01 (6.249e-02) & 0.66 (0) & 2.300e-01 (8.392e-02) & 1.93 (0) &  \\
			  \hline
	\end{longtable}}
	
	\section{Auxiliary results}
	\label{appAux}
	
	In this section, we present some results that are necessary for the proofs of this paper. From now on, we assume that $\Omega = (0,1)^{d}, d \leq 3$. First, we recall the following version of the Bramble-Hilbert Lemma \cite{bramble1970estimation}. Denote the squared $H^2(\Omega)$ semi-norm by $\lvert u \rvert_{H^2}^2 = \sum_{|\alpha| = 2} \lVert D^{\alpha}u \rVert_{L^2}^2$ and by $(H^{2}(\Omega))^{\star}$ the dual space of $H^{2}(\Omega)$.
	
	\begin{lemma}[Bramble-Hilbert Lemma]
		\label{lemma_BH}
		Let $\Omega = (0,1)^{d}, d \leq 3,$ and $F \in (H^{2}(\Omega))^{\star}$ be such that $F(u_{\ell}) = 0$ for all affine functions $u_{\ell}(x) = a_{0} + \sum_{j = 1}^{d} a_{j} x_{j}$. Then, there exists a constant $C = C(\Omega)$ such that
		\begin{align*}
			|F(u)| \leq C \, \lVert F \rVert \, \lvert u \rvert_{H^2}
		\end{align*}  
		for all $u \in H^2(\Omega)$ in which $\lVert F \rVert$ is the $(H^{2}(\Omega))^{\star}$ norm of $F$.
	\end{lemma}
	
	We will apply Lemma \ref{lemma_BH} to bound the error of the Trapezoid rule for approximating integrals of $u \in H^2(\Omega) \cap H_{0}^{1}(\Omega)$. Formally, denote
	\begin{align*}
		I(u) = \int_{\Omega} u(x) \ dx & & S_{h} = h^{d} \sum_{|\boldsymbol{k}| \leq n} u(x^{(\bs{k})})
	\end{align*}
	in which $h = 1/(n + 1)$ and $x^{(\bs{k})} = (k_{1}/(n+1),\dots,k_{d}/(n+1))$. Since $u \in H^2(\Omega)$ and $d \le 3$, the (Sobolev) embedding $H^2(\Omega) \hookrightarrow C^0(\overline{\Omega})$ holds, and the condition $u \in H_0^1(\Omega)$ implies $u(x) = 0$ for $x \in \partial\Omega$. Therefore, it holds
	\begin{align*}
		S_{h}(u) = h^{d} \sum_{|\boldsymbol{k}| \leq (n + 1)} q_{\boldsymbol{k}} \, u(x^{(\bs{k})})
	\end{align*}
	with $q_{\bs{k}} = 1/2^{\sum_j \mathds{1}\{k_{j} \in \{0,n + 1\}\}}$, so it equals the approximation of $I(u)$ obtained by applying the composite Trapezoid rule in each coordinate. In particular, $I(\cdot) - S_h(\cdot) \in (H^{2}(\Omega))^{\star}$ and $I(u_{\ell}) - S_h(u_{\ell}) = 0$ if $u_{\ell}$ is an affine function, so Lemma \ref{lemma_BH} can be applied to yield an approximation error between $S_h(u)$ and $I(u)$.
	
	\begin{proposition}
		\label{prop_trapezoid}
		If $u \in H^2(\Omega) \cap H_{0}^{1}(\Omega)$ and $d \leq 3$, then there exists $C = C(\Omega) > 0$ such that
		\begin{align*}
			\left|\int_{\Omega} u(x) \ dx - S_{h}(u)\right| \leq C \, h^2 \, \lVert u \rVert_{H^2}.
		\end{align*}
	\end{proposition}
	\begin{proof}
		To each $|\bs{k}| \leq n$ associate the cube $Q_\bs{k} = \{x_{\bs{k}} + y: y \in (0,h)^{d}\}$ and observe that $\Omega = \bigcup_{\bs{k} \leq n} Q_{\bs{k}}$, in which the union is of disjoint sets, so
		\begin{align}
			\label{int_as_sum}
			\left|I(u) - S_{h}(u)\right| \leq \sum_{|\boldsymbol{k}| \leq n} \left|I(u|_{Q_{\bs{k}}}) - S_{h}(u|_{Q_{\bs{k}}})\right|.
		\end{align}
		Now, by a change of variables,  denoting $u_{\bs{k}}(y) = u(x_{\bs{k}} + hy)$,
		\begin{align*}
			I(u|_{Q_{\bs{k}}}) = \int_{Q_{\bs{k}}} u(x) \ dx = \int_{(0,h)^d} u(x_{\bs{k}} + y) \ dy = h^d \, \int_{\Omega} u(x_{\bs{k}} + hy) \ dy = h^d I(u_{\bs{k}}) 
		\end{align*}
		and, likewise, $S_{h}(u|_{Q_{\bs{k}}}) = h^d \, S_{h}(u_{\bs{k}})$, so 
		\begin{align}
			\label{B2}
			\left|I(u|_{Q_{\bs{k}}}) - S_{h}(u|_{Q_{\bs{k}}})\right| \leq h^{d} \left|I(u_{\bs{k}}) - S_{h}(u_{\bs{k}})\right| \leq C \, h^d \, \lvert u_{\bs{k}} \rvert_{H^2}
		\end{align}
		in which the inequality is due to Lemma \ref{lemma_BH}. We note that the sup over $h > 0$ of the $(H^{2}(\Omega))^{\star}$ norm of the functional $I(\cdot) - S_{h}(\cdot)$, which is finite for $d \leq 3$, is part of the constant $C$.
		
		By the chain rule, for $|\alpha| = 2$ and $y \in \Omega$,
		\begin{align*}
			(D^{\alpha}u_{\bs{k}})(y) = h^2 \, (D^{\alpha}u)(x_{\boldsymbol{k}} + hy)
		\end{align*}
		and therefore
		\begin{align}
			\label{B3} \nonumber
			\lVert D^{\alpha}u_{\bs{k}} \rVert_{L^2}^{2} &= \int_{\Omega} \left|h^2 (D^{\alpha}u)(x_{\boldsymbol{k}} + hy)\right|^{2} \, dy\\
			&= h^{4 - d} \int_{Q_{\bs{k}}} \left|(D^{\alpha}u)(x)\right|^{2} \, dx = h^{4 - d} \lVert D^{\alpha}u \rVert_{L^2(Q_{\bs{k}})}^{2}.
		\end{align}
		Combining \eqref{int_as_sum}, \eqref{B2} and \eqref{B3} we conclude that
		\begin{align*}
			\left|I(u) - S_{h}(u)\right| &\leq C \, h^{2 + d/2} \sum_{|\boldsymbol{k}| \leq n} \lvert u \rvert_{H^2(Q_{\bs{k}})}\\
			&\leq C \, h^{2 + d/2} \, h^{-d/2} \left(\sum_{|\boldsymbol{k}| \leq n} \lvert u \rvert_{H^2(Q_{\bs{k}})}^2\right)^{1/2} = C \, h^2 \, \lvert u \rvert_{H^2}
		\end{align*}
		in which the second inequality is due to Cauchy-Schwarz. It remains to recall that $\lvert u \rvert_{H^2} \leq \lVert u \rVert_{H^{2}}$.
	\end{proof}
	
	As a corollary, we can bound $|\langle R,\phi_{\bs{k}} \rangle_{h} - \langle R,\phi_{\bs{k}} \rangle|$ for $R \in H^{2}(\Omega)$ in which $\langle \cdot,\cdot\rangle_{h}$ is defined in \eqref{discrete_IN}. Recall that $h = 1/(n + 1)$.
	
	\begin{corollary}
		\label{corollary_norms}
		If $R \in H^{2}(\Omega)$ and $|\bs{k}| \leq n$, then 
		\begin{align*}
			|\langle R,\phi_{\bs{k}} \rangle_{h} - \langle R,\phi_{\bs{k}} \rangle| \leq C \, h^{2} \, \lambda_{\bs{k}} \, \lVert R \rVert_{H^2}.
		\end{align*}
	\end{corollary}
	\begin{proof}
		Since $u \coloneqq R\phi_{\bs{k}} \in H^2(\Omega) \cap H_{0}^{1}(\Omega)$, the discrete norm $\langle R,\phi_{\bs{k}} \rangle_{h}$ is the trapezoid rule of the integral $\langle R,\phi_{\bs{k}} \rangle$ so we can apply Proposition \ref{prop_trapezoid} to obtain
		\begin{align*}
			|\langle R,\phi_{\bs{k}} \rangle_{h} - \langle R,\phi_{\bs{k}} \rangle| \leq C \, h^2 \, \lVert R\phi_{\bs{k}} \rVert_{H^2} \leq  C \, h^2 \, \lVert R \rVert_{H^2} \, \lVert \phi_{\bs{k}} \rVert_{H^2}
		\end{align*}
		in which the second inequality follows since $H^{2}(\Omega)$ is closed under multiplication for $d \leq 3$. Since $\phi_{\bs{k}} \in H^{2}(\Omega) \cap H_{0}^{1}(\Omega)$ we have that
		\begin{align*}
			\lVert \phi_{\bs{k}} \rVert_{H^2}^{2} \leq C \, \lVert \phi_{\bs{k}} \rVert_{\dot{H}^2}^{2} = C \sum_{\bs{k}^\prime} \lambda_{\bs{k}^\prime}^{2} \, |\langle \phi_{\bs{k}},\phi_{\bs{k}^\prime} \rangle|^{2} = C \, \lambda_{\bs{k}}^{2}
		\end{align*}
		and the result follows.
	\end{proof}
	
	We now prove the rate of convergence of the eigenvalues of the discretised Dirichlet Laplacian for hypercube domains.
	
	\begin{lemma}
		\label{lemma_eigenvalue}
		Let $\Omega = (0,1)^{d}$, $\{\lambda_{\bs{k}}\}$ denote the eigenvalues of Dirichlet Laplacian and $\{\lambda_{\bs{k}}^{(h)}\}$ be the eigenvalues of the respective discretised operator with mesh size $h$. Then
		\begin{align*}
			|\lambda_{\bs{k}}^{(h)} - \lambda_{\bs{k}}| \leq C \, h^{2} \, \lambda_{\bs{k}}^{2}
		\end{align*}
		for all $\bs{k}$.
	\end{lemma}	
	\begin{proof}
		Recall that
		\begin{align*}
			\lambda_{\bs{k}} = \pi^{2} \sum_{j=1}^{d} k_j^{2} & & \lambda_{\bs{k}}^{(h)} = \frac{4}{h^{2}} \sum_{j=1}^{d} \sin^{2}\!\left( \frac{k_j \pi h}{2} \right).
		\end{align*}
		with $k_{j} = 1,\dots,n$ and $h = 1/(n + 1)$. Fix $k \in \{1,\dots,n\}$ and let
		\begin{align*}
			\lambda_k = \pi^{2} k^{2} & & \lambda_k^{(h)} = \frac{4}{h^{2}} \sin^{2}\!\left( \frac{k \pi h}{2} \right).
		\end{align*}
		By the Taylor expansion of the sine function
		\begin{align*}
			\sin(x) = x - \frac{x^{3}}{6} + \mathcal{O}(x^{5}) & & \text{ and } & & \sin^{2}(x) = x^{2} - \frac{x^{4}}{3} + \mathcal{O}(x^{6}).
		\end{align*}
		Taking $x = \frac{k \pi h}{2}$ above we conclude that
		\begin{align*}
			\lambda_k^{(h)}	= \pi^{2} k^2 -  \frac{\pi^{4} k^{4} h^{2}}{12}	+ \mathcal{O}\!\left(k^{6} \pi^6 h^4\right) = \lambda_k - \frac{h^{2}}{12} \lambda_k^{2}	+ \mathcal{O}(h^{4}\lambda_k^{3})
		\end{align*}
		hence $\lambda_k^{(h)} - \lambda_k = \mathcal{O}(h^{2}\lambda_k^{2})$. Since
		\begin{align*}
			\sum_{j=1}^{d} \lambda_{k_j}^{2} \leq \left(\sum_{j=1}^{d} \lambda_{k_j} \right)^{2} = \lambda_{\bs{k}}^{2}
		\end{align*}
		we conclude
		\begin{align*}
			|\lambda_{\bs{k}}^{(h)} - \lambda_{\bs{k}}| = \left|\sum_{j = 1}^{d} \left(\lambda_{k_{j}}^{(h)} - \lambda_{k_{j}}\right)\right| = \mathcal{O}(h^{2}) \sum_{j=1}^{d} \lambda_{k_{j}}^{2} \leq C \, h^{2} \, \lambda_{\bs{k}}^{2}.
		\end{align*}
	\end{proof}
	 
\end{document}